\documentclass[11pt,a4paper]{amsart}
\usepackage{amssymb,amsmath,amsfonts,amscd}
\usepackage[dvips,final]{graphics}
\usepackage{amscd}
\usepackage{mathrsfs}
\usepackage[dvips]{color}

\theoremstyle{plain}
\newtheorem{theorem}{Theorem}[section]
\newtheorem{proposition}[theorem]{Proposition}
\newtheorem{corollary}[theorem]{Corollary}

\newtheorem{lemma}[theorem]{Lemma}
\theoremstyle{definition}
\newtheorem{definition}[theorem]{Definition}

\theoremstyle{remark}
\newtheorem{remark}[theorem]{Remark}

\numberwithin{equation}{section}

\usepackage{enumerate}

\DeclareMathOperator{\trace}{Tr}
\DeclareMathOperator{\divergence}{div}
\DeclareMathOperator{\Ker}{ker}
\DeclareMathOperator{\Lip}{Lip}
\newcommand{\diam}{\mathrm{diam}\,}

\begin{document}

\bibliographystyle{plain}

\def\average{{\mathchoice {\kern1ex\vcenter{\hrule height.4pt width 6pt
depth0pt} \kern-9.7pt} {\kern1ex\vcenter{\hrule height.4pt width 4.3pt depth0pt}
\kern-7pt} {} {} }}
\def\ave{\average\int}

\newcommand{\Rn}{\mathbb R^n}
\newcommand{\E}{\mathbb E}
\newcommand{\Rm}{\mathbb R^m}
\newcommand{\rn}[1]{{\mathbb R}^{#1}}
\newcommand{\R}{\mathbb R}
\newcommand{\C}{\mathbb C}
\newcommand{\G}{\mathbb G}
\newcommand{\M}{\mathbb M}
\newcommand{\Z}{\mathbb Z}
\newcommand{\D}[1]{\mathcal D^{#1}}
\newcommand{\Ci}[1]{\mathcal C^{#1}}
\newcommand{\Ch}[1]{\mathcal C_{\mathbb H}^{#1}}
\renewcommand{\L}[1]{\mathcal L^{#1}}
\newcommand{\BVG}{BV_{\G}({\mc U})}
\newcommand{\supp}{\mathrm{supp}\;}

\newcommand{\dom}{\mathrm{Dom}\;}
\newcommand{\Ast}{\un{\ast}}


\newcommand{\hhd}[1]{{\mathcal H}_d^{#1}}
\newcommand{\hsd}[1]{{\mathcal S}_d^{#1}}

\newcommand{\he}[1]{{\mathbb H}^{#1}}
\newcommand{\hhe}[1]{{H\mathbb H}^{#1}}

\newcommand{\cov}[1]{{\bigwedge\nolimits^{#1}{\mfrak h}}}
\newcommand{\vet}[1]{{\bigwedge\nolimits_{#1}{\mfrak h}}}

\newcommand{\covw}[2]{{\bigwedge\nolimits^{#1,#2}{\mfrak h}}}

\newcommand{\vetfiber}[2]{{\bigwedge\nolimits_{#1,#2}{\mfrak h}}}
\newcommand{\covfiber}[2]{{\bigwedge\nolimits^{#1}_{#2}{\mfrak g}}}

\newcommand{\covwfiber}[3]{{\bigwedge\nolimits^{#1,#2}_{#3}{\mfrak g}}}

\newcommand{\covv}[2]{{\bigwedge\nolimits^{#1}{#2}}}
\newcommand{\vett}[2]{{\bigwedge\nolimits_{#1}{#2}}}

\newcommand{\covvfiber}[3]{{\bigwedge\nolimits^{#1}_{#2}{#3}}}

\newcommand{\vettfiber}[3]{{\bigwedge\nolimits_{#1,#2}{#3}}}

\newcommand{\covn}[2]{{\bigwedge\nolimits^{#1}\rn {#2}}}
\newcommand{\vetn}[2]{{\bigwedge\nolimits_{#1}\rn {#2}}}
\newcommand{\covh}[1]{{\bigwedge\nolimits^{#1}{\mfrak h_1}}}
\newcommand{\veth}[1]{{\bigwedge\nolimits_{#1}{\mfrak h_1}}}
\newcommand{\hcov}[1]{{_H\!\!\bigwedge\nolimits^{#1}}}
\newcommand{\hvet}[1]{{_H\!\!\bigwedge\nolimits_{#1}}}

\newcommand{\covf}[2]{{\bigwedge\nolimits^{#1}_{#2}{\mfrak h}}}
\newcommand{\vetf}[2]{{\bigwedge\nolimits_{#1,#2}{\mfrak h}}}
\newcommand{\covhf}[2]{{\bigwedge\nolimits^{#1}_{#2}{\mfrak h_1}}}
\newcommand{\vethf}[2]{{\bigwedge\nolimits_{#1,#2}{\mfrak h_1}}}
\newcommand{\hcovf}[2]{{_H\!\!\bigwedge\nolimits^{#1}_{#2}}}
\newcommand{\hvetf}[2]{{_H\!\!\bigwedge\nolimits_{#1,#2}}}

\newcommand{\defin}{\stackrel{\mathrm{def}}{=}}
\newcommand{\gradh}{\nabla_H}
\newcommand{\current}[1]{\left[\!\left[{#1}\right]\!\right]}
\newcommand{\scal}[2]{\langle {#1} , {#2}\rangle}
\newcommand{\escal}[2]{\langle {#1} , {#2}\rangle_{\mathrm{Euc}}}
\newcommand{\Scal}[2]{\langle {#1} \vert {#2}\rangle}
\newcommand{\scalp}[3]{\langle {#1} , {#2}\rangle_{#3}}
\newcommand{\dc}[2]{d_c\left( {#1} , {#2}\right)}
\newcommand{\res}{\mathop{\hbox{\vrule height 7pt width .5pt depth 0pt
\vrule height .5pt width 6pt depth 0pt}}\nolimits}
\newcommand{\norm}[1]{\left\Vert{#1}\right\Vert}
\newcommand{\modul}[2]{{\left\vert{#1}\right\vert}_{#2}}
\newcommand{\perh}{\partial_\mathbb H}

\newcommand{\ccheck}{{\vphantom i}^{\mathrm v}\!\,}

\newcommand{\mc}{\mathcal }
\newcommand{\mbf}{\mathbf}
\newcommand{\mfrak}{\mathfrak}
\newcommand{\mrm}{\mathrm}
\newcommand{\no}{\noindent}
\newcommand{\dis}{\displaystyle}

\newcommand{\U}{\mathcal U}
\newcommand{\ga}{\alpha}
\newcommand{\gb}{\beta}
\newcommand{\gga}{\gamma}
\newcommand{\gd}{\delta}
\newcommand{\eps}{\varepsilon}
\newcommand{\gf}{\varphi}
\newcommand{\GF}{\varphi}
\newcommand{\gl}{\lambda}
\newcommand{\GL}{\Lambda}
\newcommand{\gp}{\psi}
\newcommand{\GP}{\Psi}
\newcommand{\gr}{\varrho}
\newcommand{\go}{\omega}
\newcommand{\gs}{\sigma}
\newcommand{\gt}{\theta}
\newcommand{\gx}{\xi}
\newcommand{\GO}{{\mc U}}

\newcommand{\Wedge}{\buildrel\circ\over \wedge}

\newcommand{\WO}[4]{\mathop{W}\limits^\circ{}\!_{#4}^{{#1},{#2}}
(#3)}

\newcommand{\GH}{H\G}
\newcommand{\N}{\mathbb N}

%

\newcommand{\Nhmin}{N_h^{\mathrm{min}}}
\newcommand{\Nhmax}{N_h^{\mathrm{max}}}

\newcommand{\Mhmin}{M_h^{\mathrm{min}}}
\newcommand{\Mhmax}{M_h^{\mathrm{max}}}

\newcommand{\un}[1]{\underline{#1}}

\newcommand{\curl}{\mathrm{curl}\;}
\newcommand{\curlh}{\mathrm{curl}_{\he{}}\;}
\newcommand{\hd}{\hat{d_c}}
\newcommand{\divg}{\mathrm{div}_\G\,}
\newcommand{\divh}{\mathrm{div}_{\he{}}\,}
\newcommand{\e}{\mathrm{Euc}}


\newcommand*\xbar[1]{%
  \hbox{%
    \vbox{%
      \hrule height 0.5pt 
      \kern0.5ex
      \hbox{%
        \kern-0.1em
        \ensuremath{#1}%
        \kern-0.1em
      }%
    }%
  }%
}

\newcommand{\x}{\mathrm x }
\author{Eleonora Cinti, Bruno Franchi, Mar\'ia del Mar Gonz\'alez}

\address{E.C., Dipartimento di Matematica ``G. Peano'', Universit\`a di Torino, Via Carlo Alberto 10 -- 10123 Torino (Italy).}
\email{ecinti@unito.it}
\address{B.F. Dipartimento di Matematica, Universit\`a  degli Studi di Bologna, Piazza di Porta San Donato 5-- 40126 Bologna (Italy).}
\email{ bruno.franchi@unibo.it}
\address{M.d.M. G. Universidad Autónoma de Madrid, Departamento de Matemáticas, 28049 Madrid (Spain). }
\vspace{1em}
\email{mariamar.gonzalezn@uam.es}

\thanks{E. Cinti is supported by MINECO grants MTM2011-27739-C04-01, MTM2014-52402-C3-1-P, the ERC starting grant EPSILON n° 277749, and the ERC Advanced Grant 2013 n. 339958 Complex Patterns for Strongly Interacting Dynamical Systems - COMPAT. B. Franchi is supported
by University of Bologna, Italy, funds for selected research topics,
by GNAMPA of INdAM,
and by MAnET Marie Curie Initial Training Network. M.d.M. Gonz\'alez is supported by MINECO grants MTM2011-27739-C04-01, MTM2014-52402-C3-1-P, and BBVA Foundation grant for investigadores y creadores culturales 2016.}

\title{$\Gamma$-convergence of variational functionals with boundary terms in Stein manifolds}
\date{}

\begin{abstract}
Let $\Omega$ be an open subset of a Stein manifold $\Sigma$ and let $M$ be its boundary. It is well known that $M$ inherits a natural contact structure. In this paper we consider a family of variational functionals $F_\eps$ defined by the sum of two terms: a Dirichlet-type energy associated with a sub-Riemannian structure in $\Omega$ and a potential term on the boundary $M$. We prove that the functionals $F_\eps$ $\Gamma$-converge to the intrinsic perimeter in $M$ associated with its contact structure.

Similar results have been obtained in the Euclidean space by Alberti, Bouchitt\'e, Seppecher. We stress that already in the Euclidean setting the situation is not covered by the classical Modica-Mortola Theorem because of the presence of the boundary term.

We recall also that Modica-Mortola type results (without a boundary term) have been proved in the Euclidean space for sub-Riemannian energies by Monti and Serra Cassano.
\end{abstract}
\maketitle

\tableofcontents

\section{Introduction and statement of the results}

It is well known that, roughly speaking, a contact manifold $(M,\theta)$ can be viewed as ``the boundary'' of a symplectic manifold
$(\Omega,\omega)$.
We refer for instance to \cite{CE}, Section 6.8. In particular, the Heisenberg group $\he n$ can be seen as the
boundary of the upper half-space $\mathcal U^n \subset \mathbb C^n$ (see, e.g. \cite{stein}, Chapter XII).

The aim of this note is to show that -- in the same spirit -- the notion of perimeter associated with the contact structure
of $(M,\theta)$ (see \cite{Ambrosio-Ghezzi-Magnani}) can be seen as a variational limit of ``solid functionals'' defined
in the symplectic manifold $(\Omega,\omega)$ that has $M$ as boundary (notice that similar approximation ``from
within $M$'' are already known, at least in the model case $\he n$: see \cite{Monti-SerraCassano}.)

More precisely, inspired by \cite{ABS2},
we show that the perimeter in $(M,\theta)$ is the $\Gamma$-limit of a family of ``phase transition'' functionals with ``low
dimensional tension effect'' in $\Omega$.

Let us start by introducing the setting of our results.
Let $\Omega$ be a bounded open set in a Stein manifold of complex dimension $N=n+1$, with symplectic form $\omega$. A complex manifold $\Sigma$,
endowed with a complex structure $J$,
 is said a Stein manifold if it admits an exhausting $J$-convex function $\phi$. We recall that $\Sigma$ is endowed with a Riemannian metric $g$ associated with $\omega$ and $J$.
We assume that $\Omega=\{ \phi < c\}$ is a sublevel set of $\phi$.  Then its boundary $M=\partial \Omega$, of real dimension $2n+1$, inherits a
natural  contact structure $(M,\theta)$, where $\theta$ is (roughly speaking) the restriction to $M$ of the $1$-form $\xi$, the contraction of the symplectic form $\omega$
along the so-called Liouville vector field $X_0$, that plays the role of the normal vector to $M$.

All precise definitions will be given in Section \ref{section-geometry}, but the idea is that bounded open sets
 in Stein manifolds are the natural generalization of domains of holomorphy in $\mathbb C^n$, having a contact manifold as boundary.
We denote also by $dy$  the volume element in $\Omega$ with respect to the metric compatible with the symplectic form $\omega$,
 and  $dv_\theta:=\theta \wedge (d\theta)^{N-1}$ the volume element in $M$ with respect to the contact form $\theta$.

Let $V$ be a double well potential, i.e, a function $V:\mathbb R\to \mathbb R$ satisfying
$$V(0)=V(1)=0,\quad V>0 \,\text{ in }\,\mathbb R\backslash\{0,1\}.$$
Given $\eps>0$ and $\lambda_\eps>0$, we define the energy functional
\begin{equation}\label{functional eps}
F_\eps (  u) :=  \eps \int_\Omega f(y, D  u(y))\, dy + \lambda_\eps \int_M V(\trace   u) \, dv_\theta,
\end{equation}
where $Du$ denotes the Riemannian gradient of $u$.
The first term in the functional $F_\eps(  u)$ is essentially the Dirichlet energy on $\Omega$ inherited from the sub-Riemannian structure of $(\Omega,\ker \xi)$, and it will be precisely written in Section \ref{section-geometry} after we have introduced all the necessary notations.  Thus $F_\varepsilon$ will be well defined  if $  u \in W^{1,2}(\Omega)$; we assign it the value infinity otherwise. Note that functions in this space have a well defined trace $\trace   u$ on $M$ with respect to the normal $X_0$.

The second term in the functional, coming from a double well potential (on the boundary), creates a phase transition on the boundary $M$ as $\varepsilon \to 0$. Here the sub-Riemannian geometry of $M$ plays an essential role in the understanding of the $\Gamma$-limit of the functional as $\varepsilon\to 0$, and this is the main innovation of the present paper.  \\

The model we have in mind is $M$ the $n$-Heisenberg group $\mathbb H^n$ and $\Omega=\mathbb H^n\times \mathbb R^+$, which is the flat model in this geometry; Indeed, by the Darboux Theorem, any $(2n+1)$-dimensional contact manifold is locally contact-diffeomorphic to the $n$-Heisenberg group (see e.g. Theorem 5.1.5, \cite{AMR}).
In this model case the functional reduces to \eqref{model-functional}.

For a general review on Heisenberg groups and their properties, we
refer to \cite{BLU}, \cite {Gromov}, \cite{stein}, and \cite {VarSalCou}.
We limit ourselves to fix some notations, following \cite{FSSC_advances}. The Heisenberg group
 $\he n$  is identified  with $\rn {2n+1}$ through exponential
coordinates. A point $p\in \he n$ is denoted by
$p=(\eta,t)$, with $\eta\in\rn{2n}$
and $t\in\R$.
   If $p$ and
$p'\in \he n$,   the group operation is defined as
\begin{equation*}
p\cdot p'=\big(\eta+\eta', t+t' + \frac12 \sum_{j=1}^n(\eta_j \eta_{j+n}'- \eta_{j+n} \eta_{j}')\big).
\end{equation*}

For fixed $q\in \he n$ and for $r>0$, left translations $\tau_q:
\he n\to\he n$ and
   not isotropic dilations $\gd_r:  \he n\to\he n$ are  defined as
\begin{equation*}\label{auto}
\tau_q(p):=q\cdot p\qquad \text{and as } \quad \delta_r(p)
:=(r\eta,r^2t).
\end{equation*}
    We denote by  $\mfrak h$
 the Lie algebra of the left
invariant vector fields of $\he n$. The standard basis of $\mfrak
h$ is given, for $i=1,\dots,n$,  by
\begin{equation*}
W_i^{\mathbb H} := \partial_{\eta_i}-\frac12 \eta_{i+n} \partial_{t},\quad W^{\mathbb H}_{i+n} :=
\partial_{\eta_{i+n}}+\frac12 \eta_{i} \partial_{t},\quad T :=
\partial_{t}.
\end{equation*}
The only non-trivial commutation  relations are $
[W^{\mathbb H}_{j},W^{\mathbb H}_{j+n}] = T $, for $j=1,\dots,n.$

The {\it horizontal subspace}  $\mfrak h_1$ is the subspace of
$\mfrak h$ spanned by $W^{\mathbb H}_1,\dots,W^{\mathbb H}_{2n}$.
Coherently, from now on, we refer to $W^{\mathbb H}_1,\dots,W^{\mathbb H}_{2n}$
(identified with first order differential operators) as to
the {\it horizontal derivatives}, and we write
$$ \mathbf{W}^{\mathbb H}:=
\{W_1^{\mathbb H}, \dots, W_{2n}^{\mathbb H}\}.
$$
Let $g_{\mathbb H}=g_{\mathbb H}(\cdot,\cdot)$ be the Riemannian metric
on $\he n$ making $W_1^{\mathbb H},
\dots, W_{2n}^{\mathbb H}, T$ orthonormal. We shall denote it by
$ \langle \cdot,\cdot\rangle_{ \mathbb H}$.
 We denote by $\nabla_{\mathbb H}$ the {\it horizontal gradient}
$$
\nabla_{\mathbb H}:=(W_1^{\mathbb H},\dots,W^{\mathbb H}_{2n}).
$$
Denoting  by $\mfrak h_2$ the linear span of $T$, the $2$-step
stratification of $\mfrak h$ is expressed by
\begin{equation*}
\mfrak h=\mfrak h_1\oplus \mfrak h_2.
\end{equation*}
    The dual space of $\mfrak h$ is denoted by $\cov 1$.  The  basis of
$\cov 1$,  dual to  the basis $\{W_1^{\mathbb H},\dots,W^{\mathbb H}_{2n},T\}$  is the family of
covectors $\{d\eta_1,\dots, d\eta_{2n},\gt_0\}$ where
$$ \gt_0
:= dt -\frac12 \sum_{j=1}^n (\eta_jd\eta_{j+n}-\eta_{j+n}d\eta_j)$$
is called the {\it contact form} in $\he n$.\\

%
In this particular case, the functional \eqref{functional eps} is written as
\begin{equation}\label{model-functional}\begin{split}
E_\eps (  u) :=
 \eps \int_{\mathbb H^n\times [0,\infty)} \Big(\sum_{j=1}^{2n}\ (W_j^{\mathbb H}   u)^2 +(\partial_z   u)^2\Big)\,dv_{\theta_0} dz
 + \lambda_\eps \int_{\mathbb H^n} V(\trace   u) \, dv_{\theta_0}.
\end{split}\end{equation}
where $dv_{\theta_0}=d\eta\,dt$. Here we realize that our functional corresponds to a hypoelliptic Dirichlet energy functional with a boundary phase transition on a contact manifold.

In general throughout this paper, if $u$ is a real function defined on a smooth manifold and $X$ is a smooth tangent vector field, we shall write
$$Xu:=\mathcal L_X u,$$
to denote the Lie derivative of $u$ along $X$.

Let us now state our main theorem, a \emph{boundary} $\Gamma$-convergence result. For the rest of the paper, we will assume that
\begin{equation}\label{lambda}
\lim_{\eps\to 0}\eps\log \lambda_\eps=\kappa \quad\text{for some constant } \kappa\in (0,\infty).
\end{equation}
We also define the limit functional on $M$ as
\begin{equation}\label{functional lim}
F(v) =
\left\{
\begin{aligned} &
 \mathbf c \, \|S_v\|_\theta\quad   \mbox{if } v\in BV_\theta(M,\{0,1\}),
\\&
+\infty \quad
  \mbox{otherwise,}  \end{aligned}
\right.
\end{equation}
 where $\mathbf c=\kappa/\pi$.  Here $\|\partial A\|_\theta$ denotes the intrinsic perimeter measure of the set $A\subset M$ associated with the contact form $\theta$, and $S_v=\partial \{v \equiv 1\}$ the singular set of $v\in BV_\theta(M,\{0,1\})$. Precise definitions will be given in Section \ref{section-subriemannian}.

\begin{theorem} \label{theorem}
For $\varepsilon>0$, consider the functional $F_\varepsilon:W^{1,2}(\Omega)\to [0,+\infty]$, %
Under scaling \eqref{lambda} we have that:
\begin{itemize}
\item[\emph{i)}] Given a sequence  $\{ {u}_\eps\}$  such that $F_{\eps}( {u}_\eps)$ is bounded when $\eps\to 0$, then $\{\trace  {u}_\eps\}$ is pre-compact in $L^1(M)$ and every cluster point belongs to
    $BV_\theta(M,\{0,1\})$.

\item[\emph{ii)}] Lower bound inequality: for every $v\in BV_\theta(M,\{0,1\})$ and every sequence $\{ {u}_\eps\}\subset W^{1,2}(\Omega)$ such that $\trace  {u}_\eps \to v$ in $L^1(M)$, there holds
$$\liminf_{\eps\to 0} F_\eps( {u}_\eps)\geq F(v).$$

\item[\emph{iii)}] Upper bound inequality: for every $v\in BV_\theta(M,\{0,1\}) $ there exists a sequence $\{ {u}_\eps\}\subset W^{1,2}(\Omega)$ such that $\trace  {u}_\eps \to v$ in $L^1(M)$ and
$$\lim_{\eps\to 0} F_\eps( {u}_\eps)= F(v).$$
\end{itemize}

\end{theorem}

The inspiration for this theorem comes from the Riemannian case. The classical theorem for phase transitions of Modica-Mortola states that a Dirichlet energy functional with a double well potential (in the interior) $\Gamma$-converges to the area functional, and thus, phase transitions happen at a minimal surface (see the survey paper \cite{Alberti:survey} or \cite{MoMo}, for instance). Later, Alberti, Bouchitt\'e and Seppecher \cite{ABS2} considered an energy functional on domain $\Omega\subset \mathbb R^3$ with a double well potential defined on the boundary of $\Omega$, which is a closed surface $M$. In this case the $\Gamma$-limit leads to a phase transition problem on the boundary surface $M$. This problem comes in relation to a model in capillarity with line tension effect.

Here we consider the sub-Riemannian version of \cite{ABS2}, in which the phase transition occurs at the boundary of a complex domain $\Omega$, which is a sub-Riemannian (contact) manifold $M$. Although the structure of the proof is similar to the Riemannian case, the main difficulties, detailed below, come precisely from the fact that the sub-Laplacian is a hypoelliptic, but not elliptic, operator,  and from the intrinsic geometry of a contact manifold.

The first $\Gamma$-convergence result in the sub-Riemannian setting is by Monti and Serra Cassano \cite{Monti-SerraCassano}, where they show the analog of the Modica-Mortola theorem for interior phase transitions in a  subdomain $\Omega$ in the framework of Carnot-Carath\'eodory spaces. As a particular case, their result holds in the case of the Heisenberg group, which is the flat model in contact geometry.

In contrast, looking at boundary phase transitions on complex domains presents several difficulties that one needs to deal with.
Therefore, we give now an overview of the paper, stressing the points at which we cannot plainly
traslate Euclidean techniques to our geometric setting, but we have to use new approaches or new
technical arguments.

First, in order to follow the methods in \cite{ABS2} for the Riemannian setting, one needs to compare our domain $\Omega$ to a product $M\times[0,\sigma)$ while still preserving the complex structure. However, in the process of flattening one needs to control the error in this procedure only by means of the derivatives appearing in the functional \eqref{model-functional} and not of the whole gradient. This is the content of Section \ref{straight freeze}.

Second, while there is an extensive literature on sub-Riemannian geometry for the Heisenberg group, the Carnot-Carath\'eodory theory on a general contact manifold has just recently been developed in \cite{Ambrosio-Ghezzi-Magnani}.  In \cite{Ambrosio-Ghezzi-Magnani}, the authors developed the theory of perimeter and BV
functions, but several results needed in our proofs were not available. One of the missing concepts was the Eikonal equation for the Carnot-Caratheodory (CC) distance, which we address in Section \ref{subsection:Eikonal}. Of course, the  Eikonal equation holds in the viscosity sense in the CC setting (see  Corollary 2.36 and  Remark 2.37. in \cite{Dragoni}), but we need a pointwise identity.

Section \ref{sec4} deals with the proof of the compactness and the lower bound inequality for the model functional \eqref{model-functional}. This part essentially follows, as in the Riemannian case, using a slicing theorem by \cite{Mon} to reduce the problem to a one dimensional one.

In Section \ref{sec5}, we prove point i) and ii) of Theorem \ref{theorem}. To do that, we need to pass from the corresponding results for the flat model, established in Section \ref{sec4}, to the ones for the original functional. In doing that, a crucial issue is to compare our boundary contact manifold to the Heisenberg group near a given point, in the spirit of the blow up theorems by \cite{Ambrosio-Ghezzi-Magnani}. Of course, the starting point is Darboux theorem. Let us give now a list of the difficulties we have subsequently to deal with.
Precise technical features are described in Remark \ref{centered}. In the Euclidean setting, for a smooth hypersurface $S$ basically
all reasonable notions of surface measure agree: De Giorgi perimeter, spherical Hausdorff measure with respect to Euclidean balls, as well as Minkowski content. Because of this,
in  \cite{ABS2} the authors use systematically the spherical Hausdorff measure. In a contact manifold the situation is different: indeed it is natural to formulate
our results in terms of perimeter and Minkowski content, and we are forced to use the Carnot-Carath\'eodory distance on the contact manifold, since it satisfies
the Eikonal equation. On the other hand, the proof of the liminf inequality (with exact constants)  is reached in  \cite{ABS2} by means of the estimate of the density of a suitable
measure associated with the functional, yielding a comparison with the Carnot-Carath\'eodory spherical Hausdorff measure. Unfortunately, an explicit
representation formula for the perimeter in terms of the Carnot-Carath\'eodory spherical Hausdorff measure is not known, and we have to use an
indirect comparison argument, that is stated in Theorem \ref{august 5}.

Many of the results that are needed are summarized later in Section \ref{densities}, as an appendix for the paper (see also \cite{FSSC_NA}).
Finally, Section \ref{sec6}, mostly analytical, concludes the proof of the main theorem, establishing the upper bound inequality (point iii) in Theorem \ref{theorem}).

\section{Reduction to a model problem}

\subsection{Geometric setting}\label{section-geometry}

We refer to \cite{CE} for an introduction to the results in this section.

\begin{definition}\label{phi} A complex manifold $\Sigma$ is said a \emph{Stein manifold } if admits
an \emph{exhausting $J$-convex function $\phi$}. To be a complex manifold means that:
\begin{itemize}
\item[\emph{i)}] $\Sigma$ is a smooth manifold of real dimension $2N$, endowed with an endomorphism
(the \emph{complex structure}) $J:T\Sigma \to T\Sigma$ satisfying $J^2=-I$ on each fiber;
\item[\emph{ii)}] \emph{$J$ is integrable}, i.e. $J$ is induced by complex coordinates on $\Sigma$.
\end{itemize}
Let now $\phi:\Sigma\to\R$ be a smooth function. We say that $\phi$ is
an exhausting function if:
\begin{itemize}
\item[\emph{iii)}] $\inf\phi > -\infty$;
\item[\emph{iv)}] $\phi^{-1} (K)$ is compact for any compact set $K\subset \R$.
\end{itemize}
We denote by $d^{\mathbb C}$ the operator defined by
$$
\Scal{d^{\mathbb C}\phi}{X} := \Scal{d\phi}{JX}\quad\mbox{for all smooth tangent vector fields $X$.}
$$
We can associate with $\phi$ the 2-form
$$
\omega = \omega_\phi:= d\xi_\phi ,\quad\mbox{where}\quad \xi = \xi_\phi: = -d^{\mathbb C}\phi.
$$
Then the function $\phi$ is said $J$-\emph{convex} if
\begin{equation}\label{J convex}
\omega_\phi (X,JX) >0 \quad\mbox{for all smooth tangent vector fields $X$.}
\footnote{Through this paper, we denote by $\Scal{\cdot}{\cdot}$
the duality between cotangent $h$-vectors and tangent $h$-vectors.
Moreover,
for sake of simplicity we write sometimes $\omega_\phi(X,Y)$
for $\Scal{\omega_\phi}{X\wedge Y}$ and  $\xi_\phi(X)$ for $\Scal {\xi_\phi} {X}$).
}
\end{equation}

\end{definition}

\begin{proposition}[\cite{CE}]\label{omega g} Suppose $\Sigma$ is a Stein manifold with respect to the complex
structure $J$ and the exhausting $J$-convex function $\phi$. Then:
\begin{enumerate}[i)]
\item  $\omega_\phi$ is a symplectic form;
\item $\omega_\phi$ is $J$-invariant, i.e. $\omega_\phi(JX,JY)=\omega_\phi(X,Y)$
for all smooth tangent vector fields $X,Y$;
\item  the bilinear form on $T\Sigma$ given by  $g_\phi (X,Y) =g(X,Y):= \omega_\phi(X,JY)$ is a
Riemannian scalar product and hence a K\"ahler metric. In particular the Riemannian volume form
$dy$ coincides with the symplectic volume form $\omega_\phi^N$ ;
\item $J$ is a $g$-isometry;
\item if we denote by $\nabla_\phi = \nabla_g$ the gradient associated with the
Riemannian scalar product $g_\phi$, then the vector field $X_\phi:= \nabla_\phi\phi$
satisfies
$$
\mc L_{X_\phi} \omega_\phi = \omega_\phi
\quad\mbox{or, equivalently,}\quad
\xi_\phi = \imath_{X_\phi} \omega_\phi,
$$
i.e. $X_\phi$ is a Liouville vector field for the symplectic form $\omega_\phi$. Here $\imath_{X}$ denotes the contraction along the vector field $X$.
\item $g_\phi (X_\phi, Z)=0$ in $M$ for all $Z\in TM$.
\end{enumerate}
%
%
%
\end{proposition}

\begin{proof} Assertions \emph{i)} and \emph{ii)} are proved in \cite{CE}, Sections 2.1 and 2.2;
assertions \emph{iii)} and \emph{v)} are contained in \cite{CE}, Lemma 2.20. As for \emph{iv)},
if $X,Y\in T\Sigma$
$$
g_\phi(JX,JY) = \omega_\phi(JY,J^2X)=-\omega_\phi(JY,X) = \omega_\phi(X,JY)=g_\phi(X,Y).
$$
Finally, vi) follows from the identity $g_\phi (X_\phi, Z)=\Scal{d\phi}{Z}$.
\end{proof}

The symplectic structure induced by $\phi$ is independent of $\phi$ in the following sense:

\begin{theorem}[\cite{CE}, Theorem 1.4.A] Let $\psi:\Sigma\to\R$ be another smooth
function satisfying iii), iv) in Definition \ref{phi}, and \eqref{J convex}. Then $(\Sigma,\omega_\phi)$ and
$(\Sigma,\omega_\psi)$ are symplectomorphic.
\end{theorem}

Let now $\Sigma$ be a Stein manifold, and let $\phi$ be the associated exhausting function.
If $c\in\R$ is a regular value of $\phi$, we set  $\Omega_{\phi,c} = \phi^{-1}(]-\infty,c[)$.
Clearly $\Omega_{\phi,c}$ is a bounded open set in $\Sigma$ with smooth compact boundary
$M_\phi$. We assume here, for sake of simplicity, that $M_\phi$ has only one connected component.

From now on, the exhausting function $\phi$ and the regular level $c$ will be fixed, and we drop the corresponding indices in our notations
and thus we write  $\Omega:= \Omega_{\phi,c}$ and  $M=\partial\Omega$. In addition, we shall
write $X_0$ for the Liouville vector field $\nabla_\phi \phi$. We notice that
$X_0\neq 0$ in a neighborhood $\mathcal M$ of $M$ since $c$ is a regular
value of $\phi$ and $M$ is compact.

We denote by $T\Omega:= (\Omega,T\Omega,\pi)$ the tangent bundle of $\Omega$, and by $T_y\Omega$ the fiber
 of $T\Omega$ over $y\in \Omega$. Coherently, we denote by $g_y$ the Riemannian metric $g$
 on $T_y\Omega$, and by $\xi_y$ and $\omega_y$ the forms $\xi$ and $\omega$ at the
 point $y$. However, as customary in differential geometry, we drop the index $y$ whenever
 this does not lead to misunderstandings.  An analogous notation will be used for $TM$, the tangent bundle of $M$.

 Finally, we denote by $d$ the Riemannian distance on $\overline{\Omega}$ with respect to the metric $g$.\\

The next step consists in proving that there is a natural $(2N-1)$-distribution associated
with the the Liouville form $\xi$ in a neighborhood $\mathcal M$ of $M$.

\begin{proposition} \label{ker theta} Set $\mc H:=\ker \xi = \{X\in T\Omega; \; \imath_X\xi =0\} \subset T\Omega$. Then
\begin{enumerate}[i)]
\item $X_0\in \mc H$;
\item $\dim \mc H = 2N-1$ in $\mc M$;
\item $\mc H$ has a orthonormal basis of the form
$$\mc B:= \{X_0,Z_1,JZ_1,Z_2,JZ_2,\dots,Z_{N-1},JZ_{N-1}\}$$
(in particular, $Z_1,JZ_1,\dots, Z_{N-1},JZ_{N-1}\in TM$ on $M$);
\item $\omega(Z_i,Z_j)=0$ for all $i,j=1,\dots,N-1$, $\omega(JZ_i,JZ_j)=0$ for all $i,j=1,\dots,N-1$,
$\omega(Z_i,JZ_j)=0 $ for all $i,j=1,\dots,N-1$, $i\neq j$, and $\omega(Z_i,JZ_i)=1$ for all $i=1,\dots,N-1$;
\item $\xi ([JZ_i, Z_i])=1$ for $i=1,\dots,N$;
\item $\mc H + [\mc H,\mc H] =T\Omega$, so that $(\mc H, g)$ is a regular sub-Riemannian structure on $\Omega$.
\end{enumerate}

\end{proposition}

\begin{proof} To prove \emph{i)} we write
$$
\Scal{\xi}{X_0} = \imath_{X_0}\omega(X_0) = \omega(X_0,X_0)=0.
$$
Next, obviously $\dim\ker\xi\ge 2N-1$. Suppose \emph{ii)} fails to be true. Then for some $y\in\mc M$ and
for any $Y\in T_y\Omega$ in $\mc M$
$$
0=\Scal{\xi_y}{Y}_y = \omega_y(X_0,Y),
$$
which contradicts $X_0\neq 0$ since $\omega$ is symplectic.

To prove \emph{iii)}, we prove first that, if  $g(X,X_0)=0$,
then $\Scal{\xi}{JX}=0$. Indeed
\begin{equation}\label{5_24 eq:1}
\Scal{\xi}{JX} =\omega(X_0, JX) = g(X_0,X) =0.
\end{equation}
Consider now $X_0^\perp \cap \ker\xi$, the $g$-orthogonal complement of $X_0$ in $\ker\xi$,
that has dimension $2N-2$, and take an unit vector $Z_1\in X_0^\perp \cap \ker\xi$.
Take now $JZ_1$, that is a unit vector by Theorem \ref{omega g}, part \emph{iv)}. By \eqref{5_24 eq:1}
$JZ_1\in \ker\xi$. We have also
\begin{equation*}\begin{split}
g(X_0,JZ_1) &= \omega(X_0,J^2Z_1)=-\omega(X_0,Z_1)
\\&
= - \Scal{\imath_{X_0}\omega}{Z_1}
=\Scal{\xi}{Z_1}=0.
\end{split}\end{equation*}
Thus
$JZ_1\in X_0^\perp \cap \ker\xi$. Finally
$$
g(JZ_1,Z_1)=\omega (JZ_1,JZ_1) = 0.
$$
Summing up, $Z_1$ and $JZ_1$ are two orthonormal vectors in $X_0^\perp \cap \ker\xi$. We can take now
an unitary vector $Z_2\in \mathrm{span}\,\{X_0,Z_1,JZ_1\}^\perp\cap\ker\xi$. Arguing as above,
$Z_2$ and $JZ_2$ are two orthonormal vectors in $ \mathrm{span}\,\{X_0,Z_1,JZ_1\}^\perp\cap\ker\xi$.
Repeating the argument, we achieve the proof of \emph{iii)}.

Let us prove \emph{iv)}. Let $i\neq j$ be given. Thanks to the anti-commutativity of $\omega$, we can
assume $i<j$. Then  $\omega(Z_i,Z_j)=\omega(JZ_i,JZ_j) = g(JZ_i,Z_j)=0$, by
construction. In addition, if $i\neq j$, then $\omega(Z_i,JZ_j)= g(Z_i,Z_j) =0$, whereas
$\omega(Z_i,JZ_i)= g(Z_i,Z_i) =1$ for $i=1,\dots,N-1$. This achieves the proof of \emph{iv)}.

To prove \emph{v)}, we have only to recall that, by classical Cartan's formula
\begin{equation*}\begin{split}
1 & =\omega (Z_i,JZ_i) = d\xi (Z_i,JZ_i) = JZ_i\Scal{\xi}{Z_i} - Z_i\Scal{\xi}{JZ_i}
- \Scal{\xi}{[Z_i,JZ_i]}
\\
\hphantom{xxx}&= - \Scal{\xi}{[Z_i,JZ_i]}.
\end{split}\end{equation*}
Finally, \emph{vi)}  follows from \emph{ii)} and \emph{v)}.
\end{proof}

\begin{remark}
We can always take $Z_j$ and $JZ_j$, $j=1,\ldots,N-1$, that commute with $X_0$.
\end{remark}

Let us remind now the following well-known definition.

\begin{definition}\label{contact def} Let $M$ be a smooth $(2n+1)$-manifold. A 1-form $\theta$
is said a \emph{contact form} if $\theta\wedge ( d\theta)^{2n}  \neq 0$
on $M$. The set $\ker \theta \subset TM$ is called a \emph{contact distribution}.
Let $M_1$ and $M_2$ be two contact $(2n+1)$-manifolds endowed with the contact forms
$\theta_1$ and $\theta_2$. A smooth diffeomorphism $f:M_1\to M_2$ is said a
\emph{contact map} if $\theta_1 = f^*\theta_2$ and hence $f_*\ker\theta_1=\ker\theta_2$.
\end{definition}

The following result is well known:

\begin{proposition}\label{contact 1} Denote by $i:M \to \overline\Omega$ the natural embedding.
Then the 1-form $\theta:= i^*( \imath_{X_0}\omega)$ is a contact form on $M$,
and therefore $\ker\theta$ defines a contact distribution on $M$.
\end{proposition}


\begin{remark}\label{v-theta} By the previous proposition, we can choose $dv_\theta :=\theta\wedge ( d\theta)^{N-1}$ as the volume form  in $M$.
For sake of simplicity, if $A\subset M$ we shall write $v_\theta (A)$ for $\int_A\, dv_\theta$.

Moreover (see e.g. \cite{Blair:book}) there exists a global vector field $T$ on $M$
satisfying
$\Scal{\theta}{T}=1$ and orthogonal to $\ker \theta$ with respect to the Riemannian metric induced
by $g$ on $TM$ (still denoted by $g$), that
is called the characteristic vector field or Reeb vector field
of the contact structure.

\end{remark}

\begin{proposition}\label{contact 2} The contact distribution $\ker\theta$ carries a natural symplectic structure
$$
d\theta = d i^*(\xi) = i^*(d\xi) = i^*\omega.
$$
\end{proposition}

\begin{proof} We have only to prove that $i^*\omega$ is non-degenerate on $\ker \theta$. To this end,
let $X\in \ker\theta$ be such that $i^*\omega (X,Y) =0$ for all $Y\in \ker\theta$. If $x\in M$, then, keeping
in mind that $i(x)=x$, we have
\begin{equation*}
0 = i^*\omega_x (X,Y) =  \omega_{i(x)}(di (X),di(Y)) .
\end{equation*}
We remark now that any tangent vector $Z$ to $\Omega$ at a point of $M$
can be written in the form $Z=di(Y)+\lambda X_0$ with $\lambda\in\R$ and $Y\in TM$, since
$X_0$ is normal to $TM$.
 On the other hand
$$
\omega_{i(x)}(di(X),X_0) = -\xi_{i(x)}(di(X)) = -\theta_x(X) =0,
$$
and hence $\omega_{i(x)}(di(X),Z)=0$ for all $Z\in T_{i(x)}\Omega$, achieving the proof of the proposition
since $di$ is injective.
\end{proof}

\begin{proposition}\label{contact 3} The vector fields $Z_j$ and $JZ_j$, $j=1,\dots,N-1$ (that belong to $T\Omega$), being tangent to $M$
at the points of $M$, can be identified with vectors in $\ker\theta\subset TM$ and are a symplectic basis of $\ker\theta$.
Moreover, $\ker\theta$ inherits the Riemannian metric  from the ambient space (denoted by the same letter $g$) and $Z_j$ and $JZ_j$, $j=1,\dots,N-1$ give
an orthonormal basis of $\ker\theta$.

\end{proposition}

\begin{proof} It is enough to apply Theorem \ref{ker theta}, \emph{iv)}.
\end{proof}

%
%

We are ready now to introduce our main object of study. We write $N=: n+1$. If $p$ is a tangent smooth vector field to $\Omega$, we denote
$$\Lambda(y,p):=  \sum_{j=1}^{n} g_y(  Z_j(y),p)^2 + \sum_{j=1}^{n} g_y(  JZ_j(y),p)^2 +g_{y} ( X_0(y),p)^2.$$
Let now $f:  T\Omega\to \mathbb R$ be a smooth function such that:
\begin{enumerate}
\item[\emph{H1.}] $0\le f(y,p) \le C\, g_{y}(p,p) $ for all $y\in \Omega$ and $p\in T_y\Omega$.;
\item[\emph{H2.}] for any $\sigma >0 $ small enough there exists a neighborhood $U_\sigma$ of $M$ in $\Omega$, $U_\sigma\subset\mc M$,  such that
\begin{equation*}
(1-\sigma)\Lambda(y,p) \le f(y,p) \le (1+\sigma)\Lambda(y,p)
\end{equation*}
for all $y\in U_\sigma$ and $p\in T_y\Omega$.
\end{enumerate}
If there is no way to misunderstanding, we denote by $\nabla=\nabla_g$ the Riemannian gradient in $\Omega$.
We notice that, if $X$ is any vector field on $\Omega$, then $g_y(X,\nabla_gu)^2 = |X u|_{g}^2$. Keeping in mind  that
$$
g_y(X,\nabla_g u) = \Scal{d u}{X}= \mc L_X u=Xu,
$$
we can write
\begin{equation}\label{functional 1}
 \int_{U_\sigma} \Lambda(y, \nabla  u(y))\, dy
 =\int_{U_\sigma}\Big(\sum_{j=1}^{n}
( Z_j   u)^2 + \sum_{j=1}^{n}
(  JZ_j  u)^2
+ (  X_0  u)^2\Big)\, dy.
\end{equation}

\subsection{Straightening the domain and freezing the functional}\label{straight freeze}

It is well known that, straightening the integral curve of $X_0$, we can transform the neighborhood $U_\sigma$ of $M$ into the cylinder $M\times [0,\sigma)$.
More precisely, we consider the map
$$
\Phi =\Phi(x,z):  M\times [0,\sigma) \to \Omega
$$
defined by
\begin{equation}\label{def-Phi}
\dfrac{\partial\Phi}{\partial z} = - {X_0}(\Phi)\quad\mbox{and } \Phi(x,0)=   i(x).
\end{equation}

If $\sigma>0$ is small enough, then $\Phi $ is a smooth diffeomorphism. We set now
$$
\tilde Z_j := (\Phi^{-1})_* Z_j, \quad \widetilde{JZ_j} := (\Phi^{-1})_* JZ_j, \quad j=1,\dots,n,
$$
and
$$\tilde \xi:=\Phi^*(\xi),\quad \tilde \omega:=\Phi^*(\omega).$$

In addition, we define the projection
$$
\pi: M\times [0,\sigma) \to M
$$
given by $\pi(x,z) = x$. We notice that, if $\alpha$ is a differential form on $M$,
then $\pi^*\alpha$ is its ``natural'' extension on $M\times [0,\sigma)$.

The following result follows straightforwardly by algebraic arguments.

\begin{lemma}\label{jguy} We remind that we have set $\theta:= i^*\xi$. Then we have:
\begin{itemize}
\item[i)] $\tilde\xi = e^{-z}\, \pi^*\theta$;
\item[ii)] $\tilde\omega = d(e^{-z}\, \pi^*\theta)$;
\item[iii)] $\ker \tilde\xi = \ker\theta \times  \mathbb R$.
\end{itemize}
\end{lemma}

Moreover, we have the following Lemma:
\begin{lemma}\label{Phi} We have:
\begin{itemize}
\item[i)] $ (\Phi^{-1})_* X_0 = (0,-1) = -\partial_z$;
\item[ii)] $\Phi^* (\omega^N) = e^{-Nz} \, \pi^*(dv_\theta)\wedge dz$.
\end{itemize}
\end{lemma}
\begin{proof}
Point i) comes by the way we have defined $\Phi$ in \eqref{def-Phi}. To prove ii), we notice
that, by Lemma \ref{jguy},
\begin{equation*}\begin{split}
\Phi^* (\omega^N) &= \tilde\omega^N = (d (e^{-z}\, \pi^*\theta))^N
=  e^{-Nz} (-dz\wedge \pi^*\theta + \pi^*(d\theta))^N
\\&
=  -e^{-Nz} \, dz\wedge \pi^*\theta \wedge (\pi^*(d\theta))^{N-1}
\\&
=  e^{-Nz} \, \pi^*\theta \wedge (\pi^*(d\theta))^{N-1}\wedge dz
\\&
=  e^{-Nz} \, \pi^*\big(\theta \wedge (d\theta)^{N-1}\big) \wedge dz.
\end{split}\end{equation*}
\end{proof}

\begin{remark}
For sake of simplicity, from now on we shall write $dv_\theta \wedge dz$ for $\pi^*(dv_\theta) \wedge dz$.
\end{remark}

%
%
%


If we perform the change of variables $y=\Phi(x,z)$,
keeping in mind that ${ X_0}u = \partial_z (u\circ\Phi)$ and
${ Z_j}u = \tilde  Z_j (u\circ\Phi)$, and setting $\tilde u:=   u\circ\Phi$,
the functional \eqref{functional 1} becomes
\begin{equation}\label{02_17 :1}\begin{split}
\int_{U_\sigma}  & \Lambda(y, D  u(y))\, dy
\\
=& \int_{M\times [0,\sigma)} \Big( \sum_{j=1}^{n}
({ \tilde Z_j } \tilde u)^2 +  \sum_{j=1}^{n} ({ \widetilde{JZ}_j  } \tilde u)^2
+( \partial_z  \tilde u)^2 \Big)\;  e^{- Nz}dv_\theta \wedge dz.
\end{split}\end{equation}

We recall now that the vector fields $Z_1,\dots,Z_{n}$ and
$JZ_1,\dots,JZ_{n}$ in $\overline\Omega$ are tangent to $M$ in $M$, and hence
can be identified with vector fields tangent to $M$ at the points of
the form $(x,0)\in M\times [0,\sigma)$.
 Thus in $M\times [0,\sigma)$ we set:
\begin{equation*}\label{5_25 eq:2}
 \tilde Z_j^0(x,z) :=\tilde Z_j(x,0)=Z_j(i(x))
\end{equation*}
and
\begin{equation*}\label{5_25 eq:3}
\widetilde{JZ}^0_j(x,z):= \widetilde{JZ}_j(x,0) = JZ_j(i(x)).
\end{equation*}

\bigskip

The core of this Section is the following Proposition, that states basically that our functional near
the boundary $M$ of $\Omega$ is equivalent -- in a suitable way -- to a variational functional $\tilde{F}_{\eps,\sigma}$
satisfying the following properties:
\begin{itemize}
\item $\tilde{F}_{\eps,\sigma}$ is defined in a cylindric
region $M\times [0,\sigma)$;
\item $\tilde{F}_{\eps,\sigma}$ is associated with the vector fields $\widetilde{Z}^0_j$ and $\widetilde{JZ}^0_j$
(that are tangent to $M$ and are
independent of the ``vertical'' variable) and to a purely vertical vector field $\partial_z$.
\end{itemize}

More precisely, we write
$$
\tilde{F}_{\eps,\sigma}( \tilde u) := \int_{M\times [0,\sigma)} \Big(  \sum_{j=1}^{n}
( {Z_j^0}  \tilde u)^2+
 \sum_{j=1}^{n}
( {JZ_j^0}  \tilde u)^2+( \partial_z \tilde u)^2 \Big)\, dv_\theta\wedge dz.
$$
\begin{proposition}\label{model-prop}
Using the above notations, we have
\begin{equation*}\label{model-prop eq:1}\begin{split}
(1+& O(\sigma))\int_{U_\sigma}  \Lambda(y, \nabla  u(y))\, dy = \tilde{F}_{\eps,\sigma}( \tilde u)
\end{split}
\end{equation*}
provided we take $\sigma$ small enough.
\end{proposition}

Obviously, the exponential $ e^{- Nz}$ in \eqref{02_17 :1} gives no trouble. The remaining part of the proof of Proposition \ref{model-prop} is more delicate: in $M\times [0,\sigma)$ we have to replace (e.g.) the vector fields
$\tilde Z_j$ by their value frozen at $z=0$ and to control the error. However, a straightforward
application of the mean value theorem does not fit our purposes, because this estimate
of the error would involve {\it all} derivatives of $\tilde u$, that in turn are not controlled
by the original functional, where only derivatives along a particular distribution appear.
Thus, we have to show that we can control the error only by means of the derivatives
appearing in the functional. This is the aim of the following technical lemma.

\begin{lemma}\label{deriv-z}
If $j=1,\dots, n$ and $0<s<z\leq 1$, then
\begin{equation}\label{deriv-1}
\begin{split}
\partial_z \tilde Z_j(x,s)&= \sum_{\ell=1}^n \lambda_{\ell,\, j}(x,s,z) \tilde Z_\ell(x,z)\\
 &\hspace{1em}+\sum_{\ell=1}^n \lambda_{\ell+n,\, j}(x,s,z) \widetilde{JZ}_\ell(x,z) + \lambda_{0,\,j}(x,s,z) \partial_z.
 \end{split}
\end{equation}
Similarly,
\begin{equation}\label{deriv-2}
\begin{split}
\partial_z \widetilde{JZ}_j(x,s)&= \sum_{\ell=1}^n \lambda'_{\ell,\, j}(x,s,z) \tilde Z_\ell(x,z)\\
 &\hspace{1em}+\sum_{\ell=1}^n \lambda'_{\ell+n,\, j}(x,s,z) \widetilde{JZ}_\ell(x,z) + \lambda'_{0,\,j}(x,s,z) \partial_z,
\end{split}
\end{equation}
Moreover, there exists a geometric constant $C>0$ such that
$|\lambda_{0,\,j}|+ \dots +|\lambda_{2n,\,j}|\leq C$ and  $|\lambda'_{0,\,j}|+ \dots +|\lambda'_{2n,\,j}|\leq C$ for any $j=1,\dots,n$.
\end{lemma}
\begin{proof}
We prove \eqref{deriv-1}; the proof of \eqref{deriv-2} is analogue.
First, we prove that for any $j=1,\dots,n$, the vector fields $\partial_z \tilde Z_j(x,s)$, $\partial_z \widetilde{JZ}_j(x,s)$ belong to $\ker\tilde \xi(x,s)$.
Then the assertion follows since $\ker\tilde \xi(x,s)=\ker\tilde \xi(x,z)$ for any $0<s\leq z$, by Lemma \ref{jguy}, iii).

We show that for any $j=1,\dots,n$
\begin{equation}\label{april 14 eq:1}
\begin{split}
\partial_z \tilde Z_j &= \sum_{\ell=1}^n \big\{g([Z_j,X_0],Z_\ell)\circ \Phi \big\}\tilde Z_\ell \\
\hspace{1em} &+ \sum_{\ell=1}^n \big\{g([Z_j,X_0],JZ_\ell)\circ \Phi\big\}\widetilde{JZ}_\ell + \big\{ g([Z_j,X_0],X_0)\circ \Phi\big\} \partial_z.
\end{split}
\end{equation}
In order to prove \eqref{april 14 eq:1}, we notice preliminarily that
\begin{equation*}
\partial_z \tilde Z_j=[(\Phi^{-1})_* Z_j,\partial_z]=[(\Phi^{-1})_* Z_j,(\Phi^{-1})_* X_0]=(\Phi^{-1})_*[Z_j,X_0],
\end{equation*}
where the last equality comes from \cite{AMR}, Proposition 4.2.23.

Let us prove now that $[Z_j,X_0] \in \ker \xi$. Using Proposition 7.4.11 in \cite{AMR}, we have
\begin{equation*}
\begin{split}
\omega(Z_j,X_0)&= d\xi(Z_j,X_0) \\
&=Z_j\langle \xi|X_0\rangle - X_0\langle \xi | Z_j \rangle - \langle \xi | [Z_j,X_0] \rangle = -\langle \xi | [Z_j,X_0] \rangle.
\end{split}
\end{equation*}
On the other hand
\begin{equation*}
\omega(Z_j,X_0)=\omega(X_0,J^2Z_j)=g(X_0,JZ_j)=0,
\end{equation*}
since the basis $\{X_0,Z_1,\dots,Z_n,JZ_1,\dots,JZ_n\}$ is orthonormal, hence $[Z_j,X_0] \in \ker \xi$. Thus,
\begin{equation*}
[Z_j,X_0]=\sum_{\ell=1}^n g([Z_j,X_0],Z_\ell)Z_\ell\,+\, \sum_{\ell=1}^n g([Z_j,X_0],JZ_\ell)JZ_\ell \,+\, g([Z_j,X_0],X_0)X_0,
\end{equation*}
and hence
\begin{equation*}
\begin{split}
(\Phi^{-1})_*([Z_j,X_0])&=\sum_{\ell=1}^n  \big\{g([Z_j,X_0],Z_\ell)\circ \Phi \big\} \tilde Z_\ell+ \sum_{\ell=1}^n  \big\{g([Z_j,X_0],JZ_\ell)\circ \Phi\big\} \widetilde{JZ}_\ell \\ &\hspace{1em} + \big\{g([Z_j,X_0],X_0)\circ \Phi\big\} \partial_z.
\end{split}
\end{equation*}
This proves \eqref{april 14 eq:1} and concludes the proof of Lemma \ref{deriv-z}.
\end{proof}

For the sake of simplicity, sometimes we denote the vector fields $$\tilde Z_1,\dots,\tilde Z_n, \widetilde{JZ}_1,\dots,\widetilde{JZ}_n\quad\mbox{by}\quad
\widetilde W_1,\dots,\widetilde W_{2n},$$
and we set
$$\widetilde{\textbf{{W}}}=\{\widetilde W_1,\dots,\widetilde W_{2n}\}.$$
Analogously we define the $\widetilde W_j^0$'s by freezing the $\widetilde W_j$ at $z=0$
and we set
$$\widetilde{\textbf{W}}^0=\{\widetilde W_1^0,\dots,\widetilde W_{2n}^0\}.$$

 With these notations, Lemma \ref{deriv-z} reads as follows: for any $j=1,\dots,2n$, and $0<s<z\leq1$, there exists $2n$ coefficients $\lambda_{0,\,j}, \lambda_{1,\,j},\dots,\lambda_{2n,\,j}$ such that $|\lambda_{1,\,j}|+\dots+|\lambda_{2n,\,j}|\leq C$, and
\begin{equation}\label{deriv-W}
\partial_z \widetilde W_j(x,s)=\sum_{\ell=1}^{2n}\lambda_{\ell,\,j}(x,s,z)\widetilde W_\ell(x,z) + \lambda_{0,\,j}(x,s,z)\partial_z.
\end{equation}

We can give now the proof of Proposition \ref{model-prop}.
\begin{proof}[Proof of Proposition \ref{model-prop}]
By \eqref{deriv-W}, we have that for any $j=1,\dots,2n$, the following holds:
\begin{equation*}
\begin{split}
\widetilde W_j(x,z)&=\widetilde W_j(x,0) + \int_0^z \partial_ z \widetilde W(x,s)ds\\
&=\widetilde W_j(x,0) + \sum_{\ell=1}^{2n} \left ( \int_0^z \lambda_{\ell,\, j}(x,s,z)ds \right) \widetilde W_\ell(x,z)
+ z\lambda_{0,\,j}(x,z)\partial_z;
\end{split}
\end{equation*}
so that
\begin{equation*}
\begin{split}
\widetilde W_j(x,z)&=\widetilde W_j(x,0) + \sum_{\ell=1}^{2n} \hat\lambda_{\ell,\, j}(x,z) \widetilde W_\ell(x,z)
+ \hat \lambda_{0,\,j}(x,z)\partial_z,
\end{split}
\end{equation*}
where $\hat\lambda_{0,\,j},\dots, \hat\lambda_{2n,\, j}=O(z)$
as $z\to 0$ for $j=1,\dots,2n$.
Setting, for any $j=1,\dots,2n$:
$$\widetilde W_j^0(x,z):=\widetilde W_j(x,0),$$
we have
\begin{equation}\label{L_W}
\begin{split}
({\widetilde W_j} \tilde u)(x,z)&=({\widetilde W_{j}^0}\tilde u)(x,z) + \sum_{\ell=1}^{2n} \hat\lambda_{\ell,\, j}(x,z)({\widetilde W_{\ell}}\tilde u)(x,z)\\
&\hspace{1em} + \hat\lambda_{0,\,j}(x,z){\partial_z}\tilde u(x,z).
\end{split}
\end{equation}
To conclude the proof we have to show that
\begin{equation}\label{final-deriv}
\begin{split}
\sum_{j=1}^{2n}&({\widetilde{W}_j}\tilde u)^2 + (\partial_z \tilde u)^2
-\big(\sum_{j=1}^{2n}({\widetilde{W}^0_j}\tilde u)^2 + (\partial_z \tilde u)^2\big)\\
& =\sum_{j=1}^{2n}({\widetilde{W}_j}\tilde u)^2 -\sum_{j=1}^{2n}({\widetilde{W}^0_j}\tilde u)^2
\\&
=O(\sigma)\big(
\sum_{j=1}^{2n} ({\widetilde{W}_j}\tilde u)^2 + (\partial_z \tilde u)^2
\big).
\end{split}
\end{equation}
For any $j=1,\dots,2n$, we set:
$$a_j:={\widetilde{W}_j}\tilde u,\quad b_j:={\widetilde{W}^0_j}\tilde u,\quad c_0={\partial_z}\tilde u,$$
so that \eqref{final-deriv} becomes
\begin{equation}\label{final-deriv 2}
\begin{split}
\sum_{j=1}^{2n}  a_j^2 - \sum_{j=1}^{2n} b_j^2
 = \big( \sum_{j=1}^{2n}  a_j^2 + c_0^2\big)
- \big(\sum_{j=1}^{2n} b_j^2 + c_0^2 \big)
 =
O(z) \big( \sum_{j=1}^{2n}  a_j^2 + c_0^2\big).
\end{split}
\end{equation}
By \eqref{L_W}, we have that
$$a_j=b_j + \sum_{\ell=1}^{2n}\hat \lambda_{\ell,\,j}a_\ell + \hat\lambda_{0,\,j} c_0,$$
and hence
$$\sum_{j=1}^{2n}\big(a_j - \sum_{\ell=1}^{2n}\hat \lambda_{\ell,\,j} a_\ell - \hat\lambda_{0,\,j} c_0\big)^2=\sum_{j=1}^{2n} b_j^2.$$
We compute:
\begin{equation*}
\begin{split}
&\sum_{j=1}^{2n}\big(a_j - \sum_{\ell=1}^{2n}\hat \lambda_{\ell,\,j} a_\ell - \hat\lambda_{0,\,j} c_0\big)^2
=\sum_{j=1}^{2n}a_j^2 + c_0^2\sum_{j=1}^{2n}\hat\lambda^2_{0,\,j}  + \sum_{j=1}^{2n} \big ( \sum_{\ell=1}^{2n}\hat\lambda_{\ell,\,j} a_\ell\big)^2 \\
&\hspace{1em}- 2 \sum_{j=1}^{2n}a_j\sum_{\ell=1}^{2n}\hat\lambda_{\ell,\,j} a_\ell -2 c_0\sum_{j=1}^{2n}a_j \hat\lambda_{0,\,j}
 -2c_0\sum_{j,\ell=1}^{2n}\hat\lambda_{\ell,\,j} a_\ell \hat\lambda_{0,\,j} \\
 &= \sum_{j=1}^{2n}a_j^2 +  I_0 + I_1 +I_2+I_3+I_4.
\end{split}
\end{equation*}
It remains to estimate $I_i$ for $i=0,\dots,4$:
\begin{equation*}\begin{split}
&I_0=c_0^2\sum_{j=1}^{2n}\lambda^2_{0,\,j}\leq O(\sigma)c_0^2;\\
&I_1\leq \sum_{j=1}^{2n}\big(\sum_{\ell=1}^{2n}\hat\lambda_{\ell,\,j}^2\big)\sum_{\ell=1}^{2n}a_\ell^2\leq O(\sigma)\sum_{\ell=1}^{2n}a_\ell^2;\\
&|I_2|\leq 2\big(\sum_{j=1}^{2n}a_j^2\big)^{1/2}\big(\sum_{j=1}^{2n} \big(\sum_{\ell=1}^{2n}\hat\lambda_{\ell,\,j} a_\ell\big)^2 \big)^{1/2}\leq
2\big(\sum_{j=1}^{2n}a_j^2\big)^{1/2}\big(\sum_{j,\ell=1}^{2n} \hat\lambda^2_{\ell,\,j} a_\ell^2 \big)^{1/2}\\
&\quad\leq O(\sigma)\sum_{\ell=1}^{2n}a_\ell^2.\\
&|I_3|\leq 2 |c_0|\sum_{j=1}^{2n}|\hat \lambda_{0,j}a_j|\leq O(\sigma)|c_0|\big(\sum_{j=1}^{2n}a_j^2\big)^{1/2} = O(\sigma)\big(\sum_{j=1}^{2n}a_j^2 + c_0^2\big).\\
&|I_4|\leq 2|c_0|\sum_{j,\ell=1}^{2n}|\hat\lambda_{\ell,j}a_\ell \hat\lambda_{0,j}|\leq O(\sigma)|c_0|\big(\sum_{j=1}^{2n}a_j^2\big)^{1/2}
= O(\sigma)\big(\sum_{j=1}^{2n}a_j^2 + c_0^2\big).
\end{split}
\end{equation*}

This yields \eqref{final-deriv 2} and then achieves the proof of the proposition.
\end{proof}
%

For $\varepsilon>0$, the functional
$\tilde{F}_{\eps,\sigma} :W^{1,2}(M\times [0,\sigma ))\to [0,+\infty]$
reads as
\begin{equation}\label{model bis}\begin{split} \tilde{F}_{\eps,\sigma}(\tilde u) :=
\eps  & \int_{M\times [0,\sigma)} \Big(  \sum_{j=1}^{2n}
( \tilde W_j^0  \tilde u)^2+( \partial_z \tilde u)^2 \Big)\, dv_\theta\wedge dz
\\&
+ \lambda_\eps \int_M V(\trace \tilde u) \, dv_\theta,
\end{split}\end{equation}
that, according to Proposition \ref{model-prop}, is nothing but an approximation of the original functional $F_\eps$
in a neighborhood of $M$, written in the new ``straightened'' coordinates.


\begin{remark}\label{conventions}From now on we shall work only on the straight cylinder $M\times [0,\sigma)$,
and hence, to avoid cumbersome notations,  we shall drop everywhere the tilde
if there is no way of misunderstanding.

In addition, since the vector fields $W_1^0,\dots W_{2n}^0$ are independent
of $z\in [0,\sigma)$, we can identify them with vector fields in $TM$.

\end{remark}

The proof of our $\Gamma$-convergence Theorem \ref{theorem}, at least parts  \emph{i)} and \emph{ii)}, will follow from the following analogue result for the approximate functional \eqref{model bis} using Proposition \ref{model-prop}.

\begin{theorem} \label{theorem bis}
Assume that the  scaling \eqref{lambda} holds. Then, for all $\sigma>0$ small enough, we have:
\begin{itemize}
\item[\emph{i*)}] Given a sequence  $\{u_\eps\}$  such that $\tilde{F}_{\eps,\sigma}(u_\eps)$ is bounded when $\eps\to 0$, then $\{\trace u_\eps\}$ is pre-compact in $L^1(M)$ and every cluster point belongs to
    $BV_\theta(M,\{0,1\})$.

\item[\emph{ii*)}] For every $v\in BV_\theta(M,\{0,1\})$ and every sequence $\{u_\eps\}\subset W^{1,2}(M\times [0,\sigma) )$ such that $\trace u_\eps \to v$ in $L^1(M)$, there holds
$$\liminf_{\eps\to 0} \tilde{F}_{\eps,\sigma}(u_\eps)\geq F(v).$$

\end{itemize}
\end{theorem}


The scheme of this paper is the following: in Section \ref{sec5} we shall prove \emph{i*)} and \emph{ii*)} of Theorem \ref{theorem bis}.  Finally, in Section \ref{sec6} we shall prove  \emph{iii)} of Theorem
\ref{theorem}, thus completing the proof of of Theorem
\ref{theorem}.

\section{Sub-Riemannian structures}\label{section-subriemannian}

Although there is a wide literature on  Carnot-Carath\'eodory spaces over $\mathbb R^n$,  here we are looking at  manifolds \cite{Ambrosio-Ghezzi-Magnani,Karmanova-Vodopyanov:libro}, for which some of the theory needs to be developed. We will briefly recall all the necessary ingredients. Though several of the following results hold for
general geometric structures, for reader's convenience we state them in
our setting, i.e. in the contact manifold $(M,\theta)$ endowed with the metric $g$. According to Remark
\ref{conventions}, we denote by
$$
\mathbf{W}^0 = \{W_1^0, \dots, W_{2n}^0\}
$$
our fixed orthonormal basis of $\ker\theta$, and by $T$ the Reeb vector field.

We next define the distance $d_c$ on $M$. Recall that
an absolutely continuous curve $\gamma:[0,T]\to M$ is a {\it
subunit curve} with respect to $W_1^{0},\ldots,W^{0}_{2n}$
 if there are real measurable
functions $c_1,\dots,c_{2n}$, defined in $[0,T]$, such that
$$\sum_{j=1}^{2n} c_j^2(s)\le 1\quad\text{and}\quad\dot\gamma(s)=\sum\limits_{j=1}^{2n}\,c_j(s) W^{0}_j(\gamma(s)),\quad \text{for a.e. } s\in [0,T].$$
Then, if $p,q\in M$,
    the cc-distance (Carnot-Carath\'eodory
distance) $d_c(p,q)$ is
$$
d_c(p,q)\defin\inf\left\{T>0:\;\text{$\gamma$ is subunit, }
\; \gamma(0)=p,\,\gamma(T)=q\right\}.
$$
The set of subunit curves joining $p$ and $q$ is not empty, by
Chow's theorem, since  the rank of the
Lie algebra generated by $W_1^{0},\ldots,W_{2n}^{0}$ is $2n+1$. Moreover,  $d_c$
is a distance on $M$ inducing the same topology as the standard distance on $M$ as a differentiable manifold  (cf. \cite{Ambrosio-Ghezzi-Magnani,Agrachev-Sachkov}).  $(M,d_c)$ is called a Carnot-Carath\'eodory space.




We recall that, because the topologies induced by $d_c$
and the usual one coincide, the topological dimension
of $M$ is $2n+1$.  On the contrary the {\it homogeneous dimension} of $M$
is the integer $Q:=2n+2$.

In the particular case that $M$ is the Heisenberg group, we write the Carnot-Carath\'eodory distance by $d_c^{\mathbb H}$.

Throughout the paper we will denote by $B_r(p)=B(r,p)$ the open ball (centered at $p$ of radius $r$ ) in $M$ associated with the distance $d_c$ and by $B^{\mathbb H}_r(p)=B^{\mathbb H}(p,r)$ the open ball in $\mathbb H^n$ associated with the distance $d_c^{\mathbb H}$.


\subsection{Functions of bounded variation}\label{subsection:BV}

The aim of this section is to recall some basic facts about $BV$-functions on a contact manifold $M$
and, in particular, the coarea formula, following \cite{Ambrosio-Ghezzi-Magnani} and \cite{miranda}. Since the volume form
$dv_\theta$ has been chosen once for all, if $X\in \Gamma(M,\ker\theta)$ is a continuously differentiable section of $\ker \theta$, we can define the function
$\divergence  X$ by the identity
$$
(\divergence X)dv_\theta : =\mc L_X (dv_\theta) = d(i_X( dv_\theta)).
$$
Using properties of exterior derivatives and differential forms, we see that $\divergence X$ satisfies
\begin{equation}\label{div}
-\int_M\phi \divergence X dv_\theta=\int_M (X\phi) d v_\theta \quad \mbox{for any}\;\;\phi\in C^1_c(M).
\end{equation}
Applying \eqref{div} to the product $h\phi$, with $h \in C^1(M)$ and $\phi\in C^1_c(M)$, using Leibnitz rule and the identity
$$\divergence(\phi X)=\phi \divergence X+X\phi,$$
we deduce that
$$-\int_M h \divergence (\phi X) d v_\theta=\int_M \phi (X h) dv_\theta.$$
We use this identity to define now the derivative of $h$ along $X$ in the sense of distributions. We say that a measure with finite total variation, that we will denote by $D_X h$, represents in an open set $U\subset M$ the derivative of $h$ along $X$ in the sense of distributions, if
\begin{equation*}
-\int_U h \,\divergence(\phi X)\,dv_\theta=\int_U \phi \,d D_X h,\quad \forall \phi\in\mathcal C^\infty_0(U).
\end{equation*}

In \cite{Ambrosio-Ghezzi-Magnani}, Proposition 2.1, it is proved that for $h\in L^1_{\mathrm{loc}}(M,dv_\theta)$, $D_X h$ is a signed measure with finite total variation in $U$ if and only if
\begin{equation}\label{variation}
\sup\left\{ \int_U h \divergence(\phi X) dv_\theta, \; \phi\in\mc D(U), |\phi|\le 1\right\} < \infty,
\end{equation}
and if this happens the supremum above equals $|D_X h|$.
We can now define the space $BV_\theta$.
\begin{definition} Let $U\subset M$ be an open set. We say that $h\in L^1_{\mathrm{loc}}(M,dv_\theta)$
belongs to $BV_\theta(U)$ if
$$
\sup \{ |D_X h|(U): X\in \Gamma(M,\ker\theta),\; g(X,X)\le 1\} < \infty.
$$
\end{definition}

\vspace{1em}
If ${\bf W}^{0}:=\{W_1^{0},\dots W_{2n}^{0}\}$ is the orthonormal basis of $\ker\theta$ and $f\in L^1_{\mathrm{loc}}(M,dv_\theta)$,
we define a vector-valued measure
$$
{\bf W}^{0}h := (W_1^{0}h,\dots,W^{0}_{2n}h).
$$
\begin{proposition}[see \cite{Ambrosio-Ghezzi-Magnani}, Theorem 3.1] \label{Ambrosio-Ghezzi-Magnani} If $h\in BV_\theta(U)$, then
\begin{itemize}
\item[i)] the total variation of ${\bf W}^{0}h$ in $U$ is finite. We denote it by $|{\bf W}^{0}h |(U)$;
\item[ii)] $h$ belongs to $BV(U,d_c, dv_\theta)$, the $BV$-space in metric measure space $(M,d_c, dv_\theta)$
in the sense of \cite{miranda}. We notice that $(M,d_c, dv_\theta)$ is a ``good'' metric space in the sense
of \cite{miranda}, as pointed out also in \cite{Ambrosio-Ghezzi-Magnani};
\item[iii)] $|{\bf W}^{0}h |(U) = \sup \{ |D_X h|(U): X\in \Gamma(M,\ker\theta)\; g(X,X)\le 1\} $;
\item[iv)]  $|{\bf W}^{0}h |(U) = \|Dh\|(U)$, where $\|Dh\|(U)$ is the total variation of $h$ in the sense of \cite{miranda}.

\end{itemize}

\end{proposition}

\begin{definition}\label{finite-per} If $E\subset M$ is a Borel set, we say that $E$ has (locally) \textit{finite perimeter} in $U$ if $\chi_E \in BV_\theta(U)$. Moreover we denote
$$\|\partial E\|_\theta(U) :=|{\bf W}^{0}\chi_E|(U).$$
For $h\in BV_\theta(U,\{0,1\})$, i.e., $h=\chi_E$, we denote by $S_h$ the set of points where the upper and lower approximate limits of $h$ differ. In this case we write $S_h=\partial E\cap U$,  the jump set of $h$ in $U$. 
\end{definition}

Next, from \eqref{Ambrosio-Ghezzi-Magnani} we know that if $\chi_E\in BV_\theta(U)$, then
for $\|\partial E\|_\theta$-a.e. $x\in U$,
\begin{equation}\label{formula1}\liminf_{r\downarrow 0} \frac{\min\{ v_\theta(B_r(p)\cap E),v_\theta(B_r(p)\setminus E)\}}{v_\theta(B_r(p))}>0,\quad \limsup_{r\downarrow 0} \frac{\|\partial E\|_\theta(B_r)}{v_\theta(B_r(p))/r} <\infty.\end{equation}

\begin{definition}[see \cite{Ambrosio-Ghezzi-Magnani}, Definition 3.2] \label{dual normal}(Dual normal and reduced boundary). We write in polar decomposition:
$$
{\bf W}^{0}\chi_E =\nu^*_E |{\bf W}^{0} \chi_E|,$$
where $\nu^*_E= (\nu^*_{E,1}, \dots , \nu^*_{E,2n})
: M\rightarrow \R^{2n}$ is a Borel vector field with unit norm. We call $\nu^*_E$ the \emph{dual normal} to $E$.

We denote by $\partial^*E$ the \emph{reduced boundary} of $E$, i.e. the set of all points $p$ in the support of $|{\bf W}^{0} \chi_E|$ satisfying \eqref{formula1}
and
$$\lim_{r\downarrow 0}\frac{1}{|{\bf W}^{0} \chi_E|(B_r(p))}\int_{B_r(p)} |\nu_E^*(q)-\nu_E^*(p)|^2d |{\bf W}^{0} \chi_E|(q)=0.$$
\end{definition}

We know that if $E$ has locally finite perimeter in $U$, then $|{\bf W}^{0} \chi_E|$-almost every point in $U$ belongs to $\partial^* E$. Moreover,
\begin{theorem}[Riesz Theorem: see \cite{Ambrosio-Ghezzi-Magnani}, Theorem 3.3]\label{Riesz} Let $h$ be a function in $BV_\theta(M)$. Then, there exists a Borel vector field $\nu_h$, satisfying $g(\nu_h,\nu_h)=1$ $|{\bf W}^{0} h|-a.e.$ in $M$ and
$$D_X h = g(X,\nu_u)|{\bf W}^{0}h|, \quad \mbox{for any}\;\;X \in \Gamma(M, \Ker \theta).$$
If $E$ is a set of finite perimeter and $u=\chi_E$, we call \emph{geometric normal} the vector field:
\begin{equation}\label{normal}
\nu_E:=\nu_{\chi_E}.
\end{equation}
In addition $\nu_E = \sum_i \nu^*_{E,i}W_i$.
\end{theorem}

Finally, combining Proposition \ref{Ambrosio-Ghezzi-Magnani} above and Remark 4.3  in \cite{miranda}, we obtain

\begin{proposition}[Coarea formula in $M$]\label{coarea}
If $h\in BV_\theta(M)$ and $f:M\to\R$ is a Borel-measurable function, $f\ge 0$, for any Borel set $U\subset M$
we have:
$$
\int_{U} f \,d|{\bf W}^{0}h | = \int_{-\infty}^{+\infty}  \left( \int_{U} f\,d\|\partial E_t\|_\theta(x)\right) dt,
$$
where $E_t=\{h<t\}$.
\end{proposition}

\subsection{Carnot-Carath\'eodory distance and the Eikonal equation} \label{subsection:Eikonal} The aim of this subsection is to prove the Eikonal equation for the Carnot-Carath\'eodory distance.

First we recall the following regularity result about geodesics (see the survey \cite{monti_seminar}, Theorem 4):

\begin{theorem}[Theorem 4 in \cite{monti_seminar}] \label{theo-geodesics} In contact manifolds any length minimizing curve is smooth.
\end{theorem}

A function $h:(M,d_c)\to\mathbb R$ is $L$-Lipschitz if
$$|h(p)-h(q)|\leq  Ld_c(p,q)$$
for all $p,q\in M$. The infimum of such constants $L$ is denoted by $\Lip(h)$. Lipschitz functions are differentiable a.e. along the vector fields $W_j$, $j=1,\ldots,2n$, as we see from the lemma below.






\begin{lemma}\label{rademacher} If $h:M\to \R$ is $L$-Lipschitz continuous with respect to $d_c$, then
$$
h\in BV_\theta(M),
$$
\begin{equation}\label{rademacher eq: 2}
|\mathbf W^{0}h |(U) \le L v_\theta (U)\quad\mbox{for all open sets $U\subset M$}
\end{equation}
and the
 Lie derivative
\begin{equation}\label{rademacher eq: 3}
 \mc L_X h (x^0) := \lim_{t\to 0} \dfrac1t \big(h(\exp(tX)x^0) - h(x^0) \big)
\end{equation}
 exists for all $X\in \ker\theta$ and for almost every $x^0\in M$.

 In addition $\mc L_X h$ is a distributional derivative, i.e. (with the notation of \cite{Ambrosio-Ghezzi-Magnani} as in \eqref{variation})
 \begin{equation*}\label{rademacher eq:1}
(\mc L_X h )\,dv_\theta = D_X h.
\end{equation*}

 \end{lemma}

 \begin{proof} The first two assertions follows straightforwardly from \cite{miranda}, keeping in mind
 Theorem 3.1 of \cite{Ambrosio-Ghezzi-Magnani}. Let now $\bar x\in M$ be a fixed point. Then, by Darboux theorem there exists a
 neighborhood $U$ of $\bar x$ and a contact diffeomorphism $\Psi:U\to \he n$.
 The map $\Psi$ is bi-Lipschitz continuous with respect to the  Carnot-Carath\'eodory
 distance $d_c$ in $U$ and the canonical Carnot-Carath\'eodory
 distance $d_c^{\mathbb H}$ in $\he n$. In particular, $h\circ \Psi^{-1}$ is
 $d_c^{\mathbb H}$-Lipschitz continuous. By Pansu-Rademacher  theorem (see \cite{P}), for a.e. $x^0\in U$
 there exist
 real numbers $\lambda_1(x^0),\dots, \lambda_{2n}(x^0)$ such that,
 if we set $\Psi(x^0) := p^0$ for $p^0 =(p^0_1,\dots, p^0_{2n+1})$ and $p =(p_1,\dots, p_{2n+1})$,
 \begin{equation*}\label{pansu rademacher}
 h\circ \Psi^{-1}(p) - h\circ \Psi^{-1}(p^0) = \sum_{j=0}^{2n} \lambda_j(x^0) (p_j - p^0_j) + o(d_c^{\mathbb H}(p, p^0))
 \end{equation*}
 as $p\to p^0$ and hence, if $\Psi = (\Psi_1,\dots,\Psi_{2n+1})$,
 \begin{equation*}\label{pansu rademacher 2}
 h (x) - h(x^0) = \sum_{j=0}^{2n} \lambda_j(x^0) (\Psi_j(x) - \Psi_j(x^0) )+ o(d_c(x, x^0)),
 \end{equation*}
 as $x\to x^0$.  Thus, keeping in mind that $d_c(\exp(tX)x^0, x^0) = O(t)$ as $t\to 0$, we have:
 \begin{equation}\label{pansu rademacher 3}\begin{split}
 \lim_{t\to 0} & \; \dfrac1t \big(h(\exp(tX)x^0) - h(x^0) \big) \\& =
 \lim_{t\to 0} \; \dfrac1t \sum_{j=0}^{2n} \lambda_j(x^0) (\Psi_j(\exp(tX)x^0) - \Psi_j(x^0) )
 \\& + \; \lim_{t\to 0} \dfrac1t o(d_c(\exp(tX)x^0, x^0))
 \\& = \sum_{j=0}^{2n} \mu_j^X (x_0),
 \end{split} \end{equation}
 where
 $$
 \mu_j^X (x_0) = \lambda_j(x^0) \dfrac{d}{dt} \Psi_j(\exp(tX)x^0)\quad\mbox{at $t=0$, \; $ j=1,\dots,2n$}.
 $$

 Finally, the last statement follows from \eqref{rademacher eq: 3}
 and \eqref{rademacher eq: 2} by standard arguments.

\end{proof}

\begin{remark}\label{remark rademacher}
We notice that, if $\gamma:[0,1]\to M$ is a continuously differentiable horizontal curve with
$\gamma(0)=x^0$ and $\dot{\gamma}(0)=X$, then, arguing as in \eqref{pansu rademacher 3},
$$
 \lim_{t\to 0} \; \dfrac1t \big(h(\gamma(t)) - h(x^0) \big) = \mc L_Xh(x^0).
$$
\end{remark}

%
%
%

\begin{lemma}\label{geodesic} Let $K\subset M$ be a compact set and let $x\in M$. We denote by $d_{c,K}(x)$ the Carnot-Carath\'eodory
distance of $x$ from $K$. Then
\begin{itemize}
\item[i)] $d_{c,K}(x)$ is $1$-Lipschitz continuous with respect to the $d_c$-distance;
\item[ii)] for a.e. $x^0\in M$ and for all $X\in \ker\theta$, with $g(X,X)  \le 1$
$$
\big|  X d_{c,K} (x^0)  \big|  \le 1,
$$
and there exists $X^0= X(x_0)\in \ker\theta$, with $g(X^0,X^0) =1$ such that
$$ {X^0} d_{c,K} (x^0) = 1. $$
\end{itemize}

\end{lemma}

\begin{proof} The first assertion is trivial. Moreover, it is well known that for any $x\in M$,
there exists $\bar x\in K$ such that $d_{c,K}(x) = d_c(\bar x, x)$. Let now $x^0$ be a point where all horizontal Lie derivatives exist,
and let $\gamma: [0, d_c(\bar x, x^0)] \to M$ be a minimizing geodesic with
 $\gamma(d_c(\bar x, x^0)) = \bar x$ and $\gamma(0)=x^0$.
By Theorem \ref{theo-geodesics}, $\gamma$ is smooth. Without loss of generality, we may
assume that $d_c(\gamma(t),x^0) = t$. Keeping in mind Remark \ref{remark rademacher},
if we take $X^0:=X(x^0) = \dot{\gamma}(0)$, we have
$$
 {X^0} d_{c,K} (x^0) =  \lim_{t\to 0} \; \dfrac1t \big(d_c(\gamma(t) ,x^0) \big) = 1.
$$
This concludes the proof of ii).
\end{proof}

We can finally state the Eikonal equation for the distance $d_c$:

\begin{theorem}[\textbf{The Eikonal equation}] \label{thm:Eikonal}  Let $K\subset M$ be a closed set and let $d_{c,K}$ be the distance from $K$. Then
\begin{equation}\label{Eikonal-equation}|\mathbf W^{0}d_{c,K}|=dv_\theta.
\end{equation}
\end{theorem}

\begin{proof} Let $x^0$ and $X^0=X(x^0)$ be as in Lemma \ref{geodesic}. We can write $X^0=\sum_{j=1}^{2n} \lambda_j W_j^{0}$.
Since $g(X^0,X^0)=1$ we have
$$
\sum_j \lambda_j^2 =1.
$$
Then
$$
\left(\sum_{j=1}^{2n} ({W^{0}_j} d_{c,K})^2 \right)^{1/2} \ge \sum_{j=1}^{2n} \lambda_j ({W_j^{0}} d_{c,K} ) =  {X^0} d_{c,K} =1.
$$
The reverse estimate follows from \eqref{rademacher eq: 2} and Theorem \ref{geodesic}, part i).

Finally, as in \cite{Ambrosio-Ghezzi-Magnani}, page 20, we have that $|\mathbf W^{0}d_{c,K}|=\left(\sum_{j=1}^{2n} ({W^{0}_j} d_{c,K})^2 \right)^{1/2}$, which concludes the proof of the Theorem.
\end{proof}




\subsection{Minkowski content and perimeter}
Let $E$ be an open set in $M$ and let $d_{c,\partial E}(x)$ denote the Carnot-Carath\'eodory distance of the point $x \in M$ from the boundary of $E$. We define the tubular neighborhood of $\partial E$ in $M$:
$$\mathcal U_r(\partial E):=\{p \in M\::\:d_{c,\partial E}(p) < r\}.$$
The \emph{upper} and \emph{lower Minkowski content} of $\partial E$ in $M$ are defined, respectively, as follows:
$$\mathcal M^+(\partial E):=\limsup_{r\downarrow 0} \frac{v_\theta(\mathcal U_r(\partial E))}{2r},$$
$$\mathcal M^-(\partial E):=\liminf_{r\downarrow 0} \frac{v_\theta(\mathcal U_r(\partial E))}{2r}.$$
When $\mathcal M^+(\partial E)=\mathcal M^-(\partial E)$, we call the common value the \emph{Minkowski content} of $E$ and we denote it by $\mathcal M(\partial E)$.
The following theorem is the analogue of Theorem 5.1 in \cite{Monti-SerraCassano}.
\begin{theorem}\label{Minkowski}
Let $E \subset \subset M$ be a bounded open set with $C^\infty$ boundary. Then
$\mathcal M^+(\partial E)=\mathcal M^-(\partial E)$ and we have
$$\mathcal M(\partial E)=\|\partial E\|_\theta.$$
\end{theorem}
\begin{proof}
We follow the proof of Theorem 5.1 in \cite{Monti-SerraCassano}. We prove separately the two following inequalities:
\begin{equation}\label{eq-}
\mathcal M^-(\partial E)\geq \|\partial E\|_{\theta},
\end{equation}
\begin{equation}\label{eq+}
\mathcal M^+(\partial E)\leq \|\partial E\|_{\theta}.
\end{equation}
We start by proving \eqref{eq-}. Let us introduce the signed distance from $\partial E$:
\begin{equation}\label{signed-distance}
\rho_c(x)=\begin{cases}
d_{c,\partial E}(p) & \mbox{if}\:\:p \in E,\\
-d_{c,\partial E}(p) & \mbox{if}\:\:p \in M\setminus E.
\end{cases}
\end{equation}
For $\eps>0$ we define the function:
\begin{equation*}\varphi_\eps(p)=\begin{cases}
\frac{1}{2\eps} \rho_c(p) +\frac{1}{2} &\mbox{if}\:\:|\rho_c(p)|<\eps,\\
1 &\mbox{if}\:\:\rho_c(p)\geq\eps,\\
0 &\mbox{if}\:\:\rho_c(p)\leq -\eps.
\end{cases}
\end{equation*}
Using that Theorem \ref{thm:Eikonal} on the Eikonal equation, we have
$$
|\textbf{W}^{0} \varphi_\eps|=\frac{1}{2\eps}\int_{\{|\rho_c(p)|<\eps\}} |\textbf{W}^{0} \varphi_\eps (p)|\, dv_\theta(p) \leq \frac{1}{2\eps} \,v_\theta(\mathcal U_\varepsilon(\partial E)).$$
By the lower semicontinuity of the total variation and since $\varphi_\eps \rightarrow \chi_E$ in $L^1(M)$, we deduce that
$$\|\partial E\|_\theta \leq\liminf_{\eps \rightarrow 0} |\textbf{W}^{0} \varphi_\eps| \leq \mathcal M^-(\partial E),$$
which concludes the proof of \eqref{eq-}.

It remains to prove \eqref{eq+}. Here we use a Riemannian approximation for Carnot-Carath\'eodory spaces (see e.g. \cite{Franchi_seminar} and \cite{Monti-SerraCassano}).
We consider the Carnot-Carath\'eodory distance $d_\eps$ in $M$ associated with the vector fields ${\mathbf W}^{0}_\eps=\{W_1^{0},\dots, W^{0}_{2n}, \eps T\}$.
Notice that ${\mathbf W}^{0}_\eps$ is an orthonormal basis of $TM$ with respect to the Riemannian metric $g_\eps$ defined as follows: if $X,\:Y\in TM$, we write $X=X'+X''$, $Y=Y'+Y''$, with $X',Y'\in \ker \theta$ and $X'',Y'' \in \mathrm{span}\;\{T\}$, and we set
$$g_\eps(X,Y):=g(X',Y')+\frac{1}{\eps^2}g(X'',Y'').$$
Obviously $d_\eps$ is a Riemannian distance.

Define also $d_{\eps,\partial E}(p)=\min_{q\in\partial E} d_\eps(x,y)$.
We have that
\begin{equation}\label{compare-dist}
d_\eps(p,q)\leq d_{c,\partial E}(p,q)\quad \text{for all }p,q.\end{equation}
In fact, $d_{c,\partial E}(p,q)=\sup_{\eps>0} d_\eps(p,q)$.

Define also $\rho_\eps$ to be the signed $\eps$-distance to $\partial E$ as in
\eqref{signed-distance}. Then $\rho_\eps$  is $C^\infty$ near $\partial E$ and it satisfies the Eikonal equation $|{\mathbf W}^{0}_\eps(\rho_\eps)|=1$.

We consider the usual upper and lower Minkowski content for $\rho_\eps$
$$\mathcal M^+_\eps(\partial E):=\limsup_{r\downarrow 0} \frac{v_{g_\eps}(\{|\rho_\eps|<r\})}{2r},\quad
\mathcal M^-_\eps(\partial E):=\liminf_{r\downarrow 0} \frac{v_{g_\eps}(\{|\rho_\eps|<r\})}{2r}.$$
From \eqref{compare-dist}, $|\rho_\eps|\leq |\rho|$, from which we immediately have
\begin{equation}\label{Minkowski1}
\mathcal M^+(\partial E)\leq \mathcal M^+_\eps(\partial E).
\end{equation}
To achieve the proof of Theorem \ref{Minkowski}, we need the following technical result.

\begin{lemma}\label{10.6:1} If $E\subset M$ is an open set with smooth boundary $\partial E$,
that is a compact $2n$-dimensional submanifold
without boundary, we have
\begin{equation}\label{10.6:2}
|\mathbf{W}_\eps^{0}\chi_E|(M) \to |\mathbf{W}^{0}\chi_E |(M)\quad \mbox{as $\eps\to 0$}.
\end{equation}
\end{lemma}

\begin{proof} Without loss of generality, in \eqref{10.6:2} we can replace $M$ by   an open set $U$
that is contained in the domain of a Darboux map $\Psi: U\to \he{n} \equiv \R^{2n+1}$.
We denote by $
\mu\to\Psi_\#\mu$ the push-forward of a Borel measure $\mu$, i.e.
$$
\Psi_\# \mu (\mc B) = \mu (\Psi^{-1}(\mc B))\quad\mbox{for any $\mc B\subset  \R^{2n+1}$ Borel.}
$$
Moreover, we denote by $\Psi^*g$ the pull-back metric on $\R^{2n+1}$.
By
\cite{Ambrosio-Ghezzi-Magnani}, Proposition 2.2, if $X\in \Gamma (M, TM)$, then
\begin{equation*}\begin{split}
(D_X\chi_E)(\mc B) = \Psi_\# (D_X\chi_E) (\Psi(\mc B)).
\end{split}\end{equation*}
Thus
\begin{equation}\label{april 15 eq:1}\begin{split}
|\mathbf{W}^0_\eps \chi_E|(\mc B) &=
\sup_{g_\eps(X,X)\le 1} |D_X\chi_E|(\mc B)
\\&= \sup_{g_\eps(X,X)\le 1} |\Psi_\#(D_X\chi_E) |(\Psi(\mc B))
\\&
= \sup_{g_\eps(X,X)\le 1}|\Psi_\#(D_{\Psi_*X}\chi_{\Psi(E)}| (\Psi(\mc B))
\\&= \sup_{g_\eps^*(\Psi_*X,\Psi_*X)\le 1}|\Psi_\#(D_{\Psi_*X}\chi_{\Psi(E)}| (\Psi(\mc B))
\\&
= |\Psi_*(\mathbf{W}^0_\eps) \chi_{\Psi(E)}| (\Psi(\mc B)),
\end{split}\end{equation}
where
$$
\Psi_*(\mathbf{W}^0_\eps) = \{\Psi_*W_1^0, \dots, \Psi_*W_{2n}^0\}.
$$
As in \cite{Monti-SerraCassano}, formula (5.5),
$$
|\Psi_*(\mathbf{W}^0_\eps) \chi_{\Psi(E)}| (\Psi(\mc B))
\to
|\Psi_*(\mathbf{W}^0) \chi_{\Psi(E)}| (\Psi(\mc B)).
$$
Thus, repeating backward the arguments of \eqref{april 15 eq:1}, we
conclude the proof of the Lemma.

\end{proof}

Let us go back to the proof of Theorem \ref{Minkowski}.
We will prove soon that
\begin{equation}\label{Minkowski2}
\mathcal M_\eps^+(\partial E)=\mathcal M_\eps^-(\partial E)=|{\mathbf W}^{0}_\eps \chi_E|(M).
\end{equation}
Suppose for the moment that this is true. Then, by \eqref{Minkowski1}, \eqref{Minkowski2}, and \eqref{10.6:2}, we have:
$$
\mathcal M^+(\partial E)\leq \lim_{\eps \rightarrow 0} \mathcal M_\eps^+(\partial E)=\lim_{\eps\rightarrow 0}|{\mathbf W}^{0}_\eps \chi_\eps|(M)=|{\mathbf W}^{0} \chi_\eps|(M),$$
which concludes the proof of the theorem. Therefore, it remains just to show \eqref{Minkowski2}.

Let $E_s=\{p\in M \,:\,\rho_\eps(p)>s \}$. Using the coarea formula \eqref{coarea} and the Riemannian Eikonal equation, we have that
\begin{equation*}\begin{split}
v_\theta(\{|\rho_\eps|<t\})& =\int_{\{|\rho_\eps<t|\}}\,dv_\theta
=\int_{-t}^t \frac{1}{|{\textbf W}^{0}_\eps \rho_\eps|}\,d|{\textbf W}^{0}_\eps \chi_{E_s}|\,ds
\\&=\int_{-t}^t |{\textbf W}^{0}_\eps \chi_{E_s}|(M)\,ds.
\end{split}\end{equation*}
Thus, \eqref{Minkowski2} will follow if we prove that
\begin{equation}\label{continuity}
\mbox{the map}\quad s\rightarrow  |{\textbf W}^{0}_\eps \chi_{E_s}|(M)\quad\mbox{is continuous at $s=0$.}
\end{equation}
This can be done using again the arguments of \eqref{april 15 eq:1} to
reduce ourselves to the ``flat'' case of $\mathbb R^{2n+1}$, where \eqref{continuity}
has been already established  in \cite{Monti-SerraCassano}
(see the proof of Theorem 5.1 therein).

\end{proof}

\section{Compactness and liminf inequality in Heisenberg groups}\label{sec4}


The aim of this Section is to prove a liminf inequality for the ``model case'' where $M\times [0,\sigma)$
is replaced by $\he n\times \mathbb [0,\sigma)$. To this end,
for a subset $A$ of $\mathbb H^{n}\times \R^+$, and $A'=\partial A\cap \{z=0\}$, and for a function  $u:A \rightarrow \R$, we consider the localized functional:
\begin{equation}\label{E eps}
E_\eps(u,A,A'):=\eps\int_{A}\left(|\mathbf W^{\mathbb H}u|^2
+ |\partial_z u|^2\right) d\eta dt dz+\lambda_\eps \int_{A'} V(\trace u)\,d\eta dt.
\end{equation}

The following theorem is the analogue of Proposition 4.7 of \cite{ABS2}. It establishes a compactness result and a liminf inequality for the functional $E_\eps(u_\eps,C_R,B^{\mathbb H}_R)$, where $B^{\mathbb H}_R=B^{\mathbb H}(0,R)$ is the Carnot-Caratheodory ball in $\mathbb H^{n}$ of radius $R$ centered at $0$, $C_R:=B^{\mathbb H}_R\times (0,R)\subset \mathbb H^{n}\times \R^+$ and for simplicity of notation we write $B^{\mathbb H}_R$ in place of $B^{\mathbb H}_R\times \{0\}$.
\begin{theorem}\label{liminf-thm}
 Let $\{u_\eps\} \subset W^{1,2}(C_R)$ be a countable sequence with uniformly bounded energies $E_\eps(u_\eps,C_R,B^{\mathbb H}_R)$. Then the traces $\mathrm{Tr}\,u_\eps$ are pre-compact in $L^1(B^{\mathbb H}_R)$ and every cluster point
$v$ belongs to $BV_{\theta_0}(B^{\mathbb H}_R,\{0,1\})$. Moreover, if $\mathrm{Tr}\,  u_\eps \rightarrow v$ in $L^1(B^{\mathbb H}_R)$, then
\begin{equation}\label{liminf}
\liminf_{\eps\rightarrow 0}E_\eps(u_\eps,C_R,B^{\mathbb H}_R)\geq \mathbf c\left|\int_{B^{\mathbb H}_R}\nu_v
\;d\|\partial \{v=1\}\|_{\theta_0}\right|,
\end{equation}
where $\nu_v$ is the geometrical normal to the set $\{v\equiv 1\}$ and $\mathbf c=k/\pi$ with $k$ given in \eqref{lambda}.
\end{theorem}

The proof of Theorem \ref{liminf-thm} is articulated in several steps and requires a good
amount of preliminary results.

\subsection{Slicing theorems}

We recall a Fubini type Theorem in Carnot groups, which is proven in \cite{Mon}. Here, we state it for the case of the Heisenberg group, but it holds in general Carnot groups. Let $S\subset \mathbb H^{n}$ be a $\mathcal C^1$ smooth hypersurface. By the classical Implicit Function Theorem, we may assume that $S=\partial E$, where $E\subset \mathbb H^{n}$ is an open set with finite $\mathbb H$-perimeter. Suppose that there exists an horizontal left invariant vector field $W^{\mathbb H}$ which is globally transverse to $S$, i.e.
$$\left\langle W^{\mathbb H}(p),\nu(p)\right\rangle\neq 0 \quad \forall\;\;p\in S,$$
where $\nu$ is the Euclidean unit inward normal along $S$.
The Cauchy problem
$$\left\{\begin{array}{rl}\dot{\gamma}(t)&=W^{\mathbb H}(\gamma(t))\\
\gamma(0)&=p\in S,\end{array}\right.
$$
has a unique smooth solution defined on all $\R$, which we denote by $\gamma_p(t)=\exp{(tW^{\mathbb H})}(p)$ for $t\in \R$ and $p\in S$. We call this trajectory a horizontal line. Now we consider the family of horizontal $W^{\mathbb H}$-lines starting from $S$ and we denote by $R_S$ the subset of $\mathbb H^{n}$ reachable from $S$ moving along horizontal $W^{\mathbb H}$-lines, that is
\begin{equation}\label{unique-inter}R_S:=\left\{q\in \mathbb H^{n}:\exists \:p\in S,\: \exists\:t \in \R\;\;\mbox{s.t.}\;\;q=\gamma_p(t)\text{ for some }\gamma_p\right\}.\end{equation}
Assume moreover that $\gamma_p(\R)\cap S=p$ for every $p\in S$. Since $W^{\mathbb H}$ is transverse to $S$, by the uniqueness of the solution of the Cauchy problem and by \eqref{unique-inter}, any subset $D$ of $R_S$ has a natural projection on $S$ along $W^{\mathbb H}$. We define the map $pr_S:D\subset R_S\rightarrow S$ in the following way: for $q\in D$ and $p\in S$, we set $p=pr_S(q)$ if and only if there exists $t\in \R$ such that $q=\gamma_p(t)$. Using this projection, every subset $D$ of $R_S$ can be foliated with one-dimensional leaves that are horizontal $W^{\mathbb H}$-lines.
We define now the partial perimeter along a horizontal direction.
\begin{definition}
Let $U$ be an open set in $\mathbb H^n$. Let $E$ be a measurable subset of $\mathbb H^n$. We say that $E$ has finite $W^{\mathbb H}$-perimeter in $U$ if
$$\|\partial_{W^{\mathbb H}} E\|_\theta(U) :=\sup \left\{\int_U \chi_E \;W^{\mathbb H}\varphi\; d\eta dt\::\:\ \varphi \in C_0^1(U),\;|\varphi|\leq 1 \right\} < \infty.$$
\end{definition}

With this notions, we can now state the Fubini type result, which will be used in the proof of the liminf inequality.
\begin{theorem} [see Corollary 2.3 in \cite{Mon}]\label{fubini-thm}
Let $S\subset \mathbb H^{n}$ be a $\mathbb H$-regular hypersurface and assume $S=\partial E$ globally, where $E\subset \mathbb H^{n}$ is a suitable open $\mathbb H$-Caccioppoli set. Let as before, $\gamma_p$ be the horizontal $W^{\mathbb H}$-line starting from $p\in S$ and assume that $\gamma_p(\R)\cap S=p$ for every $p\in S$. Finally let $D\subset R_S$ be a Lebesgue measurable subset of $\mathbb H^{n}$ that is reachable from $S$ by means of $W^{\mathbb H}$-lines. Then, for every function $\psi\in L^1(D)$, the following statement holds:
\begin{itemize}
\item [(i)] let $\psi_{|D_p}$ denote the restriction of $\psi$ to $D_p:=D\cap \gamma_p(\R)$ and let us define the mapping
$$\psi_p:\gamma_p^{-1}(D_p)\subset \R\rightarrow \R,\quad \psi_p(s)=(\psi\circ\gamma_p)(s).$$
Then $\psi_p$ is $\mathcal L^1$-measurable for $\|\partial E\|_{\theta_0}$-a.e. $p\in S$ or, equivalently, the restriction $\psi_{|D_p}$ is $\mathcal H_c^1$-measurable for $\|\partial E\|_{\theta_0}$-a.e. $p\in S$;
\item[(ii)]
the mapping defined by
$$S\ni p\mapsto \int_{D_p}\psi \,d\mathcal H_c^1=\int_{\gamma_p^{-1}(D_p)}\psi_p(s)\,ds$$
is $\|\partial E\|_{\theta_0}$-measurable on $S$ and the following formula holds
\begin{equation}\label{fubini}
\begin{split}
\displaystyle \int_D \psi \,d v_{\theta_0}&=\int_{pr_S(D)}\left[\int_{D_p} \psi \, d\mathcal H^1_c\right]\,d\|\partial_{W^{\mathbb H}} E\|_{\theta_0}(p)\\
\displaystyle &=\int_{pr_S(D)}\left[\int_{\gamma_p^{-1}(D_p)}\psi_p(s)\,ds\right]\cdot\left|\left\langle W^{\mathbb H},\nu_E\right\rangle_{H\mathbb H_p}\,\right|\; d\|\partial E\|_{\theta_0}(p).
\end{split}
\end{equation}
\end{itemize}
\end{theorem}
Later we will apply this result to the case in which $S$ is a vertical hyperplane. We stress that
the $\mathbb H$-perimeter on any vertical hyperplane coincides with the Lebesgue measure (\cite{CDG}).\\

The following result, which is contained in \cite{Mon}, allows to reduce the study of $BV$ functions on Carnot groups to the study of their one-dimensional restrictions. First we introduce the following notation, concerning the one-dimensional total variation along an horizontal vector field $W^{\mathbb H}$ of a function. Let $W^{\mathbb H}$ be a horizontal vector field, such that $|W^{\mathbb H}|_{\mathcal H\mathbb H^n}=1$ and let $\gamma_p$ be  a horizontal $W^{\mathbb H}$-line starting from $p\in \mathbb H^n$.
We set
\begin{align*}
var_{W^{\mathbb H}}^1[f](\mathcal U):=&\sup\left\{ \int_{\mathcal U}f W^{\mathbb H}\varphi \,d\mathcal H^1_c\: :\:\varphi\in C^1_0(\mathcal B), |\varphi|\leq 1,\right.\\
&\hspace{1em}\left.\;\mbox{where}\;\mathcal B\subset \mathbb H^n,\:\mathcal B\;\mbox{open}\;\mbox{s.t.}\:\gamma_p\cap \mathcal B=\mathcal U\right\}.\end{align*}
We give the statement for the specific case of the Heisenberg group.

\begin{theorem}[Theorem 3.7 in \cite{Mon}]\label{BV-slices}
Let $S\subset \mathbb H^{n}$ be a $\mathbb H$-regular hypersurfaces and assume that $S=\partial E$ globally, where $E\subset \mathbb H^{n}$ is a suitable open $\mathbb H^{n}$-Caccioppoli set. Let $W^{\mathbb H}\in \mathcal H\mathbb H^{n}$, $|W^{\mathbb H}|_{\mathcal H\mathbb H^{n}}=1$, be a unit horizontal left invariant vector field which is transverse to $S$, and denote by $t\to \gamma_p(t):= p\cdot\exp(tW^{\mathbb H})$ the horizontal $W^{\mathbb H}$-line starting from $p\in S$. Let $D\subset R_S$ be a Lebesgue measurable subset of $\mathbb H^{n}$ that is reachable from $S$ by means of $W^{\mathbb H}$-lines.

Then
\begin{equation}\label{BV}
|W^{\mathbb H}f|(D)=\int_{pr_S(D)}var_{W^{\mathbb H}}^1[f_p](D_p)d\|\partial_{W^{\mathbb H}} E\|_{\theta_0}(p),
\end{equation}
where $f_p:=f\circ \gamma_p$ and $D_p:=\gamma_p\cap D$.
\end{theorem}

Our next step will be to prove a compactness result in $L^1$ for a family of functions satisfying some kind of equicontinuity along $1$-dimensional horizontal lines (see Theorem \ref{compact}). To this end, we must factorize an arbitrary displacement through a finite number of \emph{horizontal} displacements of controlled length. This is the content of the following Theorem \ref{morbidelli}.

\begin{theorem}[\cite{morbidelli}, \S 3]\label{morbidelli}
There exist $m\in \mathbb N$ and three multi-indexes $I$, $J$
and $\mathbf{\omega}$
of length $m$
\begin{equation*}
\begin{split}
I &= (i_1,\dots,i_m),\quad i_n\in \{1,\dots,2n\}\\
J &= (j_1,\dots,j_m),\quad j_n \in \{1,\dots,2n+1\}
\\
\mathbf{\omega}&=(\omega_1,\dots, \omega_M) \quad \omega_n \in \{-1,1\}
\end{split}\end{equation*}
and two geometric constants $0<b<a<1$ such that, if we set
$$
\mathcal E_{I,J,\mathbf{\omega}}: \rn{2n+1}   \to \he n
$$
$$
\mathcal E_{I,J,\mathbf{\omega}}(t_1,\dots,t_{2n+1}) := \exp( \omega_1t_{j_1}W^{\mathbb H}_{i_1})\cdots  \exp( \omega_m t_{j_m}W^{\mathbb H}_{i_m}),
$$
then for all $R>0$
$$
B_c(0,bR)\subset \mathcal E_{I,J,\mathbf{\omega}}(Q(0,aR)) \subset B_c(0,R),
$$
where
$$
Q(0,r) = \{(t_1,\dots,t_{2n+1})\in \rn{2n+1},\; \max_\ell \{|t_\ell|\}< r\}.
$$
In particular, if $h\in \he n$, then there exist $t_\ell = t_\ell(h)$, $\ell=1,\dots,2n+1$, $\max_\ell \{|t_\ell|\}< a\,d_c(0,h)/b$ such that
$$
\mathcal E_{I,J,\mathbf{\omega}}(t_1,\dots,t_{2n+1})  = h.
$$
\end{theorem}

The main idea of Theorem \ref{morbidelli} is that each point in $\mathbb H^n$ can be reached by integral curves of horizontal vector fields, and when a commutator of two vector fields is needed, it can be approximated by a finite length "square path" along the two fields, taken successively with opposite sign. This is an important difference between this result  and the classical result due to Nagel, Stein and Wainger \cite{NSW}, Theorem 7, where instead the authors work directly with integral curves of commutators.

The following result is the analogue of Theorem 6.6 in \cite{ABS2}, and will be used to deduce compactness of the $\trace u_\eps$ from the compactness of their restrictions to the horizontal slices. We first fix some notations. Let $e_1,...,e_{2n}$ be the first $2n$ unit vectors of the canonical basis of $\mathbb{H}^{n}$. Let $D\subset \mathbb H^{n}$ and let $\Pi_i$ be the vertical hyperplane orthogonal to $e_i$. Obviously we have that $W_i^{\mathbb H}$ is globally transverse to $\Pi_i$, and therefore we can consider the projection $D_i$ of $D$ on $\Pi_i$ along $W_i^{\mathbb H}$. We denote by $\gamma_i^p(s)$ the horizontal $W_i^{\mathbb H}$-line starting from a point $p\in \Pi_i$. For a function $v$ defined on $D$, we consider the function
$v_i^p(s):=v(\gamma_i^p(s))$
 defined on the set $D_i^p:=\{s\in \R|\gamma_i^p(s)\in D\}$. Accordingly, for every family $\mathcal F$ of functions on $D$, we define the family $\mathcal F_i^p:=\{v_i^p|v\in \mathcal F\}$.

 We say that a family $\mathcal F'$ is $\delta$-dense in $\mathcal F$ if $\mathcal F$ lies in a $\delta$-neighborhood of $\mathcal F'$ with respect to the $L^1$ topology. We have the following theorem:
\begin{theorem}\label{compact}
Let $\mathcal F$ be a family of functions $v:D\rightarrow [-L,L]$ and assume that for every $\delta>0$ there exists a family  $\mathcal F_{\delta}$ $\delta$-dense in $\mathcal F$ such that $(\mathcal F_\delta)_i^p$ is pre-compact in $L^1(D_i^p)$ for $|\Pi_i|_{\mathcal H}$- a.e. $p\in D_i$ for every $i=1,...,2n$. Then $\mathcal F$ is pre-compact in $L^1(D)$.
\end{theorem}

\begin{proof} We can assume $L=1$ and $|D_i^p|\le 1$ for every $p\in \Pi_i$. Every function defined
on $D$ is extended to be zero outside  $D$, and accordingly every function defined $D_i^p$ is extended to be zero outside $D_i^p$.
Arguing as in \cite{ABS2}, Theorem 6.6, we have but
to show that for any $\delta>0$
\begin{equation}\label{july4:1}
\int_{\he n} |v( q\cdot h) - v(q)|\, dq \to 0
\end{equation}
as $d^{\mathbb H}_c(h,0)\to 0$, uniformly for $v\in \mathcal F_\delta$.

If $i=1,\dots,2n$ is fixed, $p\in D_i$, $r>0$, we set
$$
\omega_\delta^p(r) = \sup\Big\{ \int_\R |v_i^p(s+\sigma)-v_i^p(s)|\, ds\; : \;
v\in \mathcal F_\delta, \; |\sigma|\le r \Big\}.
$$
By our assumptions, $\omega_\delta^p(r) \le 2$ for all $r>0$ and,
as in \cite{ABS2}, by Fr\'echet-Kolmogorov compactness theorem,
$\omega_\delta^p(r)\searrow 0$ as $r \searrow 0$.

By Theorem \ref{morbidelli} we can write
$$
h=\mathcal E_{I,J,\mathbf{\omega}}(t_1,\dots,t_{2n+1}) ,
$$
with $t_\ell = t_\ell(h)$, $\ell=1,\dots,2n+1$, $\max_\ell \{|t_\ell|\}< ad_c(0,h)/b$.
For sake of brevity we write $t_h= (t_1,\dots, t_{2n+1})$.
With the notations
of Theorem \ref{morbidelli}, for $1\le k\le m$ we set
$$
I_k = (i_1,\dots,i_k)\quad\mbox{,}\quad J_k = (j_1,\dots,j_k)
\quad\mbox{and}\quad \mathbf{\omega}_k = (\omega_1,\dots,\omega_k).
$$
If we set $\mathcal E(I_{0},J_0,\mathbf{\omega}_{0})=e$, we have
\begin{equation*}\begin{split}
v(& x\cdot h) - v(x) = \sum_{k = 1}^m  \Big(v(x\cdot \mathcal E_{I_k,J_k,\mathbf{\omega}_k}(t_h) )-
v(x\cdot \mathcal E_{I_{k-1},J_{k-1},\mathbf{\omega}_{k-1}}(t_h)) \Big)
\\&=
\sum_{k = 1}^m  \Big(v(x\cdot \mathcal E_{I_{k-1},J_{k-1},\mathbf{\omega}_{k-1}}(t_h)\cdot
\exp( \omega_k t_{j_k}W^{\mathbb H}_{i_k})  )-
v(x\cdot \mathcal E_{I_{k-1},J_{k-1},\mathbf{\omega}_{k-1}}(t_h)) \Big).
\end{split}\end{equation*}
Thus, keeping in mind that Lebesgue measure in $\he n$ (that is unimodular) is the group
Haar measure and therefore is right invariant, we have
\begin{equation*}\begin{split}
\int_{\he n}& |v( q\cdot h) - v(q)|\, dq
\\&
\le
\sum_{k=1}^m
\int_{\he n} |v( q\cdot \exp( \omega_k t_{j_k}W^{\mathbb H}_{i_k})) - v(q)|\, dq.
\end{split}\end{equation*}
Take now $i=i_k$ for a generic $k=1,\dots, m$, and set $t:= t_{j_k} $ and, for example, $\omega_k=1$.
By \eqref{fubini},  we have
\begin{equation*}\begin{split}
\int_{\he n} &  |v( q\cdot \exp(tW^{\mathbb H}_{i})) - v(q)|\, dq =
\int_{D_i}\left( \int_\R  |v_i^p(s+ t) - v_i^p(s)|\, ds\right)\, dp
\\&
\le
\int_{D_i} \omega_\delta^p (t)\, dp
\le
\int_{D_i} \omega_\delta^p (ad_c(h,0)/b)\, dp,
\end{split}\end{equation*}
and \eqref{july4:1} follows as in \cite{ABS2}.

\end{proof}

\subsection{Fractional energy in $\mathbb R$}

In this Subsection we recall a liminf inequality for a one-dimensional fractional energy. We follow \cite{ABS1}.
Let $A\subset \R$ be an interval, $v\in L^1(A)$, we define
\begin{equation}\label{G}
G_\eps(v,A):=\frac{\eps}{2\pi}\int_{A^2}\left|\frac{v(s)-v(s')}{s-s'}\right|^2 \,ds\,ds' + \lambda_\eps \int_A V(v(s))\,ds.
\end{equation}
We recall two results that we will use in the proof of the liminf inequality, and that are contained in \cite{G} and \cite{ABS2}. The first one is a trace inequality in rectangles  with optimal constant.
\begin{theorem}[\cite{G}, Theorem 19]\label{trace-thm}
Let $u\in W^{1,2}((0,1)\times (0,1))$. Then, the trace of $u$ on $(0,1)\times \{0\}$, call it $v$, is a well defined function $v\in H^{1/2}(0,1)$, and we have
\begin{equation}\label{trace}
\iint_{(0,1)^2}\left|\frac{v(s)-v(s')}{s-s'}\right|^2 \,ds\, ds'\leq 2\pi \int_0^1\int_{0}^1 |\nabla u|^2 \,ds\, dz.
\end{equation}
\end{theorem}
The following theorem is a liminf inequality for the energy functional $G_\eps$.

\begin{theorem}[Lemma 1 in \cite{ABS1} and Theorem 4.4 in \cite{ABS2}]\label{liminf-trace-thm}
We have:
\begin{itemize}
\item [(i)] Every countable sequence $\{v_\eps\}\subset L^1(A)$ with uniformly bounded energies $G_\eps(v_\eps,A)$ is pre-compact in $L^1(A)$ and every cluster point belongs to $BV(A,\{0,1\})$;
\item [(ii)] For every $v\in BV(A,\{0,1\})$ and every sequence $\{v_\eps\}$ such that $v_\eps\rightarrow v$ in $L^1(A)$,
\begin{equation*}\label{liminf-trace}
\liminf_{\eps\rightarrow 0}G_\eps(v_\eps,A)\geq \mathbf{c}\#  (S_v),
\end{equation*}
where $\#(S_v)$ denotes the number of points of discontinuity of $v$ and $\mathbf{c}=\kappa/\pi$ with $k$ given in \eqref{lambda}.
\end{itemize}
\end{theorem}

\subsection{Proof of Theorem \ref{liminf-thm}}

With these preliminaries in hand, we can give now the proof of our Theorem 4.1.
By a standard truncation argument, we can assume that $0\leq u_\eps\leq 1$ for every $\eps>0$.
We follow the proof of Proposition 4.7 in \cite{ABS2}, which is based on a slicing argument. Let $\boldsymbol e$ be an horizontal vector at the origin with $|\boldsymbol e|=1$, and let $W^{\mathbb H}$ be a left invariant horizontal vector field such that $W^{\mathbb H}(0)=\boldsymbol e$. We denote by $\Pi$ the $(2n)$-dimensional vertical hyperplane orthogonal to $W^{\mathbb H}(0)=\boldsymbol e$.  We apply Theorem \ref{fubini-thm} above with $S=\Pi\cap B^{\mathbb H}_R$, $D=B^{\mathbb H}_R$, $D_p=D\cap \gamma^p$, where as before $\gamma^p$ is the integral curve of $W^{\mathbb H}$ starting from $p\in S$. Observe that if $u\in H^1(C_R)$, where as before $C_R=B^{\mathbb H}_R\times (0,R)$, then for a.e. $p\in S=\Pi\cap B^{\mathbb H}_R$ its restriction to $D_p$, denoted by $u_p$, belongs to $H^1(D_p)$ (see Proposition 6.8 in \cite{ABS2}).
Moreover, using that $|\boldsymbol e|=1$ and $W^{\mathbb H}$ is left invariant, a simple computation show that
$$ \sum_{i=1}^{2n} |W^{\mathbb H}_i u_\eps|^2\geq |W^{\mathbb H} u_\eps|^2.$$
Indeed, if we write $\boldsymbol e=\sum_{i=1}^{2n} c_j W^{\mathbb H}_i(0)$ with $\sum_{i=1}^{2n}c_i^2=1$, by the left invariance of $W^{\mathbb H}$ we have
$$|W^{\mathbb H} u_\eps|^2=|\sum_{i=1}^{2n}c_i W^{\mathbb H}_i u_\eps|^2\leq \left(\sum_{i=1}^{2n}c_i^2\right)\left( \sum_{i=1}^{2n} |W^{\mathbb H}_i u_\eps|^2\right)\leq \sum_{i=1}^{2n} |W^{\mathbb H}_i u_\eps|^2.$$

Hence we have:
\begin{align*}
E_\eps(u_\eps,C_R,B^{\mathbb H}_R)
&=\eps\int_{C_R} \left(\sum_{i=1}^{2n} |W^{\mathbb H}_i u_\eps(\eta,t,z)|^2 + (\partial_z u_\eps(\eta,t,z))^2\right) d\eta dt dz\\
&\hspace{1em}+\lambda_\eps\int_{B^{\mathbb H}_R} V(\trace u_\eps(\eta,t,0))d\eta dt, \\
&\geq \eps\int_0^R\int_{B^{\mathbb H}_R} \left( |W^{\mathbb H} u_\eps(\eta,t,z)|^2 + (\partial_z u_\eps(\eta,t,z))^2\right) d\eta dt dz\\
&\hspace{1em}+\lambda_\eps\int_{B^{\mathbb H}_R} V(\trace u_\eps(\eta,t,0))d\eta dt. \\
\end{align*}
 Set  $D^p=(\gamma^p)^{-1}(\gamma^p(\R)\cap B^{\mathbb H}_R)=\{s\in \R\,|\,\gamma^p(s)\in B^{\mathbb H}_R\}$, and $d\mathcal L_{\Pi}$  the Lebesgue measure on $\Pi$. Using \eqref{fubini}, we obtain
\begin{align*}
&E_\eps(u_\eps,C_R,B^{\mathbb H}_R)\\
&\hspace{1em}\geq \eps \int_{\Pi\cap B^{\mathbb H}_R}d\mathcal L_{\Pi}(p)\left(\int_0^R dz\int_{D^p}\left (|W^{\mathbb H} u_\eps(\gamma_p(s),z)|^2 +|\partial_z u_\eps(\gamma_p(s),z)|^2\right)ds\right.\\
&\hspace{2em}+\left.\lambda_\eps \int_{D^p}V(\trace u_\eps(\gamma_p(s),0))ds\right).
\end{align*}
%
Since $\gamma^p$ is the integral curve of $W^{\mathbb H}$, setting
$$\widetilde u_\eps^p(s,z)=u_\eps(\gamma^p(s),z),$$
we deduce that
$$W^{\mathbb H} u_\eps(\gamma^p(s),z)=\partial_s \widetilde u_\eps^p(s,z)\quad \mbox{and}\quad \partial_z u_\eps(\gamma_p(s),z)=\partial_z \widetilde u_\eps^p(s,z).$$
Therefore, we get
\begin{align*}
&E_\eps(u_\eps,C_R,B^{\mathbb H}_R)\geq \\
&\hspace{1em}\eps \int_{\Pi\cap B^{\mathbb H}_R}d\mathcal L_{\Pi}(p)\left(\int_0^R dz\int_{D^p}\left( |\partial_s \widetilde u^p_\eps(s,z)|^2+|\partial_z \widetilde u^p_\eps(s,z)|^2\right)ds \right.\\
&\hspace{2em}\left.+ \lambda_\eps \int_{D^p}V(\trace \widetilde u^p_\eps(s,0))ds\right).
\end{align*}
We apply now the trace inequality \eqref{trace} to get
\begin{equation}\label{F>G}
\begin{split}
E_\eps(u_\eps,C_R,B^{\mathbb H}_R)&\geq
 \int_{\Pi\cap B^{\mathbb H}_R}d\mathcal L_{\Pi}(p)\left[\frac{\eps}{2\pi}\int_{(D^p)^2}\left|\frac{\trace \widetilde u^p_\eps(s',0)-\trace \widetilde u^p_\eps(s,0)}{s'-s}\right|^2 dsds'\right.\\
 &+ \left.\lambda_\eps \int_{D^p}V(\trace \widetilde u^p_\eps(s,0))\right]ds\\
&=\int_{\Pi\cap B^{\mathbb H}_R}d\mathcal L_{\Pi}(p) \:G_\eps(\trace\widetilde u^p_\eps,D^p),
\end{split}
\end{equation}
where $G_\eps$ is defined as in \eqref{G}.  The proof of Theorem
\ref{liminf-thm} follows from the following two steps:\\

\textbf{Step 1. Compactness:}
We first show that the sequence $\trace u_\eps$ is pre-compact in $L^1(B^{\mathbb H}_R)$. In order to  prove this, it is enough to show that the family $\mathcal F:=\{\trace u_\eps\}$ satisfies the assumptions of Theorem \ref{compact}. We  choose a constant $C$ such that
\begin{equation}\label{bound}
E_\eps(u_\eps,C_R,B^{\mathbb H}_R)\leq C.
\end{equation}
Fix now $\delta>0$ and consider the sequence $v_\eps:B^{\mathbb H}_R\rightarrow [0,1]$ defined as follows:
$v_\eps(\gamma^p(s)):= v_\eps^p(s)$,
where
\begin{equation}\label{v_eps}
v^p_\eps:=\begin{cases} \trace\widetilde u^p_\eps\quad \mbox{for all}\;\;p\in \Pi\cap B^{\mathbb H}_R\;\;\mbox{such that}\;\;G_\eps(\trace\widetilde u_\eps^p,E_p)\leq |\Pi\cap B^{\mathbb H}_R| C/\delta,\\
1\quad \mbox{otherwise.}
\end{cases}
\end{equation}
Observe that $v_\eps$ is well-defined by the uniqueness of integral curves of horizontal vector fields starting from a given point.
Using \eqref{F>G}, \eqref{bound}, and \eqref{v_eps} we deduce that $v^p_\eps=\trace\widetilde u^p_\eps$ for all $p\in \Pi\cap B^{\mathbb H}_R$ apart from a subset of measure smaller that $\delta/|\Pi\cap B^{\mathbb H}_R|$. Therefore $v_\eps=\trace\widetilde u_\eps$ in $B^{\mathbb H}_R$ minus a set of measure smaller than $\delta$ and, since $0\leq \trace u_\eps\leq 1$, we deduce that $\|v_\eps-\trace u_\eps\|_{L^1(B^{\mathbb H}_R)}\leq \delta$. This implies that the family $\mathcal F_\delta$ is $\delta$-dense in $\mathcal F$. By \eqref{v_eps} we have that $G_\eps(v_\eps^p,D^p)\leq |\Pi\cap B^{\mathbb H}_R|C/\delta$ for every $p\in  \Pi\cap B^{\mathbb H}_R$ and every $\eps$, and hence we can apply statement (i) of Theorem \ref{liminf-trace-thm} to deduce that the sequence $(v^p_\eps)$ is pre-compact in $L^1(D^p)$. Thus the family $\mathcal F$ satisfies the assumption of Theorem \ref{compact} for any horizontal tangent vector $\boldsymbol e$ at the origin, and thus in particular for $e_1,\dots,e_{2n}$, and we conclude that the sequence $(\trace u_\eps)$ is pre-compact in $B^{\mathbb H}_R$.\\

\textbf{Step 2. Liminf inequality}:
It remains to prove that if $\trace u_\eps\rightarrow v$ in $L^1(B^{\mathbb H}_R)$, then $v\in BV_{\theta_0}(B^{\mathbb H}_R,\{0,1\})$ and inequality \eqref{liminf} holds. Using \eqref{F>G} and passing to the limit as $\eps \rightarrow 0$, by Fatou's Lemma we deduce that
$$\liminf_{\eps\rightarrow 0}E_\eps(u_\eps,C_R,B^{\mathbb H}_R)\geq \int_{\Pi\cap B^{\mathbb H}_R}\liminf_{\eps \rightarrow 0} G_\eps(\trace\widetilde u^p_\eps,D^p)\,d\mathcal L_{\Pi}(p),$$
and then $\liminf_{\eps \rightarrow 0} G_\eps(\trace\widetilde u^p_\eps,D^p)$ is finite for a.e. $p\in \Pi\cap B^{\mathbb H}_R$. Since $\trace u_\eps \rightarrow v$ in $L^1(B^{\mathbb H}_R)$,  possibly passing to a subsequence, we have that $\trace \widetilde u^p_\eps \rightarrow v^p$ in $L^1(D^p)$ for a.e. $p\in\Pi\cap B^{\mathbb H}_R$ (see Remark 6.7 in \cite{ABS2}). Then, using Theorem \ref{liminf-trace-thm} we deduce that $v^p\in BV(D^p,\{0,1\})$ and
\begin{equation}\label{liminf-slice}
\liminf_{\eps\rightarrow 0} E_\eps(u_\eps,C_R,B^{\mathbb H}_R)\geq \int_{\Pi\cap B^{\mathbb H}_R} \mathbf c \#(S_{v^p})\,d\mathcal L_{\Pi}(p).
\end{equation}
Finally, applying Theorem \ref{BV-slices} we deduce that $v \in BV_{\theta_0}(B^{\mathbb H}_R,\{0,1\})$, that $S_{v^p}$ agrees with $S_v\cap D^p$ for a.e. $p \in \Pi\cap B^{\mathbb H}_R$, and that
\begin{equation*}\begin{split} \liminf_{\eps\rightarrow 0} E_\eps(u_\eps,C_R,B^{\mathbb H}_R)&\geq \mathbf c\int_{B^{\mathbb H}_R\cap S_v}\left\langle \nu_v,\boldsymbol e\right\rangle\, d\|\partial \{v=1\}\|_{\theta_0}
\\&
= \big\langle \int_{B^{\mathbb H}_R\cap S_v} \nu_v\, d\|\partial \{v=1\}\|_{\theta_0},
\boldsymbol e\big\rangle.
\end{split}\end{equation*}
We conclude the proof of Theorem \ref{liminf-thm} by choosing a suitable vector $\boldsymbol e$.

\section{Proof of the liminf inequality near the boundary $M$}\label{sec5}

%
%
%
%

In this Section we prove Theorem \ref{theorem bis}.
To this aim,
we need to pass from the ``flat case'' $\he n \times[0,\sigma)$ to $M\times [0,\sigma)$.
This will be the
content of  the following Sections \ref{flattening}, \ref{conclusion} and \ref{densities}.

Given $A\subset M\times [0,\sigma)$, and $A'\subset M$, we define the localized energy
\begin{equation*}\begin{split}
\tilde F_{\eps} (u,A,A') :=  \eps  & \int_{A} \Big(  \sum_{j=1}^{2n}
(  W_j^0   u)^2+( \partial_z  u)^2 \Big)\, dv_\theta\wedge dz
\\&
+ \lambda_\eps \int_{A'} V(\trace  u) \, dv_\theta
\end{split}\end{equation*}
(compare with \eqref{model bis} and keep in mind Remark \ref{conventions}).

\subsection{Flattening}\label{flattening}

Following \cite{ABS2} we give the definition of \textit{contact isometry defect}.

\begin{definition} Let $M_1$ and $M_2$ be two contact $(2n+1)$-manifolds endowed with the contact forms $\theta_1$ and $\theta_2$,
and let $g_{\theta_1} $ and $g_{\theta_2} $ be fixed Riemannian metrics on $\ker \theta_1$ and $\ker \theta_2$, respectively. If $p_i\in M_i$,
$i=1,2$, we denote
by $HO (T_{p_1}M_1,T_{p_2} M_2)$ the space of linear  maps from $T_{p_1}M_1$ to $T_{p_2}M_2$ that are isometries
on $\ker \theta_1(p_1)$ and are induced by contact maps.
\end{definition}

\begin{definition}
Let $M_1$ and $M_2$ be two contact $(2n+1)$-manifolds endowed with the contact forms $\theta_1$ and $\theta_2$,
respectively, and let $U_1\subset M_1$ and $U_2\subset M_2$
be open sets. Let $\Psi:U_1\to U_2$ be a diffeomorphism. We call \textit{contact isometry defect} $\delta(\Psi)$ the smallest $\delta>0$ such that
$$\mbox{dist}(d\Psi(p),HO(T_pM_1, T_{\Psi(p)}M_2))\leq \delta\quad \mbox{for a.e.}\;\;p\in U_1.$$
\end{definition}

%

\begin{theorem}\label{isometry}
Let $(M,\theta)$ be the $(2n+1)$-dimensional contact manifold endowed with the Riemannian
metric $g$, as in Propositions \ref{contact 1} and \ref{contact 3}. Let $\bar p\in M$ be
any fixed point. Let $(W_1^0,\dots,W^0_{2n})$ be the orthonormal symplectic basis of $\ker\theta(\bar p)$
(see Remark \ref{conventions}),
and let $(W_1^{\mathbb H},\dots,W^{\mathbb H}_{2n})$ be the orthonormal symplectic basis of $\ker\theta_0$ at the origin in $\he {n}$.
Then there exist an open neighborhood $\mc U$ of $\bar p$ and a local diffeomorphism
$$
\Psi: \mc U\to \he {n},
$$
such that
\begin{itemize}
\item [\emph{i)}] $\Psi$ is a contact map (i.e. $\Psi^*\theta_0= \theta$);
\item [\emph{ii)}] $\Psi (\bar p)=0$ and $\mc U_0:=\Psi(\mc U)$ is open;
\item [\emph{iii)}]  $D\Psi(\bar p)W_j^0= W^{\mathbb H}_j$,  $j=1,\dots,2n$. In particular,
$D\Psi(\bar p): \ker\theta(\bar p) \to \ker\theta_0$ is an isometry
when the horizontal fiber of $\ker\theta_0$  at the origin is endowed with the canonical Riemannian metric $\scal{\cdot}{\cdot}_{\mathbb H}$.
\end{itemize}

\end{theorem}

\begin{proof} Darboux Theorem implies that there exists a neighborhood $\mathcal U$ of $\bar p$ and a
diffeomorphism  $\Psi_0:  \mc U\to \he n$  such that $\Psi_0^*\theta_0= \theta$,
and thus $\Psi_0^*(d\theta_0)= d\theta = i^*\omega$. Hence
\begin{equation*}\begin{split}(\hat W_1,\cdots,\hat W_{2n})
:=
((\Psi_{0})_* W^0_1,\dots,(\Psi_{0})_*W^0_{2n} )
\end{split}\end{equation*}
is a symplectic basis of $\ker\theta_0$. Then, in particular,
$$(\hat W_1(0),\dots, \hat W_{2n}(0))$$ can be identified with
a symplectic basis of $\rn{2n}$, and therefore there exists  $A\in Sp(n)$ such that
$$
A \hat W_j(0) = e_j=W_j^{\mathbb H}(0)  \qquad j=1,\dots,2n.
$$
Put now
\begin{displaymath}
\Psi:=
\left( \begin{array}{cc}
A & 0_{2n\times 1}  \\
0_{1\times 2n}  & 1
\end{array} \right)\Psi_0.
\end{displaymath}
Obviously, $\Psi $ satisfies \emph{i)} by Lemma \ref{pansu} below and \emph{ii)}. Moreover
\begin{displaymath}
D\Psi(\bar p) (W^0_i(\bar p)) = \left( \begin{array}{cc}
A & 0_{2n\times 1}  \\
0_{1\times 2n}  & 1
\end{array} \right)
 \hat W_i(0) = W_i^{\mathbb H}(0),
\end{displaymath}
and the assertion follows.
\end{proof}

\begin{lemma}[see \cite{FT5, pansu_thesis}]\label{pansu}If $a>0$ and $\dfrac{1}{\sqrt a}A\in Sp(n)$, then the (Euclidean) linear map $T:\he n\to\he n$
\begin{displaymath}
T:=
\left( \begin{array}{cc}
A & 0_{2n\times 1}  \\
0_{1\times 2n}  & a
\end{array} \right)
\end{displaymath}
belongs to $GL(\mathbb R^{2n+1}, \mathbb R^{2n+1})$ and is a contact map.
\end{lemma}

Then, for each $p\in M$ and any $r>0$ (close to $0$), there exists a neighborhood $U(p,r)\subset M$ and a diffeomorphism $\Psi_p$ such that the image $\Psi_p(U(p,r))$ is the $d^{\mathbb H}_c$-ball of radius $r$ centered at the origin in the Heisenberg group, denoted by $B^{\mathbb H}_r$, and
$$\|D(\Psi_p) -I_{2n+1}\|\leq \delta(r),$$
for some $\delta(r)\to 0$ when $r\to 0$. Here $I_n$ denotes the identity map in $n$-dimensions.
We also point out that, by Lemma \ref{one year later} (which will be proven later on in Section \ref{densities}), we have that in $M$:
\begin{equation}\label{Aug 25 eq:1}
U(p,r)\subset B(p, r(1+o(1)))\qquad\mbox{as $r\to 0$.}
\end{equation}

Adding the normal variable $z>0$, we may cover  $M\times [0,r]$ by a finite number of neighborhoods $\{\tilde U(p_j,r)\}_{j=1}^K$, $p_j\in M$ such that for each $j$, there exists a diffeomorphism
$$\tilde \Psi_{p_j}
: \{\tilde U(p_j,r)\}_{j=1}^K\to \he n\times [0,r]
$$
 satisfying
\begin{equation*}
\begin{split}
&\tilde\Psi_{p_j}(\tilde U(p_j,r))=C_r^{\mathbb H}\subset \mathbb H^{n}\times \mathbb R_+,\\
&\tilde\Psi_{p_j}(U(p_j,r))=B_r^{\mathbb H}\subset \mathbb H^{n},
\\&
\tilde \Psi_{p_j}((p_j,0))=(0,0),
\end{split}
\end{equation*}
and
\begin{equation*}\label{isom-defect}\|D\tilde \Psi_{p_j} -I_{2(n+1)}\|\leq \tilde\delta(r),\end{equation*}
for some $\tilde\delta(r)\to 0$ when $r\to 0$.\\

Since
\begin{equation}\label{change-variable-grad}|D(u\circ \tilde\Psi^{-1}_{p_j})|\leq (1+\delta)|Du\circ \tilde\Psi^{-1}_{p_j}|,\end{equation}
this in particular implies that the localized energy
$\tilde F_\eps(u_\eps,\tilde U(p_j,r),U(p_j,r))$
 can be replaced by the energy
 $E_\eps (w_\eps, C_r^{\mathbb H}, B_r^{\mathbb H})$, where $w_\eps=u_\eps\circ \tilde\Psi_{p_j}$.
 More precisely, arguing exactly as in \cite{ABS2}, Proposition 4.9, we have that
 \begin{equation}\label{est-flat}
 \tilde F_\eps(u_\eps,\tilde U(p_j,r),U(p_j,r))\geq (1-\delta^5) E_\eps (w_\eps, C_r^{\mathbb H}, B_r^{\mathbb H}).
 \end{equation}

\subsection{Conclusion of the proof of Theorem \ref{theorem bis}}\label{conclusion}

Let $\{u_\eps\}\subset W^{1,2}(\Omega)$ be a countable sequence such
that $\tilde F_{\eps,r}(u_\eps)$ is bounded independently of $\eps$.
We have to prove that the sequence of the traces $\{\trace u_\eps\}$ is pre-compact in
$L^1(M)$. But since we have just shown that we can cover $M\times [0,r]$ with finitely many neighborhoods $\{\tilde U(p_j,r)\}_{j=1}^K$, it is enough to show that $\{\trace u_\eps\}$ is
is pre-compact in $L^1(U(p_j,r))$ for every $j=1,\ldots,K$.

For every fixed $j$, let $w_\eps=u_\eps\circ\tilde\Psi^{-1}_{p_j}$. In particular, \eqref{change-variable-grad} implies that $E_\eps(w_\eps,C_r^{\mathbb H}, B_r^{\mathbb H})$ is  uniformly bounded in $\eps$. Hence the pre-compactness
follows from Theorem \ref{compact}. This proves statement \emph{i*)} of Theorem \ref{theorem bis}.\\

Next, we would like to prove statement \emph{ii*)} in Theorem \ref{theorem bis}.
Then things become more delicate.

Let us start by recalling some classical definitions.
For $m>\,0$, we denote
\[
\ga_m:=\frac{\Gamma(\frac{1}{2})^m}{\Gamma(\frac{m}{2}+1)},
\]
being $\Gamma$ the Euler function and
\begin{equation}\label{beta1}
\gb_m:=\,2^{-m} \ga_m.
\end{equation}
According to Federer's notation \cite{federer}, we define a \emph{centered}  density of an outer measure $\mu$ on $X$:

\begin{definition}\label{densitydef} Let $(X,d)$ be a separable metric space, and let $\mu$ be an outer measure on $X$.
If $m>\,0$, the \emph{upper and lower centered $m$-densities} of $\mu$ at $p\in X$ are
\begin{equation*}
\Theta^{*\,m}(\mu,p):=\limsup_{r\to 0}\frac{\mathcal \mu(\overline B(p,r))}{\beta_m\, (\diam \overline B(p,r))^m}\,
\end{equation*}
and
\begin{equation*}
\Theta^m_*(\mu,p):=\liminf_{r\to 0}\frac{\mathcal \mu(\overline B(p,r))}{\beta_m\, (\diam \overline B(p,r))^m}\,.
\end{equation*}
If they agree their common value
$$
\Theta^m(\mu,p):=\,\Theta^{*\,m}(\mu,p)=\,\Theta^m_*(\mu,p)
$$
is called the \emph{$m$-density} of $\mu$ at $p$.
\end{definition}

The crucial step of the proof of the liminf inequality $ii^*)$ is provided by the following
theorem that allows us to pass from an inequality between densities to the corresponding
inequality between measures. We point out that this theorem is well known in the Euclidean
setting, but fails to be true in general Carnot-Carath\'eodory
spaces, and its proof in our special setting is postponed to Section \ref{densities}.

We have:

\begin{theorem} \label{august 5} Let $M$ be $(2n+1)$-dimensional contact manifold endowed with a contact form $\theta$ and a Riemannian
metric $g$ on the fibers of $\ker\theta$.  Let $\mathbf{W}^{0}:=(W_1^{0},\dots,W_{2n}^{0})$ be an orthonormal  basis of $\ker\theta$, and
let $E\subset M$ be a set of locally finite sub-Riemannian perimeter associated  with $\mathbf{W}^{0}$. We denote by $|{\bf W}^{0}\chi_E|$
the associated perimeter measure. If $\mu $ is a $\sigma$-finite Borel measure on $X$,
 then
\begin{equation}\label{density 1}
\Theta^{* ,2n+1} (\mu,p) \ge \Theta^{* ,2n+1} (|{\bf W}^{0}\chi_E|,p) \qquad\mbox{ for $\mc H_d^{2n+1}$-a.e. $p\in \partial^* E$}
\end{equation}
yields
\begin{equation}\label{density 2}
\mu\res \partial E (\mathcal B) \ge  |{\bf W}^{0}\chi_E|(\mathcal B)
\end{equation}
for any Borel set $\mathcal B\subset \partial E$.
\end{theorem}

\begin{remark}\label{centered}
Let us explain why we do need Theorem \ref{august 5} precisely in that form, and then we have to go through all the
arguments of Section \ref{densities}. First of all, we recall the following definition:
let $\mu$ be an outer measure on the metric space $(X,d)$. Then the \emph{$m$-Federer densities} of $\mu$ at $x\in X$ are
\begin{equation*}
\Theta_F^{*\,m}(\mu,x):=\inf_{\eps> 0}\sup\left\{\frac{\mathcal \mu(B(y,r))}{\gb_m\,\diam(B(y,r))^m}:\,x\in B(y,r),\,\rho_0\,r\le\,{\eps}\,\right\}.
\end{equation*}
It is easy to see that
\begin{equation}\label{compardensities}
\Theta^{*\,m}(\mu,x)\le\,\Theta_F^{*\,m}(\mu,x)\le\,2^m\,\Theta^{*\,m}(\mu,x)\quad\forall x\in X\,.
\end{equation}
If $X$ is separable and endowed with a Radon measure $\mu$, absolutely continuous  with respect to the $m$-dimensional spherical Hausdorff measure $\mathcal S^m$,
by \cite{magnani_centered} (see also \cite{FSSC_NA}), the area formula for $\mu$ with respect to $\mathcal S^m$ i.e.
\begin{equation}\label{introduz.0}
\mu(B)=\,\int_B \Theta_F^{*\,m}(\mu,x)\,d\mathcal S^m(x)
\end{equation}
for any Borel set $B$ may fail to be true  in general, if the $m$-dimensional           density $\Theta_F^{*\,m}(\mu,\cdot)$ is replaced by the centered $m$-dimensional density $\Theta^{*\,m}(\mu,\cdot)$ (see Definition \ref{densitydef}).

To be more precise, the representation formula \eqref{introduz.0} is known to hold in Heisenberg groups
only for suitable left-invariant distances, as $d_\infty$ (see \cite{FSSC_NA}, Remark 4.25). In particular,
we do not know whether it holds for the spherical Hausdorff measure associated with the Carnot-Carath\'eodory distance,
that we use throughout the present paper (keep in mind its connection with the Minkowski content).

In fact, Magnani provides a counterexample precisely in the first Heisenberg group.

On the other hand, following \cite{ABS2}, a crucial step of the proof of the liminf inequality $ii^*)$ is provided by
the following estimate:
\begin{equation}\label{nov 29}
 \Theta^{* ,2n+1} (\mu,p) \ge  \mathbf{c}\, \Theta^{* ,2n+1} (|{\bf W}^{0}\chi_E|,p),
\end{equation}
where $\mu$ is the limit measure of the  energy distribution associated with $\tilde F_\eps$ and $p\in S_v$.

Unfortunately, due to Magnani's result, if $\mathcal B$ is a Borel set,
we cannot derive from \eqref{nov 29} the corresponding inequality \emph{with the explicit constant $\mathbf c$} for the measures
$\mu(\mathcal B)$ and $|{\bf W}^{0}\chi_E|(\mathcal B)$, that we would need in the sequel.

\end{remark}

\bigskip

Assuming Theorem \ref{august 5}, we can complete the proof of Theorem \ref{theorem bis} as follows.

Let now $\{u_\eps\}$ be a sequence in $W^{1,2}(M\times[0,\sigma))$ such that $\{\trace u_\eps\}$ converges to $v\in BV_\theta(M,\{0,1\})$ in the $L^1(M)$ norm. We need to show that
$$\liminf_{\varepsilon\to 0} \tilde F_\eps(u_\eps)\geq F(v).$$
If we write $v=\chi_E$, then $F(v)=|\mathbf W^{0}\chi_E|$.

Without loss of generality, assume that this liminf is finite.

For every $\eps\in(0,1)$, let $\mu_\eps$ be the energy distribution associated with $\tilde F_\eps$ for $u_\eps$, i.e., $\mu_\eps$ is the positive measure given by
$$\mu_\eps(\mc B):= \eps \int_{\mc B} \left( \sum_{j=1}^{2n} (W_j^0 u_\eps)^2 + (\partial_z u_\eps)^2\right)dv_\theta\wedge dz + \lambda_\eps \int_{\mc B_0} V(\trace u_\eps) \, dv_\theta$$
for every Borel set $\mc B\subset M\times[0,\sigma)$, $\mc B_0=\overline{\mc B}\cap M$. The total variation $\norm{\mu_\eps}$ of the measure $\mu_\eps$ is equal to $\tilde F_\eps(u_\eps)$.

Without loss of generality, we can assume $0\le \tilde F_\eps(u_\eps)\le C$ for every $0<\eps<1$, and therefore the
$\{ \mu_\eps\}$ is an equibounded family of Radon measures in $\Omega$. By De La Vall\'ee Poussin's Theorem
(\cite{AFP}, Theorem 1.59), there exist a subsequence $(\eps_h)_{h\in\mathbb N}$ and a Radon measure $\mu$
in $\Omega$ such that $\mu_{\eps_h}\to \mu$ in the sense of the convergence of measures. Then, by the lower semicontinuity of the total variation we have
$$\liminf_{\eps \to 0}\tilde  F_\eps(u_\eps)=\liminf_{\eps \to 0} \norm{\mu_\eps}\geq \norm{\mu}.$$
Similarly, we define
$$\mu_0(\mc B):=|\mathbf W^{0}\chi_E|(\mc B).$$
We just need to show that
\begin{equation}\label{compare-measures}
\mu\geq \mu_0.
\end{equation}

Take now a point $p\in S_v$.
 For $r$ small enough, we choose a map $\tilde \Psi:=\tilde\Psi_p$ as in the discussion right after Theorem \ref{isometry}. Set $w_\eps:= u_\eps\circ \tilde \Psi^{-1}$ and $\bar v:= v\circ \Psi^{-1}$. Hence, $\mathrm{Tr} w_\eps \to \bar v$ in $L^1(B_r^{\mathbb H})$ and $\bar v\in BV(B_r^{\mathbb H},\{0,1\})$. Moreover, if $v=\chi_E$, then $\bar v=\chi_{\Psi(E)}$ and $\nu_{ v}(\Psi(z))=D\Psi^{-1}(z)\cdot \nu^{\mathbb H}_{\bar v}(z)$, for any $z\in S_{\bar v}$ (here $\nu^{\mathbb H}_{\bar v}$ denotes the geometric normal to $S_{\bar v}$ in $\mathbb H^n$) .
 Keeping in mind \eqref{Aug 25 eq:1} and \eqref{est-flat}, we have
\begin{equation*}\begin{split}
\mu (B(p, & r(1+o(1)))  \ge \mu(U(p,r))
=\lim_{\eps \to 0} \mu_\eps(\tilde U(p,r))\\
& =\lim_{\eps \to 0} \tilde F_\eps(u_\eps,\tilde U(p,r),U(p,r))\\
& \geq \liminf_{\eps \to 0} (1-\delta(\Psi))^{5}  E_\eps(w_\eps,C_r^{\mathbb H}, B_r^{\mathbb H}).
\end{split}\end{equation*}
Notice that $\delta(\Psi)\to 0$ as $r\to 0$.
On the other hand, by Theorem \ref{liminf-thm}, we have that
$$\liminf_{\eps\to 0} E_\eps(w_\eps, C_r^{\mathbb H}, B_r^{\mathbb H})\geq \mathbf c\left|\int_{B_r^{\mathbb H}} \nu^{\mathbb H}_{\overline v} d|\mathbf W^{\mathbb H} \chi_{\Psi(E)}|\right|.$$

We have now, by Lemma \ref{august 1}, ii), and \cite{FSSC_step2}, Lemma 3.8, iii),
\begin{equation}\label{aug 29 eq:1}\begin{split}
\Theta^{*\,2n+1} & (\mu,p):=\limsup_{r\to 0}\frac{\mathcal \mu(\overline B(p,r))}{\beta_{2n+1}\, (\diam \overline B(p,r))^{2n+1}}
\\&=
\limsup_{r\to 0}\frac{\mathcal \mu(\overline B(p,r))}{\alpha_{2n+1}\, r^{2n+1}}
\\&
\ge
\mathbf{c} \,\liminf_{r\to 0}\frac{|\mathbf{W}^{\mathbb H} \chi_{\Psi(E)}|(B_r^{\mathbb H})}{\alpha_{2n+1}\, r^{2n+1}}
\left|\ave_{B_r^{\mathbb H}} \nu_{\bar v}^{\mathbb H}\, d|\mathbf W^{\mathbb H} \chi_{\Psi(E)}| \, \right|.
\end{split}\end{equation}

Let us prove now the following approximation lemma.

\begin{lemma}\label{lemma-perimeters-equivalent}
\begin{equation}\label{sep 5 eq:1}
\lim_{r\to 0}\frac{|\mathbf W^0\chi_E|(\overline B(p,r))}{|\mathbf W^{\mathbb H}\chi_{\Psi(E)}|(B^\mathbb H_r)} = 1.
\end{equation}
\end{lemma}

\begin{proof}
In the notation from Section \ref{subsection:BV}, the perimeter measure in $M$ is defined as
\begin{equation}\label{perimeter1}\begin{split}
|\mathbf W^0\chi_E| & (\overline B(p,r))
\\&
=\sup\{|D_X (\chi_E)|(\overline B(p,r))\,:\, X\in\Gamma (M,\ker \theta),\, g(X,X)\leq 1\}.
\end{split}\end{equation}
Note that from the definition of $D_X$ in \eqref{variation}, it is enough to restrict our attention to vector fields $X$ supported on $\overline B(p,r)$.

On the other hand, by Lemma \ref{one year later} and with the notations therein, if we put
$$
\rho =\rho (r):= r(1+Cr^{1/2}),\qquad\mbox{then $B(p, r) \subset U(p,\rho)$.}
$$
Then
$$
|\mathbf W^\mathbb H\chi_{\Psi(E)}|(B_r^{\mathbb H})= |\mathbf W^\mathbb H\chi_{\Psi(E)}|(\delta^{\mathbb H}_{r/\rho}(B_\rho^{\mathbb H})) =
(1+o(1)) |\mathbf W^\mathbb H\chi_{\Psi(E)}|(B_\rho^{\mathbb H}),
$$
where $\delta$ is the standard group dilation in the Heisenberg group.
We recall now that
\begin{equation}\label{perimeter2}
\begin{split}
|\mathbf W^\mathbb H & \chi_{\Psi(E)}|(B_\rho^{\mathbb H})
\\&=\sup\{|D_Y (\chi_{\Psi(E)})|(B_\rho^{\mathbb H})\,:\, Y\in\Gamma (\mathbb H^n,\ker \theta_0),\, \langle Y,Y\rangle_{\mathbb H}\leq 1\},
\end{split}\end{equation}
where again we can assume $\supp Y\subset B_\rho^{\mathbb H}$.

It remains to compare the metrics $g$ on $M$ and $\langle\,,\,\rangle_{\mathbb H}$ on $\mathbb H^n$. Note that $\Psi$ is a contact map, so we can always write $Y=\Psi_*X$ for $X\in\Gamma (M,\ker \theta)$.
By the change of variables formula  (14) in \cite{Ambrosio-Ghezzi-Magnani}
\begin{equation}\label{change-of-variable1}
\Psi_\#(D_X h)=D_{\Psi_* X}(h \circ \Psi^{-1}),\end{equation}
we have
\begin{equation}\label{change-of-variable}|D_{\Psi_* X}(h\circ\Psi^{-1})|=|\Psi_\#D_X h|.\end{equation}
Using also the definition of push forward of a measure,
\begin{equation*}\begin{split}
|D_Y (\chi_{\Psi(E)}) & |(B_\rho^{\mathbb H}) =\Psi_\# |D_X(\chi_E)|(B^{\mathbb H}_\rho)
\\&
=|D_X(\chi_E)|(U(p,\rho))
\ge |D_X(\chi_E)|(B(p,r)) .
\end{split}\end{equation*}

Finally, in order to compare the perimeter measures \eqref{perimeter1} and \eqref{perimeter2},
we notice that, by Theorem \ref{isometry}, \emph{iii)}  if $ \langle Y,Y\rangle_{\mathbb H}\leq 1$, then $g(X,X)\le 1+o(1)$ as $r\to 0$.

This proves that
$$\limsup_{r\to 0}\frac{|\mathbf W^0\chi_E|(\overline B(p,r))}{|\mathbf W^{\mathbb H}\chi_{\Psi(E)}|(B^\mathbb H_r)} \le 1.$$
The proof of the reverse inequality can be carried out in the same fashion.
\end{proof}

Before going back to the proof of the lower bound inequality, we need the following last lemma.
\begin{lemma} We have:
\begin{equation}\label{almost reduced}
 \left|\ave_{ B_r^{\mathbb H}} \nu^{\mathbb H}_{\bar v}\, d|\mathbf W^{\mathbb H} \chi_{\Psi(E)}| \, \right| = 1 + o(1) \qquad\mbox{as $r\to 0$.}
\end{equation}
\end{lemma}

\begin{proof} We use Lemma \ref{one year later}, with the notations therein, and we put \\
$\phi(r):=(1+ C\sqrt{r})^{-1}$. We have
$$
\phi(r)(1+C\sqrt{r\phi(r)}) \le 1 \qquad\mbox{and}\qquad  \phi(r) = 1+ o(1) \quad \mbox{as $r\to 0$.}
$$
Let us prove first that
\begin{equation}\label{sep 6 eq:1}\begin{split}
 & \left|\ave_{B_{r\phi(r)}^{\mathbb H}}   \nu_{\bar v}^{\mathbb H}\, d|\mathbf W^{\mathbb H} \chi_{\Psi(E)}| \, \right|
 \\&
 \hphantom{xxxx}=
\dfrac{1}{ |\mathbf W^{\mathbb H} \chi_{\Psi(E)}|(B^{\mathbb H}_{r\phi(r)})}
 \left|\int_{\Psi(B(p,r))} \nu_{\bar v}^{\mathbb H}\, d|\mathbf W^{\mathbb H} \chi_{\Psi(E)}| \, \right| + o(1).
\end{split}\end{equation}
First of all, we notice that
$$
B_{r\phi(r)}^{\mathbb H} \subset \Psi(B(p,r)),\qquad 0<r<r_0.
$$
Indeed, take $z\in B_{r\phi(r)}^{\mathbb H}$. Since $\Psi$ is a diffeomorphism, we can assume that $z=\Psi (\zeta)$,
with $\zeta\in M$, provided $r$ is small enough. Therefore
\begin{equation*}\begin{split}
d_c(p, \zeta) = d^\Psi_c (0,z) \le r\phi(r)(1+C\sqrt{r\phi(r)}) \le r.
\end{split}\end{equation*}
Analogously
$$
\Psi(B(p,r)) \subset B_{r/ \phi(r)}^{\mathbb H},\qquad 0<r<r_0.
$$
Therefore, in order to prove \eqref{sep 6 eq:1}, we have to show in the first place that
$$
\dfrac{1}{ |\mathbf W^{\mathbb H} \chi_{\Psi(E)}|(B^{\mathbb H}_{r\phi(r)})}
 \left|\int_{ \Psi(B(p,r))\setminus B^{\mathbb H}_{r\phi(r)}} \nu_{\bar v}^{\mathbb H}\, d|\mathbf W^{\mathbb H} \chi_{\Psi(E)}| \, \right| = o(1).
$$
On the other hand, keeping in mind the homogeneity of $ |\mathbf W^{\mathbb H} \chi_{\Psi(E)}|$
with respect to group dilations $\delta^{\mathbb H}$, we have:
\begin{equation*}\begin{split}
& \dfrac{1}{ |\mathbf W^{\mathbb H} \chi_{\Psi(E)}|(B^{\mathbb H}_{r\phi(r)})}
 \left|\int_{ \Psi(B(p,r))\setminus B^{\mathbb H}_{r\phi(r)}} \nu^{\mathbb H}_{\bar v}\, d|\mathbf W^{\mathbb H} \chi_{\Psi(E)}| \, \right|
 \\& \hphantom{xxxx}
 \le
 \dfrac{|\mathbf W^{\mathbb H} \chi_{\Psi(E)}|( \Psi(B(p,r)))
 - |\mathbf W^{\mathbb H} \chi_{\Psi(E)}|( B^{\mathbb H}_{r\phi(r)})}
 { |\mathbf W^{\mathbb H} \chi_{\Psi(E)}|(B^{\mathbb H}_{r\phi(r)})}
  \\& \hphantom{xxxx}
 \le
 \dfrac{|\mathbf W^{\mathbb H} \chi_{\Psi(E)}|( B^{\mathbb H}_{r /\phi(r)}) -|\mathbf W^{\mathbb H} \chi_{\Psi(E)}|( B^{\mathbb H}_{r\phi(r)})}
 { |\mathbf W^{\mathbb H} \chi_{\Psi(E)}|(B^{\mathbb H}_{r\phi(r)})}
   \\& \hphantom{xxxx}
=
 \dfrac{|\mathbf W^{\mathbb H} \chi_{\Psi(E)}|( B^{\mathbb H}_{1 /\phi(r)})-
 |\mathbf W^{\mathbb H} \chi_{\Psi(E)}|(B^{\mathbb H}_{\phi(r)})}
 { |\mathbf W^{\mathbb H} \chi_{\Psi(E)}|(B^{\mathbb H}_{\phi(r)})}
 \\& \hphantom{xxxx} = o(1).
\end{split} \end{equation*}
This yields \eqref{sep 6 eq:1}.

Take now $Y:=\Psi_*X$, with $\langle Y,Y\rangle_{\mathbb H} = 1$. By the change of variable formula \eqref{change-of-variable1},
$$\Psi_\#(D_X (\chi_E))=D_{\Psi_* X}(\chi_{\Psi(E)}),$$
and thus
\begin{equation}\label{sep 7 eq:2}\begin{split}
\Big\langle Y, & \int_{ \Psi(B(p,r))} \nu_{\bar v}\, d|\mathbf W^{\mathbb H} \chi_{\Psi(E)}|\Big\rangle_{\mathbb H}=
 \int_{ \Psi(B(p,r))} \langle Y,\nu_{\bar v}\rangle\, d|\mathbf W^{\mathbb H} \chi_{\Psi(E)}|
\\&=
D_Y \chi_{\Psi(E)} (\Psi(B(p,r))) = \Psi_\#(D_X \chi_E)(\Psi(B(p,r)))
\\& =
D_X \chi_E(B(p,r))
=
g\big( X,  \int_{B(p,r)} \nu_{v}\, d|\mathbf W^{0} \chi_{ E}|\big)
\\&\le
\|X\|_g \,\Big\| \int_{B(p,r)} \nu_{v}\, d|\mathbf W^{0} \chi_{ E}|\Big\|_g.
 \end{split}\end{equation}
As in the proof of previous lemma, $\|X\|_g = 1 + o(1)$. On the other hand,
keeping in mind that $p$ belongs to the reduced boundary of $E$,
\begin{equation*}
\lim_{\rho\to 0}\frac{1}{|\mathbf W^0 \chi_E (B(p,r))|}\,\Big\| \int_{B(p,r)} \nu_{v}\, d|\mathbf W^{0} \chi_{ E}|\Big\|_g=1.
\end{equation*}
But by the previous formula \eqref{sep 5 eq:1}, and the fact that
$$\lim_{r\to 0}\frac{|\mathbf W^{\mathbb H} \chi_{\Psi(E)}|(B^{\mathbb H}_{r\phi(r)})}
{|\mathbf W^{\mathbb H} \chi_{\Psi(E)}|(B^{\mathbb H}_{r})}=1$$
using a rescaling argument by dilations in the Heisenberg group, we conclude from \eqref{sep 7 eq:2} that
\begin{equation*}\lim_{\rho\to 0}\Big\langle Y,\dfrac{1}{ |\mathbf W^{\mathbb H} \chi_{\Psi(E)}|(B^{\mathbb H}_{r\phi(r)})}
\int_{\Psi(B(p,r))} \nu_{\bar v}^{\mathbb H}\, d|\mathbf W^{\mathbb H} \chi_{\Psi( E)}| \,\big\rangle_{\mathbb H}\leq 1.
 \end{equation*}
A standard argument taking the sup among all $Y$ (or equivalently, all $X$) with norm less than one, looking back at \eqref{sep 6 eq:1}, completes the proof of the Lemma.
\end{proof}

 We can go back to the proof of \eqref{compare-measures}.
 Replacing both \eqref{sep 5 eq:1} and \eqref{almost reduced} into \eqref{aug 29 eq:1} we conclude that
 $$
 \Theta^{* ,2n+1} (\mu,p) \ge  \mathbf{c}\, \Theta^{* ,2n+1} (|{\bf W}^{0}\chi_E|,p).
 $$
The proof of the lower bound inequality is completed by Theorem \ref{august 5}.

\section{Proof of the main theorem - limsup}\label{sec6}

Now we show statement \emph{iii)} of Theorem \ref{theorem}. Given $v\in BV_{\theta}(M,\{0,1\})$, we need to construct a sequence $\{  u_\eps\}$ in $W^{1,2}(\Omega)$ such that
$\trace   u_\eps \to v$ in $L^1(M)$ and
$$\limsup_{\eps\to 0} F_\eps(  u_{\eps})\leq F(v).$$

The proof of the limsup inequality will be divided into several steps:\\

\emph{Step 1:} It is enough to assume that $S_v$ is a smooth closed submanifold in $M$. This fact follows from the next two results. The first one is a reduction Lemma.  It is valid for general metric spaces, and the proof is only a minor variant of the one given in \cite{MoMo}, Lemma $IV$ (see also \cite{Alberti:survey}), hence we shall omit such a proof.
\begin{lemma}\label{reduction}
Let $(\mathcal X, \textsl{d})$ be a metric space, let
$F_k,\;F:\;\mathcal X\longrightarrow[-\infty,+\infty]$  with $k\in\mathbb N$;
consider $\mathcal D\subset \mathcal X$ and $x\in \mathcal X$.
Let us suppose that
\begin{itemize}
    \item[1)]for every $y\in \mathcal D$ there exists a sequence
$(y_k)_{k\in\mathbb N}\subset \mathcal X$ such that $y_k\rightarrow y$ in $\mathcal X$ and
    $$\displaystyle\limsup_{k\rightarrow\infty}F_k(y_k)\leq F(y);$$
    \item[2)]there exists a sequence $(x_k)_{k\in\mathbb N}\subset \mathcal D$ such
that $x_k\rightarrow x$  and
     $$\limsup_{k\rightarrow\infty}F(x_k)\leq F(x);$$
\end{itemize}
then there exists a sequence $(\overline x_k)_{k\in\mathbb N}\subset \mathcal X$ such that
$\displaystyle\limsup_{k\rightarrow\infty}F_k(\overline x_k)\leq F(x).$
   \end{lemma}
The following approximation result is the analogue of Corollary 2.3.6 in \cite{Franchi-Serapioni-SerraCassano} for the case of contact manifolds.
\begin{lemma}
Each $v\in BV_{\theta}(M,\{0,1\})$ may be approximated in $L^1(M)$ by a sequence $\{v_k\}$ in $BV_{\theta}(M,\{0,1\})$ such that $S_{v_k}$ is a smooth closed submanifold and
$$\|S_{v_k}\|_\theta\to \|S_v\|_\theta.$$
\end{lemma}
\begin{proof}
The result follows by standard arguments from the Meyers-Serrin type result, Theorem 2.4 in \cite{Ambrosio-Ghezzi-Magnani}, and the coarea formula (Proposition \ref{coarea}).
\end{proof}
Next, possibly modifying $v$ on a negligible subset, we can assume that it is constant in each connected component of $M\setminus S_v$.\\

\emph{Step 2:} (Preliminary calculations). Following the idea in \cite{ABS2}, we take a function defined as follows:
consider the half-plane $\mathbb R^2_+$ with coordinates $s\in\mathbb R$, $z>0$. Let $(\rho,\vartheta)$, $\rho>0$, $\vartheta\in[0,\pi]$ be the polar coordinates in $\mathbb R^2_+$.

We set
\begin{equation*}
\bar w_\eps(\rho,\vartheta):=\left\{
\begin{split}
\rho\frac{\lambda_\eps}{\eps}(1-\tfrac{2}{\pi}\vartheta)\quad & \text{if }0\leq\rho\leq\frac{\eps}{\lambda_\eps},\\
1-\tfrac{1}{\pi}\vartheta \quad &\text{if }\frac{\eps}{\lambda_\eps}\leq \rho,
\end{split}\right.\end{equation*}
and $w_\eps(s,z)=\bar w_\eps (\rho,\vartheta)$.
A straightforward calculation gives:
\begin{equation}\label{first-deriv}
|\partial_s w_\eps|,|\partial_z w_\eps|\leq\left\{
\begin{split}
C\frac{\lambda_\eps}{\eps}\quad & \text{if }0\leq\rho\leq\frac{\eps}{\lambda_\eps},\\
\frac{C}{\rho} \quad &\text{if }\frac{\eps}{\lambda_\eps}\leq \rho,
\end{split}\right.\end{equation}
and
\begin{equation}\label{second-deriv}
|\partial_{ss} w_\eps|,|\partial_{zs} w_\eps|\leq\left\{
\begin{split}
\frac{C}{\rho}\frac{\lambda_\eps}{\eps}\quad & \text{if }0\leq\rho\leq\frac{\eps}{\lambda_\eps},\\
\frac{C}{\rho^2} \quad &\text{if }\frac{\eps}{\lambda_\eps}\leq \rho.
\end{split}\right.\end{equation}
Moreover, the following estimates hold:
\begin{lemma}\label{lemma:calculation} Let $t_\eps\to 0$  as $\eps\to 0$ and  $\sigma>0$ in such a way that $\frac{\eps}{\lambda_\eps}\ll t_\eps\ll \sigma$. Then, as $\eps \to 0$,
\begin{equation*}
\begin{split}
&\eps\int_{\{\rho<t_\eps\}}|\nabla w_\eps|^2\,dsdz=\frac{1}{\pi}\eps\log\frac{\lambda_\eps}{\eps}(1+o(1)),\\
&\eps\int_{\{t_\eps<\rho<\sigma\}}|\nabla w_\eps|^2\,dsdz=\eps\log t_\eps(1+o(1)) =o\left(\eps\log\frac{\lambda_\eps}{\eps}\right),\\
&\lambda_\eps\int_{\{z=0\}\cap\{\rho<t_\eps\}} V(\trace w_\eps)\,ds=O(\eps),
\quad \lambda_\eps\int_{\{z=0\}\cap\{\rho>t_\eps\}} V(\trace w_\eps)\,ds=O(\eps).
\end{split}
\end{equation*}
\end{lemma}

\begin{proof}
While the first two identities follow from straightforward calculation from the previous estimates, for the third one we use that $V\equiv 0$ unless $0\leq \rho \leq \frac{\eps}{\lambda_\eps}$. Also, from the proof it follows that these estimates are independent of the choice of $\sigma$.
\end{proof}
\medskip

\emph{Step 3:} (Set up).
As we saw in Section \ref{straight freeze}, given $\sigma>0$ small enough, there exists a diffeomorphism $\Phi$ such that a tubular neighborhood of $M$ in $\overline\Omega$ may be written as $M\times [0,\sigma)$, with coordinates $p\in M$ and $z\in[0,\sigma)$.
In the product $M\times [0,\sigma)$ we shall define the distance
$$d((p',z'),(p'',z''))=\sqrt{d_c(p',p'')^2+(z'-z'')^2}.$$
For each $r$ small consider the following subset of $M\times[0,\sigma)$:
\begin{equation*}
\tilde A_r=\{(p,z)\in M\times[0,\sigma) \,:\, d(p,S_v)<r \},\\
\end{equation*}
and set
$$\partial^0 \tilde A_r=\overline{\tilde A_r}\cap M.$$
In coordinates $(p,z)\in \tilde A_{\sigma}$ where $p\in M$ and $z>0$, let
$$u_\eps(p,z):=w_\eps(d_c(p,S_v),z),$$
and transplant it back to $\Omega$ by
$$  u_\eps= \tilde u_\eps\circ \Phi^{-1}, \quad A_{r}=\Phi(\tilde A_{r}),
$$
for each $0<r<\sigma.$
Note that $\Phi$ can be defined independently of $\varepsilon$. Next, because of hypothesis \emph{H2.} for $f$ in Section \ref{section-geometry}, and Proposition \ref{model-prop}, in the calculation of the energy functional $F_\varepsilon$ in a neighborhood of $M$ we have
\begin{equation}
\label{equation50}F_\eps(  u_\eps,A_{\sigma},\partial^0 A_{\sigma})\leq (1+O(\sigma))\tilde F_{\eps,\sigma} (\tilde u_\eps,\tilde A_{\sigma},\partial^0 \tilde A_{\sigma}),
\end{equation}
so it is enough to estimate the integral in the right hand side.

Now, the phase transition should happen at scale $\eps$. For this, let $t_\eps$ be as in Lemma \ref{lemma:calculation}, actually it is enough to  take
$t_\eps=\eps$. Then,
$$\tilde F_{\eps,\sigma}(\tilde u_\eps,\tilde A_{\sigma},\partial^0 \tilde A_{\sigma})=\tilde F_{\eps,\sigma}(\tilde u_\eps,\tilde A_{\sigma}\setminus \tilde A_{t_\eps},\partial^0(\tilde A_{\sigma}\setminus \tilde A_{t_\eps}))+\tilde F_{\eps,\sigma}(\tilde u_\eps, \tilde A_{t_\eps},\partial^0 \tilde A_{t_\eps}).$$
The last term in the right hand side above will be considered in Step 4, while the first one will be handled in Step 5.

On the other hand, it is not important how we define $u_\eps$ in the set $\Omega\backslash A_\sigma$, as long as $u_\eps=v$ in $\Omega\backslash \partial^0 A_\sigma$ and its Lipschitz constant is bounded by $\frac{C}{\sigma}$. Recall that $v$ is a function that only attains the values $0$ or $1$ on $M\setminus \partial^0 A_{\sigma}$, so that for the potential energy we have
  $$\int_{M\setminus \partial^0 A_{\sigma}} V(\trace   u_\eps)\,dv_\theta=0.$$
Then we immediately have that
\begin{equation}\label{limsup1}\limsup_{\eps\to 0}F_\eps(  u_\eps,\Omega\backslash A_{\sigma},M\backslash\partial^0 A_{\sigma})=0.\end{equation}

\emph{Step 4.} (Construction near the singular set).
We follow the ideas of \cite{Monti-SerraCassano} to estimate the value of $\tilde F_{\eps,\sigma}(\tilde u_\eps,\tilde A_{t_\eps},\partial^0\tilde A_{t_\eps})$.
Let $s=d_c(p,S_v)$. Then, using Fubini's theorem,
\begin{equation}\begin{split}
\tilde F_{\eps,\sigma}&(\tilde u_\eps,\tilde A_{t_\eps},\partial^0\tilde A_{t_\eps})\\
&= \int_{\partial^0\tilde A_{t_\eps}}\left[\eps\int_{0}^{\sqrt{t_\eps^2-s^2}} \sum_{j=1}^{2n}|\tilde W_j \tilde u_\eps(p,z)|^2 \,dz+\lambda_\eps V(\trace \tilde u_\eps(p))\right]dv_\theta.
\end{split}\end{equation}
Using the coarea formula from Theorem \ref{coarea} and the Eikonal equation for $d_c$ \eqref{Eikonal-equation} we have
\begin{equation*}
\begin{split}
\tilde F_{\eps,\sigma}(\tilde u_\eps,\tilde A_{t_\eps},\partial^0 \tilde A_{t_\eps})&
= \int_{-t_\eps}^{t_\eps} h_\eps(s)\,d\|\partial H_s\|_{\theta}\,ds,
\end{split}\end{equation*}
where we have set
\begin{equation}
\label{h}h_\eps(s):=\eps\int_{0}^{\sqrt{t_\eps^2-s^2}}\left[(\partial_s w_\eps(s,z))^2+(\partial_z w_\eps(s,z))^2\right]\,dz+\lambda_\eps V(\trace w_\eps(s))\end{equation}
and $H_s=\{p\in M : d_c(p,S_v)>s\}$.
Next, notice that for all $s\in [-t_\eps,t_\eps]$, $h_\eps(s)=h_\eps(-s)$, so that
$$\tilde F_{\eps,\sigma}(u_\eps,\tilde A_{t_\eps},\partial^0\tilde A_{t_\eps})\leq \int_{0}^{t_\eps} h_\eps(s)\,\left(d\|\partial H_s\|_{\theta}+d\|\partial H_{-s}\|_\theta\right)ds.$$
We can rewrite this expression as follows: let
$$Z(t)=\int_{-t}^t \|\partial H_s\|_\theta\,ds,\quad Z'(t)=\|\partial H_s\|_{\theta}+\|\partial H_{-s}\|_\theta,$$
so that
\begin{equation}\label{equation10}
\tilde F_{\eps,\sigma}(\tilde u_\eps,\tilde A_{t_\eps},\partial^0 \tilde A_{t_\eps})\leq \int_{0}^{t_\eps} h_\eps(s) Z'(s)\,ds=-\int_0^{t_\eps} h'_\eps(s) Z(s)\,ds
\end{equation}
after integration by parts. Note that we have used that $h_\eps(t_\eps)=0$.

Next, by Theorem \ref{Minkowski} we have
$$\lim_{t\to 0^+}\frac{Z(t)}{2t}=L:=\|\partial H\|_\theta,$$
and thus, there exists a function $\delta:[0,\infty)\to\mathbb R$ such that
\begin{equation}\label{delta}
Z(t)=2Lt+\delta(t)t,\quad \text{with }\lim_{\eps\to 0^+}\sup_{t\in[0,t_\eps]}|\delta(t)|=0.\end{equation}
Substituting the above into \eqref{equation10} we obtain that
\begin{equation}\label{equation20}\begin{split}
\tilde F_{\eps,\sigma}(\tilde u_\eps,\tilde A_{t_\eps},\partial^0 \tilde A_{t_\eps})&\leq -\int_0^{t_\eps} s\delta(s)h'_\eps(s)\,ds-
2L\int_0^{t_\eps}sh'_\eps(s)\,ds.\\
&=:I_\eps+J_\eps.
\end{split}\end{equation}
In order to estimate the term $J_\eps$ above, we use again integration by parts
$$J_\eps=2L\int_0^{t_\eps} h_\eps(s)\,ds=L\int_{-t_\eps}^{t_\eps}h_\eps(s)\,ds.$$
From the estimates in Lemma \ref{lemma:calculation}, using our initial hypothesis on $\lambda_\varepsilon$ from \eqref{lambda}, we may conclude
$$J_\eps\,{\longrightarrow}\, \frac{\kappa}{\pi}L\quad\text{as }\eps \to 0.$$
Finally, we need to show that the remaining term $I_\eps$ has limit zero when $\eps\to 0$.
But
$$|I_\eps|\leq \sup_{t\in[0,t_\eps]}|\delta(t)|\int_0^{t_\eps} s|h_\eps'(s)|\,ds.$$
From the behavior of $\delta$ in \eqref{delta}, it is enough to show that the integral
\begin{equation}\label{tildeI}\tilde I_\eps:=\int_0^{t_\eps} s|h_\eps'(s)|\,ds
\end{equation}
is bounded independently of $\varepsilon$. Differentiating in \eqref{h}, $h'_\eps(s)=h^1_\eps+h^2_\eps+h^3_\eps$ for
\begin{equation*}
\begin{split}
&h^1_\eps(s)=\eps\left[(\partial_s w_\eps(s,\sqrt{t_\eps^2-s^2}))^2+(\partial_z w_\eps(s,\sqrt{t_\eps^2-s^2}))^2\right]\cdot\left(-\frac{s}{\sqrt{t_\varepsilon^2-s^2}}\right),\\
&h^2_\eps(s)=2\eps\int_{0}^{\sqrt{t_\eps^2-s^2}} \left[ \partial_s w_\eps \partial_{ss}w_\eps+\partial_z w_\eps \partial_{zs}w_\eps\right]\,dz,\\
&h^3_\eps(s)=\lambda_\eps V'(\trace w_\eps(s))\partial_s w_\eps(s,0).
\end{split}
\end{equation*}
Since we know that $t_\eps\gg\frac{\eps}{\lambda_\eps}$, using the estimates in \eqref{first-deriv}, we deduce
$$|h^1_\eps(s)|\leq C\frac{\eps}{t_\eps^2}\frac{s}{\sqrt{t_\eps^2-s^2}},$$
so we may conclude
\begin{equation}\label{h1}
\begin{split}
\int_0^{t_\eps} s|h^1_\eps(s)|\,ds &\leq C\frac{\eps}{t_\eps^2}\int_0^{t_\eps}\frac{s^2}{\sqrt{t_\eps^2-s^2}}\,ds\\
&\leq C\frac{\eps t_\eps}{t_\eps^2}\int_0^{t_\eps}\frac{s}{\sqrt{t_\eps^2-s^2}}\,ds\\
&\leq C\frac{\eps  t_\eps}{t_\eps^2}\left[\sqrt {t_\eps^2-s^2}\right]_0^{t_\eps}\leq C
\end{split}
\end{equation}
independent of $\varepsilon$. For the second integral, note that the estimates in \eqref{first-deriv}-\eqref{second-deriv} give
\begin{equation}\label{h2}
\begin{split}
\int_0^{t_\eps} s |h^2_\eps(s)|\,ds&\leq C\eps\left[\int_{\{0<\rho<\frac{\eps}{\lambda_\eps}\}} s\left(\frac{\lambda_\eps}{\eps}\right)^2\,d\rho
+\int_{\{\frac{\eps}{\lambda_\eps}<\rho<t_\eps\}} \frac{s}{\rho^2}\,d\rho\right]\\
& \leq C\eps\log \lambda_\eps <\infty
\end{split}
\end{equation}
by our initial hypothesis \eqref{lambda}. Finally, looking again at the estimates \eqref{first-deriv} for $\partial_s w_\eps$, we have
\begin{equation}\label{h3}
\int_0^{t_\eps} s|h^3_\eps(s)|\,ds\leq C\lambda_\eps\int_0^{\frac{\eps}{\lambda_\eps}} s\frac{\lambda_\eps}{\eps}\,ds<\infty.
\end{equation}
Putting together \eqref{h1}, \eqref{h2} and \eqref{h3} we conclude that the integral $\tilde I_\eps$ from \eqref{tildeI} is uniformly bounded independently of $\eps$. This shows that, looking at \eqref{equation20} and \eqref{equation50},
\begin{equation}\label{limsup2}
\limsup_{\eps\to 0} F_\eps(  u_\eps,A_{t_\eps},\partial^0 A_{t_\eps}) \leq (1+O(\sigma)) \frac{\kappa}{\pi}L,
\end{equation}
as desired.\\

\emph{Step 5:} (Construction in $A_{\sigma}\setminus A_{t_\eps}$). This argument is very close to that of \cite{ABS2}.

First we set $ u_\eps\equiv v$ on $M\setminus \partial^0 A_{t_\eps}$ (recall that $v$ is a function that only attains the values $0$ or $1$ on $M\setminus \partial^0 A_{t_\eps}$), so that
  $$\int_{M\setminus \partial^0 A_{t_\eps}} V(\trace   u_\eps)\,dv_\theta=0.$$

To conclude the proof we need  the following extension lemma, which is a much simplified version of Lemma 4.11 in \cite{ABS2}.
\begin{lemma}
Let $A$ be a domain in $\mathbb R^{2N}$ and $A'\subset \partial A$. Let $\eps\in(0,1)$, and $v$ a Lipschitz function $v:A'\to [0,1]$. Then $v$
admits an extension $u:A\to[0,1]$ such that its Lipschitz constant satisfies
$$Lip(u)\leq \frac{1}{\eps}+ \Lip(v)$$
and
$$\eps\int_A |\nabla u|^2\leq (\eps \Lip(v))^2(|\partial A|+o(1)),$$
and $o(1)$ is a function of $\eps$ which  does not depend on $v$.
  \end{lemma}
From the previous steps we have constructed a function $  u_\eps$ that has a smooth transition from $0$ to $1$ along $\partial A_{t_\eps}$ and along $A_\sigma$, so at most its Lipschitz constant is $\frac{C}{t_\eps}$ (recall that $t_\eps\ll\sigma)$. Thus, using the previous Lemma, we may extend $u_\eps$ to $A_{\sigma}\backslash A_{t_\eps}$ in a Lipschitz fashion while
\begin{equation}\label{limsup3}F_\eps(u_\eps,A_\sigma\backslash A_{t_\eps},\partial^0 (A_\sigma\backslash A_{t_\eps}))=\eps \int_{A_\sigma\backslash A_{t_\eps}} f(y,D  u_\eps(y))\,dy\leq C(1+o(1))O(\sigma).
\end{equation}
as $\eps\to 0$ because of our hypothesis on $f$.\\

By construction, it is clear that $Tu_\eps \to v$ in $L^1(M)$. Putting together \eqref{limsup1}, \eqref{limsup2} and \eqref{limsup3}, the proof of the $\limsup$ is completed by taking $\sigma$ small enough.

%
%
%

\section{Appendix: Densities and measures}\label{densities}

In this Appendix we prove Theorem \ref{august 5}, which was a crucial ingredient in the proof of the liminf inequality. In order to do that, we need some preliminaries on densities and measures.

As in Theorem \ref{isometry}, let $(W_1^{0},\dots,W^{0}_{2n})$ be an orthonormal symplectic basis of $\ker\theta(\bar p)$,
and let $(W_1^{\mathbb H},\dots,W^{\mathbb H}_{2n})$ be the canonical orthonormal symplectic basis of $\ker\theta_0$
($\theta_0$ being the canonical contact form of $\he n$).
Let now $\mc U\subset M$ and, for $\bar p\in \mc U$, let
$
\Psi: \mc U\to \he {n}
$
be the \emph{contact} diffeomorphism constructed
 in Theorem \ref{isometry}. In $\Psi(\mc U)$, consider now the vector fields
 $\Psi_*W_i^0$,  $i=1,\dots,2n$. Notice that
 $$\mathrm{span}\; \{\Psi_*W_1^0,\dots, \Psi_*W_{2n}^0\}
 = \ker\theta_0 = \mathrm{span}\; \{{W}^{\mathbb H}_1,\dots,{W}^{\mathbb H}_{2n} \}.
 $$
 Remember that $\Psi(\bar p) =0$. By
the same theorem, $\Psi_*W_i^0(0)= W^{\mathbb H}_i(0)$ for  $i=1,\dots,2n$. We denote by
$d_c^\Psi$ the Carnot-Carath\'eodory distance in $\Psi(\mc U)$ associated with
the Riemannian metric $(\Psi^{-1})^*g$, and by $d^{\mathbb H}_c$ the standard
Carnot-Carath\'eodory distance in $\he n$.
We denote also by $\overline B_\Psi$ and $\overline B_{\mathbb H}$ the closed balls associated with $d_c^\Psi$ and $d^{\mathbb H}_c$, respectively.

 It is easy to see that for $p,q\in \mc U$
$$
d_c (p,q) = d_c^\Psi(\Psi(p),\Psi(q)).
$$
In the sequel,  $B^\Psi $ will be the open balls with respect to $d_c^\Psi$.

\begin{lemma}\label{one year later} For $z$ in a neighborhood of $0\in \mathbb H^n$, the following estimates hold:
\begin{equation}\label{1aug eq:1}
d_{\mathbb H}(z,0) \le d_c^\Psi(z,0)(1+ Cd_c^\Psi(z,0)^{1/2});
\end{equation}
\begin{equation}\label{1aug eq:2}
d_c^\Psi(z,0) \le d_{\mathbb H} (z,0)(1+ Cd_{\mathbb H}(z,0)^{1/2}).
\end{equation}

\end{lemma}

\begin{proof}
We denote by $\mc W_\Psi$ and $\mc W_{\mathbb H}$ the $(2n\times 2n)$-matrices whose columns are
$ \Psi_*{W}_1^0,\dots, \Psi_*{W}_{2n}^0 $ and ${W}^{\mathbb H}_1,\dots,{W}^{\mathbb H}_{2n} $, respectively. If we set
$$
\mc A := (a_{ij})_{i,j =1,\dots, 2n} := \mc W_{\mathbb H}^{-1} \mc W_\Psi,
$$
we obtain that $\mc A$ transforms the coordinates with respect to
$(\Psi_*{W}_1^0,\dots,\Psi_*{W}_{2n}^0)$
of a generic point in $\ker \theta_0$ into its coordinates with respect to $({W}^{\mathbb H}_1,\dots,{W}^{\mathbb H}_{2n} )$.
If we denote by $z$ a generic point of $\Psi(\mc U)$, by Theorem \ref{isometry},
$$
\mc A(z) = \mathrm{Id} + O(|z|)\qquad\mbox{as $z\to 0$.}
$$
Let now $z\in K\subset\subset \Psi(\mc U)$ be fixed, and let $\gamma : [0,1]\to \he n$ a (smooth)
$d_c^\Psi$-geodesic connecting $0$ and $z$. If $t\in [0,1]$, we can write
$$
\gamma'(t) = \sum_i \gamma_i (t) (\Psi_*W_i^0)(\gamma(t)) \quad\mbox{and}\quad
d_c^\Psi(z,0) = \int_0^1 \big(\sum_i \gamma_i^2(t)\big)^{1/2}\, dt.
$$
Thus,  if $t\in [0,1]$, we have
$$
\gamma'(t) = \sum_i \big\{\sum_j a_{i,j}(\gamma(t))\gamma_j(t)\big\} W^{\mathbb H}_i(\gamma(t)),
$$
and hence
\begin{equation*}\begin{split}
d_{\mathbb H}(z,0) & \le \int_0^1 \Big( \sum_i \big\{\sum_j a_{i,j}(\gamma(t))\gamma_j(t)\big\}^2\Big)^{1/2}\; dt
\\&=
 \int_0^1 \Big( \sum_i \big\{\sum_j (\delta_{i,j}+ O(|\gamma(t)|) )\gamma_j(t)\big\}^2\Big)^{1/2}\; dt
\\&=
 \int_0^1 \Big( \sum_i \big\{\gamma_i(t) + O(|\gamma(t)|^2)\big\}^2\Big)^{1/2}\; dt
 \\&\le
  \int_0^1 \Big( \sum_i \gamma_i(t)^2\Big)^{1/2}\; dt +\int_0^1 O(|\gamma(t)|^{3/2})\, dt
   \\&=
   d_c^\Psi(z,0)+\int_0^1 O(|\gamma(t)|^{3/2})\, dt.
\end{split}\end{equation*}
On the other hand, since the Euclidean distance may be locally bounded by $d_c^\Psi$,
\begin{equation*}\begin{split}
|\gamma(t)| \le C_1 d_c^\Psi(\gamma(t),0)   \le  C d_c^\Psi(z,0),
\end{split}\end{equation*}
so that \eqref{1aug eq:1} follows.
We can carry out the same argument interchanging the roles of $d_{\mathbb H}$ and  $d_c^\Psi$, and we get \eqref{1aug eq:2}.

\end{proof}

To keep our paper as self-contained as possible, we gather here few more or less known results about
Hausdorff measures in metric spaces. This part is  taken  almost verbatim from \cite{FSSC_NA}.



We recall first the definition of a centered  density for an outer measure $\mu$ on $X$ from Definition \ref{densitydef}. In Euclidean spaces (and more generally in Carnot groups) we can replace in this definition
the diameter $ \diam \overline B(x,r)$ by $2r$. This ``elementary'' statement fails to be true in general metric spaces,
but still holds in contact manifolds endowed with their Carnot-Carath\'eodory distance. This
will follow from the following results.

\begin{lemma} \label{august 1} Let $M$ be a $(2n+1)$-dimensional contact manifold endowed with the contact form $\theta$,
with the volume form $v_\theta:= \theta\wedge (d\theta)^n$,
 and the Riemannian
metric $g$ on $\ker\theta$ as introduced in Propositions \ref{contact 1} and \ref{contact 3}. We denote by $d_c$
the associated Carnot-Carath\'eodory distance. Let $\bar p\in M$ be
a fixed point. We have:
\begin{itemize}
\item[i)] if $c_0$ is the volume of the unit ball in $\he n$ for the
Carnot-Carath\'eodory distance associated with the canonical basis
$(W_1^{\mathbb H},\dots,W^{\mathbb H}_{2n})$ of $\he n$ (see Theorem \ref{isometry}), then
$$\lim_{r\to 0}\dfrac{ v_\theta (\overline B(x,r))}{r^{2n+2}} = c_0;$$
\item[ii)] Moreover,
$$ \lim_{r\to 0}  \dfrac{\diam \overline B(x,r)}{2r} = 1.  $$
\end{itemize}
\end{lemma}

\begin{proof}

Take a ball $\overline B_r:=\overline B(\bar p,r)\subset M$ with $r>0$ sufficiently small.
For sake of simplicity, in Lemma \ref{one year later},
put $\phi(t):= t(1+C\sqrt{t})$. Obviously,
$\phi(r) = r + o(r)$ and
$\phi^{-1}(s) = s + o(s)$ as $s\to 0$.

By \eqref{1aug eq:1} and \eqref{1aug eq:2}
\begin{equation}\label{2aug eq:1}\begin{split}
\overline B_{\mathbb H}(0,  \phi^{-1}(r)) &\subset \Psi(\overline B_r) = B^\Psi (0,r)
\subset \overline B_{\mathbb H}(0, \phi(r)).
\end{split}\end{equation}
We recall now that for $\rho>0$
$$
c_0\rho^{2n+2} =\mc L^{2n+1} (\overline B_{\mathbb H}(0,\rho)) = \int_{\overline B_{\mathbb H}} dv_{\theta_0},
$$
and that
\begin{equation*}\begin{split}
v_\theta (\overline B_r) & = \int_{\overline B_r} \theta\wedge(d\theta)^n = \int_{\Psi(\overline B_r)} (\Psi^{-1})^*(\theta\wedge(d\theta)^n)
\\&
= \int_{\Psi(\overline B_r)} (\Psi^{-1})^*\theta\wedge(d(\Psi^{-1})^*(\theta)^n)
= \int_{\Psi(\overline B_r)} \theta_0\wedge(d\theta_0)^n
\\&
= \int_{ B^\Psi (0,r)} dv_{\theta_0} = v_{\theta_0}( B^\Psi (0,r)),
\end{split}\end{equation*}
so that
\begin{equation*}\begin{split}
c_0 (\phi^{-1}(r))^{2n+2} \le v_\theta (\overline B_r) \le c_0  \phi(r)^{2n+2}.
\end{split}\end{equation*}
Then i) follows straightforwardly.

Let us prove ii). If $r>0$ By \cite{FSSC_step2}, Proposition 2.4,  there exist $z_r,\zeta_r\in \overline B_{\mathbb H}(0, \phi^{-1}(r))$
such that $d_{\mathbb H}(z_r,\zeta_r)= 2\phi^{-1}(r)$. Arguing as above, if $\gamma:[0,1]\to\he n$ is a $d_c^\Psi$-geodesic connecting
$z_r$ and $\zeta_r$, then
$$
d_{\mathbb H}(z_r,\zeta_r) \le   d_c^\Psi(z_r,\zeta_r)+\int_0^1 O(|\gamma(t)|^{3/2})\, dt.
$$
On the other hand, $\gamma(t)\in \overline
B_{\mathbb H}(0, 3 \phi^{-1}(r))$, and hence, if $r>0$ is sufficiently small,
$$
O(|\gamma(t)|^{3/2}) \le C_1 |\gamma(t)|^{3/2} \le C_2 d_{\mathbb H}(0,\gamma(t))^{3/2}
\le C(\phi^{-1}(r))^{3/2}= Cr^{3/2}(1+o(1)) ,
$$
so that
$$
2\phi^{-1}(r) =d_{\mathbb H}(z_r,\zeta_r) \le   d_c^\Psi(z_r,\zeta_r) + Cr^{3/2}(1+o(1)).
$$
Therefore
$$
d_c^\Psi(z_r,\zeta_r) \ge 2r (1+o(1)).
$$
By \eqref{2aug eq:1}, $z_r,\zeta_r\in B^\Psi_r$, so that
$$
\Psi(z_r), \Psi(\zeta_r) \in \overline B_r.
$$
Hence
\begin{equation*}\begin{split}
1  \ge \dfrac{\diam (\overline B_r) }{2r} \ge \dfrac{d_c(\Psi(z_r), \Phi(\zeta_r))}{2r}
= \dfrac{d_c^\Phi(z_r,\zeta_r)}{2r}
 \ge 1 + o(1),
\end{split}\end{equation*}
and ii) follows.
\end{proof}

Lemma \ref{august 1}  immediately yields the following equivalent definition of densities in contact manifolds:

\begin{corollary}\label{august 2} Let $M$ be $(2n+1)$-dimensional contact manifold endowed with a contact form $\theta$ and a Riemannian
metric $g$ on the fibers of $\theta$ as introduced in Propositions \ref{contact 1} and \ref{contact 3}. We denote by $d_c$
the associated Carnot-Carath\'eodory distance. Let $\mu$ be an outer measure on $M$.
Then
\begin{equation*}
\Theta^{*\,m}(\mu,x):=\limsup_{r\to 0}\frac{\mathcal \mu(\overline B(x,r))}{\alpha_m\, r^m }\,
\end{equation*}
and
\begin{equation*}
\Theta^m_*(\mu,x):=\liminf_{r\to 0}\frac{\mathcal \mu(\overline B(x,r))}{\alpha_m\, r^m}\,.
\end{equation*}

\end{corollary}

\begin{remark}\label{august 3}
In Corollary \ref{august 2} we can replace closed balls $\overline B(x,r)$ by open balls $B(x,r)$ (see \cite{Ambrosio-Tilli}, Remark 2.4.2).
\end{remark}

Keeping in mind Corollary \ref{august 2} and Remark \ref{august 3}, the following result can be proved by the same arguments
used in the proof of Theorem 3.1 in \cite{FSSC_NA}.

\begin{proposition}\label{borel continuity}
Let $M$ be $(2n+1)$-dimensional contact manifold endowed with a contact form $\theta$ and a Riemannian
metric $g$ on the fibers of $\theta$ as introduced in Propositions \ref{contact 1} and \ref{contact 3}. We denote by $d_c$
the associated Carnot-Carath\'eodory distance. Let $\mu$ be a $\sigma$-finite regular Borel measure on $M$.
Then the map
  \[
  \Theta^{*\,m}(\mu,\cdot): X\to [0,+\infty]
  \]
is Borel measurable.
\end{proposition}



%

We give now the following:
\begin{definition}\label{Hausmeasdef} Let $A\subset X$, $m \in [0,\infty)$, $\delta\in (0,\infty)$, and let $\gb_m$ be the constant  \eqref{beta1}.

{\bf (i)} The  \emph{$m$-dimensional Hausdorff measure} $\mathcal H^m$ is defined as
\[\mathcal H^m(A):=\lim_{\delta\to 0}\mathcal H_{\delta}^m(A)\] where
\begin{equation*}
\mathcal H_{\delta}^m(A)=\inf \left\{\sum_i \gb_m \diam(E_i)^m:\;A\subset \bigcup_i E_i,\quad   \diam(E_i)\leq \delta \right\}.
\end{equation*}

{\bf(ii)} The \emph{$m$-dimensional spherical Hausdorff measure} $\mathcal S^m$ is defined as
 \[\mathcal S^m(A):=\lim_{\delta\to 0}\mathcal S_{\delta}^m(A)\]where
\begin{equation*}
 \begin{aligned}
\mathcal S_{\delta}^m(A)=\inf \Big\{\sum_i \gb_m &\diam(B(x_i,r_i))^m:\,  A\subset \bigcup_i B(x_i,r_i),\\
&\diam(B(x_i,r_i))\leq \delta \Big\}
\end{aligned}
\end{equation*}

{\bf(iii)}    The \emph{$m$-dimensional centered Hausdorff measure} $\mathcal C^m$ is defined as
\begin{equation*}
\mathcal C^m(A):=\,\sup_{E\subseteq A}\mathcal C_0^m(E)\,.
\end{equation*}
where $\mathcal C_0^m(E):=\lim_{\delta\to 0^+}\mathcal C_{\delta}^m(E)$, and, in turn, $\mathcal C_{\delta}^m(E)=\,0\text{ if } E=\,\emptyset$ and for
$E\neq \emptyset$,
\begin{equation*}
 \begin{aligned}
 \mathcal C_{\delta}^m(E)=\inf \Big\{\sum_i \gb_m &\diam(B(x_i,r_i))^m:\,  E\subset \bigcup_i B(x_i,r_i),\\
&\, x_i\in E,\quad \diam(B(x_i,r_i))\leq \delta \Big\}.
\end{aligned}
\end{equation*}

\end{definition}
Notice that  the set function $\mathcal C_0^m$ is not necessarily monotone (see \cite[Sect. 4]{SRT}) while $\mathcal C^m$ is monotone.

 For reader's convenience we collect a few results about the measures  $\mathcal C^m$. Most of these results are taken from \cite {E}
 and \cite{FSSC_NA}.\\
 Let
$$
\mrm{dist}(E,F):=\,\inf  \left\{d(x,y):\,x\in E,\,y\in F\right\}
$$
denote the \emph{distance} between $E$ and $F$.
Recall that an outer measure $\mu$ on $X$ is said to be \emph{metric} if
\begin{equation*}
\mu(A\cup B)=\,\mu(A)+\,\mu(B)\qquad\text{ whenever }\mathrm {dist}(A,B)>\,0\,.
\end{equation*}
Being obtained  by Carath\"eodory's construction,  $\mathcal H^m$ and $\mathcal S^m$ are metric (outer) measures (see \cite[2.10.1]{federer}
 or \cite[Theorem 4.2]{M}). Also the measures  $\mathcal C^m$ are metric measures in any metric space, but this fact is not as immediate as for $\mathcal H^m$ and $\mathcal S^m$.
\begin{lemma}[\cite{E}, Proposition 4.1]\label{Cmom} $\mathcal C^m$ is a Borel regular outer measure.
\end{lemma}

\begin{remark}\label{trivial comparison} The measures $\mathcal H^m$, $\mathcal S^m$ and $\mathcal C^m$ are all equivalent measures. Indeed, it is well known that
 (see, for instance, \cite[2.10.2]{federer})
$$
\mathcal H^m\le\,\mathcal S^m\le\,2^m\,\mathcal H^m\,
$$
and, by definition,
\begin{equation*}\label{trivialcomparhausmeas}
\mathcal H^m\le\,\mathcal S^m\le\,\mathcal C^m\,.
\end{equation*}
The opposite inequality between $\mathcal H^m$ (or  $\mathcal S^m$) and  $\mathcal C^m$ is less immediate: it was proved in  \cite[Lemma 3.3]{SRT} for the case $X=\R^n$. See also \cite{Sch},  but for a differently defined centered Hausdorff-type measure.  The comparison in a general metric space is contained in \cite{E}.
\begin{lemma}[\cite{E}, Proposition 4.2]\label{equivHC}
$
\mathcal H^m\le\,\mathcal C^m\le\,2^m\,\mathcal H^m
\,.
$
\end{lemma}
By Lemma \ref{equivHC}, it follows in particular that the metric dimensions induced by $\mathcal H ^m$ or $\mathcal S^m$ or $\mathcal C^m$ are the same.
\end{remark}

The  estimates needed to relate the  $m$-dimensional density  $\Theta^{*\,m}(\mu,\cdot)$ with the centered Hausdorff measure $\mathcal C^m$ are the following ones.

\begin{theorem}[\cite{E}, Theorem 4.15] \label{estimedgar} Let $(X,d)$ be a separable metric space, let $\mu$ be a finite Borel outer measure in $X$ and let $B\subset X$ be a Borel set.  Then
\begin{itemize}
\item[(i)]
$$
\mu(B)\le\,\sup_{x\in B}\Theta^{*\,m}(\mu,x)\,\mathcal C^m(B),
$$
except when the product is $\infty\cdot 0$;
\item[(ii)]
$$
\inf_{x\in B}\Theta^{*\,m}(\mu,x)\,\mathcal C^m(B)\le\,\mu(B)\,.
$$
\end{itemize}
\end{theorem}
By easy modifications of the proof of Theorem \ref{estimedgar}, one gets the following density estimates involving $\Theta^{*\,m}(\mu,x)$ and $\mathcal C^m$. These estimates are analogous to Federer's ones involving $\Theta_F^{*\,m}(\mu,x)$ and $\mathcal S^m$ (see \cite{federer}).

\begin{theorem}\label{main} Let $(X,d)$ be a separable metric space, let $\mu$ be an outer measure in $X$ and $t>\,0$.
\begin{itemize}
\item[(i)] If $\mu$ is Borel regular and
\[
\Theta^{*\,m}(\mu\res A,x)<\,t,\qquad\forall x\in A\subset X\,
\]
then
$$
\mu(A)\le\, t\;\mathcal C^m(A)\,.
$$
\item[(ii)] If $V\subset X$ is  an open set and
\begin{equation*}
\Theta^{*\,m}(\mu,x)>\,t,\qquad\forall x\in B\subset V
\end{equation*}
then
$$
\mu(V)\ge\, t\;\mathcal C^m(B)\,.
$$
\end{itemize}
\end{theorem}

\begin{remark} If $\mu$ is supposed to be a Radon measure, approximating from above by open sets, we can strengthen the conclusion  in Theorem \ref{main} (ii)  getting the inequality $\mu(B)\ge\,t\;\mathcal C^m(B)$.
\end{remark}

  Using Lemma \ref{august 1} (i.e. relying on the equivalence of the two notions of density)
  and Proposition \ref{borel continuity},  the following result can be proved following  step by step the proof of
  Theorem 3.1 in \cite{FSSC_NA}.

  \begin{theorem}\label{areacentredhaus} Let $M$ be $(2n+1)$-dimensional contact manifold endowed with a contact form $\theta$ and a Riemannian
metric $g$ on the fibers of $\theta$ as introduced in Propositions \ref{contact 1} and \ref{contact 3}. We denote by $d_c$
the associated Carnot-Carath\'eodory distance. Let $\mu$ be a $\sigma$-finite regular Borel measure on $M$, and let $A\subset X$ be a Borel set.
If $\mathcal C^m(A)<\,\infty$ and $\mu\res A$ is absolutely continuous with respect to $\mathcal C^m\res A$, then
 for each Borel set $B\subset A$,
\begin{equation*}\label{areaformcentred}
\mu(B)=\,\int_B \Theta^{*\,m}(\mu,x)\,d\mathcal C^m(x).
\end{equation*}
\end{theorem}
\begin{remark} Since $\mathcal C^m$ and $\mathcal S^m$ are equivalent, then  $\mathcal C^m(A)<\,\infty$  if and only if  $\mathcal S^m(A)<\,\infty$ and $\mu\res A$ is absolutely continuous with respect to $\mathcal C^m$ if and only if $\mu\res A$ is absolutely continuous with respect to $\mathcal S^m$.
\end{remark}


Now we can give the proof of  Theorem \ref{august 5}.

\begin{proof}[Proof of Theorem \ref{august 5}] Since $|{\bf W}^{0}\chi_E|$ is supported on $\partial^*E$, without loss of generality
we may assume that \eqref{density 1} holds for all $x\in  \partial E$.

Suppose first \begin{equation}\label{abs cont}
\mu\res \partial E\ll\mc H^{2n+1}\res \partial E,
\end{equation}
 and denote by $A\subset\partial E$
the set of points where \eqref{density 1} holds, so that $\mc H^{2n+1}(\partial E\setminus A)=0$. We remind also
that $ |{\bf W}^{0}\chi_E|\ll\mc H^{2n+1}\res \partial E$, by \cite{ambrosio_advances}, Lemma 5.2.  Thus,
if $B\subset \partial E$ is a Borel set, we can apply
Theorem \ref{areacentredhaus} to get
\begin{equation*}\begin{split}
\mu\res \partial E (B) &= \mu (\partial E\cap B) = \int_{\partial E\cap B}\Theta^{* ,2n+1} (\mu,x) d\mc C^{2n+1}(x)
\\&
\ge \int_{\partial E\cap B}\Theta^{* ,2n+1} (|{\bf W}^{0}\chi_E|,x) d\mc C^{2n+1}(x) = |{\bf W }^{0}\chi_E|(\partial E\cap B)
\\&
\vphantom{\int_{A\cap B}}
= |{\bf W}^{0}\chi_E| (B).
\end{split}\end{equation*}
Let us drop now the assumption \eqref{abs cont}. We can write
\begin{equation*}\label{lebesgue}
\mu\res\partial E = \mu_{ac} + \mu_s
 \end{equation*} with
 $$
  \mbox{$\;\mu_{ac} \ll\mc H^{2n+1}\res \partial E$
\; and \; $ \mu_s \perp \mc H^{2n+1}\res \partial E$}
$$
(see \cite{rudin} Theorem 6.10), i.e. there exists $K \subset M$ such that
$$
 \mu_s =  \mu_s\res K\quad\mbox{and}\quad (\mc H^{2n+1}\res \partial E)(K)=0.
$$
Set now
$$
S_0:= \{x\in M\; ; \; \Theta^{*,2n+1}( \mu_s , x) = 0\}.
$$
Notice that $S_0$ is a Borel set, since $\Theta^{* 2n+1} (\mu_s,\cdot)$ is a Borel function.

If $x\in S_0$, then
\begin{equation*}\begin{split}
\Theta^{* ,2n+1} & (|{\bf W}^{0}\chi_E|,x)\le
\Theta^{* ,2n+1} (\mu,x)
\\&
\le
\Theta^{* ,2n+1} (\mu_s,x) + \Theta^{* ,2n+1} ( \mu_{ac},x)
\\&
= \Theta^{* ,2n+1} ( \mu_{ac},x).
\end{split}\end{equation*}
Thus, as above, we can  apply
Theorem \ref{areacentredhaus} to get for any Borel set $B$
$$
|{\bf W}^{0}\chi_E|(B\cap S_0)\le \mu_{ac} (B\cap S_0)\le  \mu (B\cap S_0)\le \mu (B).
$$
To complete the proof of \eqref{density 2}, we shall prove that
\begin{equation}\label{density 4}
(\mc H^{2n+1}\res \partial E)(S_0^c) =0,
\end{equation}
that yields
\begin{equation*}\label{density 3}
|{\bf W}^{0}\chi_E|(S_0^c) = 0,
\end{equation*}
by \cite{ambrosio_advances}, Lemma 5.2 (here $S_0^c$ denotes the complement of $S_0$).

In order to prove \eqref{density 4}, we can write
$$
S_0^c = \cup_{n=1}^\infty \{x\in M\; ; \; \Theta^{* ,2n+1} (\mu_s,x) > \tfrac 1 n\}:=  \cup_{n=1}^\infty T_n.
$$
Then
\begin{equation}\label{density 5}\begin{split}
(\mc H^{2n+1}\res \partial E)(S_0^c) & = (\mc H^{2n+1}\res \partial E)(S_0^c\cap K) + (\mc H^{2n+1}\res \partial E)(S_0^c\cap K^c)
\\&
= (\mc H^{2n+1}\res \partial E)(S_0^c\cap K^c),
\end{split}\end{equation}
since
$$
(\mc H^{2n+1}\res \partial E)(S_0^c\cap K) \le (\mc H^{2n+1}\res \partial E)( K) = 0.
$$

On the other hand
\begin{equation}\label{density 6}\begin{split}
(\mc H^{2n+1}\res \partial E)(S_0^c\cap K^c)
= \lim_{n\to\infty} (\mc H^{2n+1}\res \partial E)(S_0^c\cap K^c\cap T_n).
\end{split}\end{equation}
The set $\partial E\cap S_0^c\cap K^c\cap T_n$ is a Borel set, so that,
by Federer's differentiation theorem (see, e.g., \cite{Ambrosio-Tilli} Theorem 2.4.3)
\begin{equation}\label{density 7}\begin{split}
(\mc H^{2n+1}\res \partial E) & (S_0^c\cap K^c\cap T_n) \le
n\,\mu_s(S_0^c\cap K^c\cap T_n)
\\&
= n\,(\mu_s\res K)(S_0^c\cap K^c\cap T_n) =0.
\end{split}\end{equation}
Combining \eqref{density 5}, \eqref{density 6} and \eqref{density 7} we obtain eventually \eqref{density 4}. This
completes the proof of the theorem.

\end{proof}


\end{document}